\documentclass[a4,11pt]{amsart}
\usepackage[top=3cm, bottom=3cm, left=2.5cm, right=2.5cm]{geometry}

\usepackage{amssymb,amsmath,amsthm,txfonts}
\usepackage{cite}
\usepackage{nicefrac}
\usepackage[pdftex]{graphicx}
\usepackage{lscape}
\usepackage[hang,small,bf]{caption}
\usepackage[subrefformat=parens]{subcaption}
\newtheorem{thm}{Theorem}[section]
\newtheorem{lem}[thm]{Lemma}
\newtheorem{cor}[thm]{Corollary}
\newtheorem{prop}[thm]{Proposition}
\theoremstyle{definition}
\newtheorem{defn}[thm]{Definition}
\theoremstyle{remark}
\newtheorem{rem}[thm]{\textbf{Remark}}
\newtheorem{rems}[thm]{\textbf{Remarks}}

 \makeatletter
    
    \@addtoreset{equation}{section}
  \makeatother

\makeatletter
      \def\@makefnmark{%
         \leavevmode
            \raise.9ex\hbox{\check@mathfonts
                \fontsize\sf@size\z@\normalfont%
                            \@thefnmark}%
       }
      \makeatother

\makeatletter
\@namedef{subjclassname@2020}{%
  \textup{2020} Mathematics Subject Classification}
\makeatother

\newcommand{\dd}{\textrm{d}}

\begin{document}

\title[]{Existence of homogeneous Euler flows of degree $-\alpha\notin [-2,0]$}
\author[]{Ken Abe}
\date{}
\address[K. ABE]{Department of Mathematics, Graduate School of Science, Osaka Metropolitan University, 3-3-138 Sugimoto, Sumiyoshi-ku Osaka, 558-8585, Japan}
\email{kabe@omu.ac.jp}

\date{\today}

\begin{abstract}
We consider ($-\alpha$)-homogeneous solutions to the stationary incompressible Euler equations in $\mathbb{R}^{3}\backslash\{0\}$ for $\alpha\geq 0$ and in $\mathbb{R}^{3}$ for $\alpha<0$. Shvydkoy (2018) demonstrated the \textit{nonexistence} of ($-1$)-homogeneous solutions $(u,p)\in C^{1}(\mathbb{R}^{3}\backslash \{0\})$ and ($-\alpha$)-homogeneous solutions in the range $0\leq \alpha\leq 2$ for the Beltrami and axisymmetric flows. Namely, no ($-\alpha$)-homogeneous solutions $(u,p)\in C^{1}(\mathbb{R}^{3}\backslash \{0\})$ for $1\leq \alpha\leq 2$ and $(u,p)\in C^{2}(\mathbb{R}^{3}\backslash \{0\})$ for $0\leq \alpha< 1$ exist among these particular classes of flows other than irrotational solutions for integers $\alpha$. The nonexistence result of the Beltrami ($-\alpha$)-homogeneous solutions $(u,p)\in C^{2}(\mathbb{R}^{3}\backslash \{0\})$ holds for all $\alpha<1$. We show the nonexistence of axisymmetric ($-\alpha$)-homogeneous solutions without swirls $(u,p)\in C^{2}(\mathbb{R}^{3}\backslash \{0\})$ for $-2\leq \alpha<0$.

The main result of this study is the \textit{existence} of axisymmetric ($-\alpha$)-homogeneous solutions in the complementary range $\alpha\in \mathbb{R}\backslash [0,2]$. More specifically, we show the existence of axisymmetric Beltrami ($-\alpha$)-homogeneous solutions $(u,p)\in C^{1}(\mathbb{R}^{3}\backslash \{0\})$ for $\alpha>2$ and $(u,p)\in C(\mathbb{R}^{3})$ for $\alpha<0$ and axisymmetric ($-\alpha$)-homogeneous solutions with a nonconstant Bernoulli function $(u,p)\in C^{1}(\mathbb{R}^{3}\backslash \{0\})$ for $\alpha>2$ and $(u,p)\in C(\mathbb{R}^{3})$ for $\alpha<-2$, including axisymmetric ($-\alpha$)-homogeneous solutions without swirls $(u,p)\in C^{2}(\mathbb{R}^{3}\backslash \{0\})$ for $\alpha>2$ and $(u,p)\in  C^{1}(\mathbb{R}^{3}\backslash \{0\})\cap C(\mathbb{R}^{3})$ for $\alpha<-2$. This is the first existence result on ($-\alpha$)-homogeneous solutions with no explicit forms.

The level sets of the axisymmetric stream function of the irrotational ($-\alpha$)-homogeneous solutions in the cross-section are the Jordan curves for $\alpha=3$. For $2<\alpha<3$, we show the existence of axisymmetric ($-\alpha$)-homogeneous solutions whose stream function level sets are the Jordan curves. They provide new examples of the Beltrami/Euler flows in $\mathbb{R}^{3}\backslash\{0\}$ whose level sets of the proportionality factor/Bernoulli surfaces are nested surfaces created by the rotation of the sign $``\infty"$.
\end{abstract}

\maketitle

\tableofcontents

\section{Introduction}
In this study, we investigate ($-\alpha$)-homogeneous solutions to the Euler equations in $\mathbb{R}^{3}\backslash \{0\}$ for $\alpha \geq 0$ and in $\mathbb{R}^{3}$ for $\alpha< 0$:

\begin{equation}
\begin{aligned}
u\cdot \nabla u+\nabla p&=0,\\
\nabla \cdot u&=0.
\end{aligned}
\end{equation}\\
We say that $u=(u^{1},u^{2},u^{3})$ is ($-\alpha$)-\textit{homogeneous} if there exists $\alpha\in \mathbb{R}$ such that $u(x)= \lambda^{\alpha} u(\lambda x)$ for all $\lambda>0$ and $x=(x_1,x_2,x_3)$. We say that $(u,p)$ is a ($-\alpha$)-homogeneous solution to (1.1) if ($-\alpha$)-homogeneous $u$ and $(-2\alpha)$-homogeneous $p$ satisfy (1.1).

The well-known $(-1)$-homogeneous solutions to the Navier--Stokes equations are the Landau solutions \cite{Landau}, \cite{Squire51}, \cite[p.81]{LL}, \cite[p.205]{Bat}, \cite{TX98}, \cite{CK04}. They are explicit solutions, smooth away from the origin, and axisymmetric without swirls. Tian and Xin \cite{TX98} showed that all axisymmetric $(-1)$-homogeneous solutions $u\in C^{2}(\mathbb{R}^{3}\backslash \{0\})$ to the Navier--Stokes equations are the Landau solutions. \v{S}ver\'{a}k \cite{Sverak2011} demonstrated that all $(-1)$-homogeneous solutions $u\in C^{2}(\mathbb{R}^{n}\backslash \{0\})$ for $n=3$ are Landau solutions, as well as the nonexistence of $(-1)$-homogeneous solutions for $n\geq 4$ and their rigidity for $n=2$ under the flux condition. The Landau solutions are relevant to the regularity of stationary solutions \cite{Sverak2011} and their asymptotic behavior as $|x|\to\infty$ \cite{Sverak2011}, \cite{KorolevSverak2011}, \cite{MiuraTsai}, \cite{KMT12}.

It is conjectured in the work of \v{S}ver\'{a}k \cite{Sverak2011} that the Landau solutions are rigid among all smooth solutions in $\mathbb{R}^{3}\backslash \{0\}$ satisfying the following: 

\begin{align*}
|u(x)|\leq \frac{C}{|x|},\quad x\in \mathbb{R}^{3}\backslash \{0\}.
\end{align*}\\
Korolev and \v{S}ver\'{a}k \cite{KorolevSverak2011} and Miura and Tsai \cite{MiuraTsai}, \cite[8.2]{Tsaibook} demonstrated that this conjecture holds for small constant $C$. Li et al. \cite{YanYanLi} discovered explicit axisymmetric $(-1)$-homogeneous solutions without swirls smooth away from the negative part of the $x_3$-axis, i.e., $u|_{\mathbb{S}^{2}}\in C^{\infty}(\mathbb{S}^{2}\backslash\{\textrm{S}\})$ for the South pole S. The work \cite{YanYanLi} also demonstrates the existence of axisymmetric $(-1)$-homogeneous solutions with swirls $u|_{\mathbb{S}^{2}}\in C^{\infty}(\mathbb{S}^{2}\backslash\{\textrm{S}\})$. The subsequent works \cite{YanYanLi18} and \cite{YanYanLi19} show the existence of axisymmetric $(-1)$-homogeneous solutions with swirls smooth away from the $x_3$-axis, i.e., $u|_{\mathbb{S}^{2}}\in C^{\infty}(\mathbb{S}^{2}\backslash\{\textrm{S}\cup \textrm{N}\})$. Kwon and Tsai \cite{KwonTsai} explored the bifurcations of the Landau solutions in the class of axisymmetric and discrete homogeneous (self-similar) solutions.\\

Luo and Shvydkoy \cite{LuoShvydkoy} and Shvydkoy \cite{Shv} investigated homogeneous solutions to the Euler equations. The work by Shvydkoy \cite{Shv} is motivated by the Onsager conjecture \cite{Onsager}, \cite{Shv10}, \cite{DS13}, \cite{Isett13}, \cite{CS14} and demonstrates the nonexistence of ($-\alpha$)-homogeneous solutions to the Euler equations in the following cases.\\

\noindent
\textbf{Case 0}: Irrotational flows $\nabla \times u=0$. $(-\alpha)$-homogeneous solutions $(u,p)\in C^{1}(\mathbb{R}^{3}\backslash \{0\})$ exist if and only if $\alpha\in \mathbb{Z}\backslash \{1\}$. They are given by spherical harmonics. \\
\noindent
\textbf{Case 1}: $\alpha=1$. No $(-1)$-homogeneous solutions $(u,p)\in C^{1}(\mathbb{R}^{3}\backslash \{0\})$ exist.\\
\noindent
\textbf{Case 2}: $\alpha>1$. For $1< \alpha\leq 2$, no ($-\alpha$)-homogeneous solutions $(u,p)\in C^{1}(\mathbb{R}^{3}\backslash \{0\})$ exist other than the irrotational solution among the following:\\

\noindent
(A) Beltrami flows $(\nabla \times u)\times u=0$ or \\
(B) Axisymmetric flows.  \\

\noindent
\textbf{Case 3}: $\alpha<1$. In classes (A) for $\alpha<1$ and (B) for $0\leq \alpha<1$, no $(-\alpha)$-homogeneous solutions $(u,p)\in C^{2}(\mathbb{R}^{3}\backslash \{0\})$ exist other than the irrotational solutions.\\ 

These rigidity results are based on the homogeneous solution's equations on the sphere \cite{Sverak2011}, \cite{Shv} and do not assume their continuity at $x=0$ for $\alpha<0$ (Section 2). We include them \cite{Shv} in the main statements of this study ((i) and (ii) of Theorems 1.1, 1.4, and 1.5, except (ii) of Theorem 1.4 for $-2\leq \alpha<0$). \\

On the existence side, only explicit homogeneous solutions to (1.1) are known \cite{LuoShvydkoy}, \cite{Shv} (Remarks 1.6). This study aims to show the existence of axisymmetric homogeneous solutions. We use the cylindrical coordinates $(r,\phi,z)$ defined by the following: 

\begin{align*}
x_1=r\cos\phi,\quad x_2=r\sin\phi,\quad x_3=z,
\end{align*}\\
and the associated orthogonal flame 

\begin{equation*}
e_{r}=\left(
\begin{array}{c}
\cos \phi \\
\sin \phi \\
0
\end{array}
\right),\quad 
e_\phi=\left(
\begin{array}{c}
-\sin\phi \\
\cos\phi \\
0
\end{array}
\right),\quad 
e_{z}=\left(
\begin{array}{c}
0 \\
0 \\
1
\end{array}
\right).
\end{equation*}\\
For axisymmetric $u=u^{r}e_r+u^{\phi}e_{\phi}+u^{z}e_z$, we denote the poloidal component by $u^{P}=u^{r}e_r+u^{z}e_z$ and the toroidal component by $u^{\phi}e_{\phi}$. We say that $u$ is axisymmetric without swirl if $u^{\phi}=0$. 

\vspace{3pt}

\subsection{Statements of the main results}

We consider continuously differentiable ($-\alpha$)-homogeneous solutions $(u.p)\in C^{1}(\mathbb{R}^{3}\backslash \{0\})$ for $\alpha>2$ in $\mathbb{R}^{3}\backslash\{0\}$ and continuous ($-\alpha$)-homogeneous solutions $(u.p)\in C(\mathbb{R}^{3})$ for $\alpha<0$ in $\mathbb{R}^{3}$ satisfying (1.1) in the distributional sense. We say that a ($-\alpha$)-homogeneous solution $(u.p)\in C(\mathbb{R}^{3})$ for $\alpha<0$ is a Beltrami flow in $\mathbb{R}^{3}$ if the Bernoulli function $\Pi=p+|u|^{2}/2$ vanishes.

\vspace{3pt}

\begin{thm}
The following holds for rotational Beltrami $(-\alpha)$-homogeneous solutions to (1.1):

\noindent
(i) For $1\leq \alpha\leq 2$, no solutions $(u,p)\in C^{1}(\mathbb{R}^{3}\backslash \{0\})$ exist.\\
\noindent
(ii) For $\alpha< 1$, no solutions $(u,p)\in C^{2}(\mathbb{R}^{3}\backslash \{0\})$ exist.\\
\noindent
(iii) For $\alpha>2$, axisymmetric solutions $(u,p)\in C^{1}(\mathbb{R}^{3}\backslash \{0\})$ such that $u^{P}, p\in C^{2}(\mathbb{R}^{3}\backslash \{0\})$ and $u^{\phi}e_{\phi}\in C^{1}(\mathbb{R}^{3}\backslash \{0\})$ exist. \\
\noindent
(iv) For $\alpha<0$, axisymmetric solutions $(u,p)\in C(\mathbb{R}^{3})$ such that $u^{P}, p\in  C^{1}(\mathbb{R}^{3}\backslash \{r=0\})\cap C(\mathbb{R}^{3})$ and $u^{\phi}e_{\phi}\in C(\mathbb{R}^{3})$ exist.
\end{thm}

  \begin{figure}[h]
  \begin{minipage}[b]{0.45\linewidth}
\hspace{-108pt}
\includegraphics[scale=0.180]{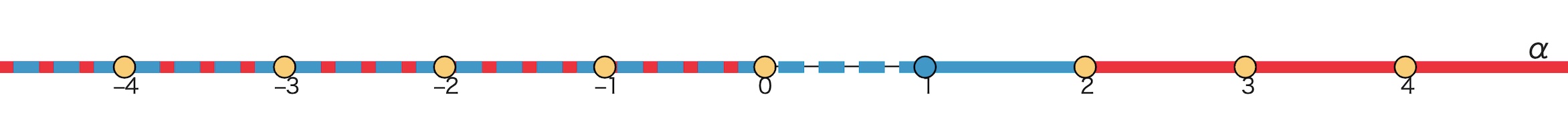}
    \subcaption{Beltrami ($-\alpha$)-homogeneous solutions in Theorem 1.1.}
  \end{minipage}\\
  \begin{minipage}[b]{0.45\linewidth}
\hspace{-108pt}
\includegraphics[scale=0.165]{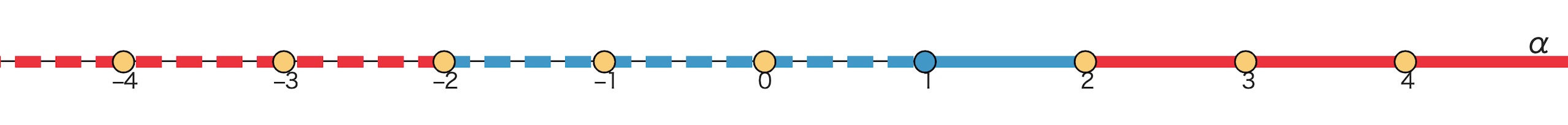}
    \subcaption{Axisymmetric ($-\alpha$)-homogeneous solutions without swirls in Theorem 1.4.}
  \end{minipage}\\
    \begin{minipage}[b]{0.45\linewidth}
\hspace{-108pt}
\includegraphics[scale=0.165]{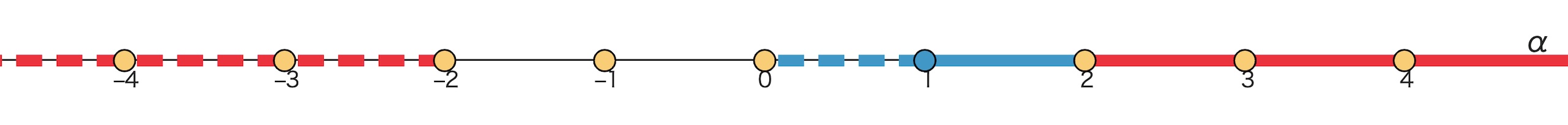}
    \subcaption{Axisymmetric ($-\alpha$)-homogeneous solutions with a non-constant Bernoulli function in Theorem 1.5.}
  \end{minipage}\\
      \begin{minipage}[b]{0.45\linewidth}
\hspace{-108pt}
\includegraphics[scale=0.165]{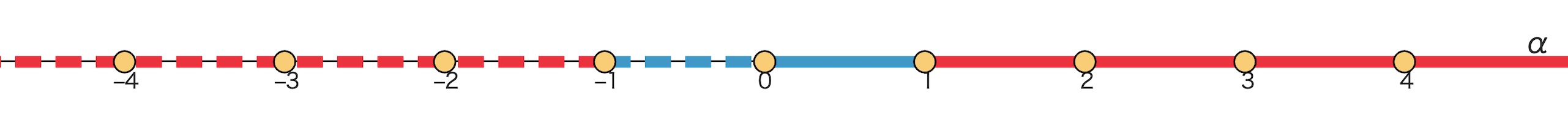}
    \subcaption{2D reflection symmetric ($-\alpha$)-homogeneous solutions in Theorem 1.7.}
  \end{minipage}
  \caption{The existence and nonexistence ranges of ($-\alpha$)-homogeneous solutions to (1.1) in Theorems 1.1, 1.4, 1.5, and 1.7. The yellow dots represent irrotational solutions. The blue dot and line represent the nonexistence range (i). The blue dashed line represents the nonexistence range (ii). The red line represents the existence range (iii). The dashed red line represents the existence range (iv).}
\end{figure}

\begin{rems}
(i) A vector field $u\in C^{1}$ is a Beltrami flow if there exists a proportionality factor $\varphi$ such that 

\begin{align}
\nabla \times u=\varphi u,\quad 
\nabla \cdot u=0.
\end{align}\\
The factor $\varphi$ is a first integral of $u$, i.e., $u\cdot \nabla \varphi=0$, and streamlines of $u$ are constrained on the level sets of $\varphi$. It is known \cite{EP12}, \cite{EP15} that there exists a smooth Beltrami flow with a constant factor $\varphi\equiv \textrm{const.}$ in $\mathbb{R}^{3}$ having arbitrary knotted and linked vortexlines and decaying by the order $u=O(|x|^{-1})$ as $|x|\to\infty$.

\noindent
(ii) It is known \cite{Na14}, \cite{CC15}, \cite{CW16} that the locally square integrable Beltrami flows in $\mathbb{R}^{3}$ do not exist if $u=o(|x|^{-1})$ as $|x|\to\infty$, cf. \cite{EP12}, \cite{EP15}. 

\noindent
(iii) It is also known \cite{EP16} that the smooth Beltrami flows in a domain do not exist if the proportionality factor $\varphi\in C^{2+\gamma}$ ($0<\gamma<1$) admits a level set diffeomorphic to a sphere, e.g., radial or having local extrema. 

\noindent
(iv) Constantin et al. \cite{CDG22} demonstrate that the axisymmetric Beltrami flows in a hollowed out periodic cylinder are translation invariant if the poloidal component $u^{P}$ has no stagnation points in the cross-section, cf. \cite{HN17}, \cite{HN19}. 

\noindent
(v) There exist asymptotically constant axisymmetric Beltrami flows in $\mathbb{R}^{3}$ with a nonconstant factor $\varphi\nequiv \textrm{const.}$ whose level set is a ball \cite{Moffatt}, a solid torus \cite{Tu89}, and nested tori \cite{A8}.

\noindent
(vi) Axisymmetric Beltrami $(-\alpha)$-homogeneous solutions $u\in C^{1}(\mathbb{R}^{3}\backslash \{0\})$ for $\alpha>2$ in Theorem 1.1 (iii) possess the axisymmetric stream function $\psi(z,r)$ and the proportionality factor 

\begin{align}
\varphi=C\left(1+\frac{1}{\alpha-2}\right)|\psi|^{\frac{1}{\alpha-2}}.
\end{align}\\
This solution is not square integrable at $x=0$ and decaying faster than $o(|x|^{-1})$ as $|x|\to\infty$, cf.\cite{Na14}, \cite{CC15}, \cite{CW16}. The level sets of the proportionality factor $\varphi$ are nested surfaces created by the rotation of multifoils, cf. \cite{EP16} (Remark 1.12). The solution $u\in C(\mathbb{R}^{3})$ for $\alpha<0$ in Theorem 1.1 (iv) is growing as $|x|\to\infty$.  
\end{rems}

\vspace{3pt}

A simple class of rotational flows with a nonconstant Bernoulli function is as follows:\\

\noindent
(C) Radially irrotational flows $\nabla \times u\cdot x=0$.\\

We remark that the tangentially irrotational homogeneous flows $(\nabla\times u)\times x=0$ are irrotational (Remark 2.14). The radially irrotational flows include axisymmetric flows without swirls. On the contrary, we demonstrate the following:\\

\begin{thm}
All radially irrotational $(-\alpha)$-homogeneous solutions $(u,p)\in C^{2}(\mathbb{R}^{3}\backslash \{0\})$ for $\alpha\in \mathbb{R}$ with a nonconstant Bernoulli function to (1.1) are axisymmetric without swirls.
\end{thm}

\vspace{3pt}

The existence and nonexistence ranges of axisymmetric $(-\alpha)$-homogeneous solutions without swirls are split into $\alpha\in \mathbb{R}\backslash [-2,2]$ and $\alpha\in [-2,2]$.

\vspace{3pt}

\begin{thm}
The following holds for rotational axisymmetric $(-\alpha)$-homogeneous solutions without swirls to (1.1):

\noindent
(i) For $1\leq \alpha\leq 2$, no solutions $(u,p)\in C^{1}(\mathbb{R}^{3}\backslash \{0\})$ exist.\\
\noindent
(ii) For $-2< \alpha< 1$, no solutions $(u,p)\in C^{2}(\mathbb{R}^{3}\backslash \{0\})$ exist. For $\alpha=-2$, no solutions $(u,p)\in C^{2}(\mathbb{R}^{3}\backslash \{0\})$ exist provided that $\nabla \times u\cdot e_{\phi}/r$ vanishes on the $z$-axis.\\
\noindent
(iii) For $\alpha>2$, solutions $(u,p)\in C^{2}(\mathbb{R}^{3}\backslash \{0\})$ exist. \\
\noindent
(iv) For $\alpha<-2$, solutions $(u,p)\in  C^{1}(\mathbb{R}^{3}\backslash \{0\})\cap C(\mathbb{R}^{3})$ exist.
\end{thm}

\vspace{3pt}

We state a general existence result on axisymmetric ($-\alpha$)-homogeneous solutions with a nonconstant Bernoulli function. 
 
\vspace{3pt}

\begin{thm}
The following holds for rotational axisymmetric $(-\alpha)$-homogeneous solutions to (1.1):\\
\noindent
(i) For $1\leq \alpha\leq 2$, no solutions $(u,p)\in C^{1}(\mathbb{R}^{3}\backslash \{0\})$ exist.

\noindent
(ii) For $0\leq \alpha<1$, no solutions $(u,p)\in C^{2}(\mathbb{R}^{3}\backslash \{0\})$ exist.

\noindent
(iii) For $\alpha>2$, solutions $(u,p)\in C^{1}(\mathbb{R}^{3}\backslash \{0\})$ such that $u^{P},p\in C^{2}(\mathbb{R}^{3}\backslash \{0\})$ and $ u^{\phi}e_{\phi}\in C^{1}(\mathbb{R}^{3}\backslash \{0\})$ exist. 

\noindent
(iv) For $\alpha<-2$, solutions $(u,p)\in C(\mathbb{R}^{3})$ such that $u^{P},p\in  C^{1}(\mathbb{R}^{3}\backslash \{r=0\})\cap C(\mathbb{R}^{3})$ and $u^{\phi}e_{\phi}\in C(\mathbb{R}^{3})$ exist.   
\end{thm}

\vspace{3pt}

\begin{rems}
(i) The explicit rotational axisymmetric ($-\alpha$)-homogeneous solution with a nonconstant Bernoulli function $(u,p)$ exists for $\alpha\leq 0$ \cite[p.2521, (13)]{Shv}: 

\begin{align*}
u(x)&=b^{2}\frac{x_3}{x_1^{2}+x_2^{2}}K^{-\frac{\alpha}{2}}
\left(
\begin{array}{c}
x_1 \\
x_2\\
0
\end{array}
\right)
-\frac{b}{x_1^{2}+x_2^{2}}K^{\frac{1-\alpha}{2}}
\left(
\begin{array}{c}
-x_2 \\
x_1\\
0
\end{array}
\right)
+a^{2}K^{-\frac{\alpha}{2}}
\left(
\begin{array}{c}
0 \\
0\\
1
\end{array}
\right),\\
K(x)&=\left(a^{2}(x_1^{2}+x_2^{2})-b^{2}x_3^{2} \right)_{+},\quad a^{2}+b^{2}=1,\quad 
\alpha p=0,\quad \alpha\leq 0.
\end{align*}\\
Here, $s_+=\max\{s,0\}$ for $s\in \mathbb{R}$. For $b=0$, $u=r^{-\alpha} e_z$ is without swirl and belongs to $C^{1}(\mathbb{R}^{2})$ for $\alpha<-1$ and $C^{2}(\mathbb{R}^{2})$ for $\alpha\leq  -2$. For $\alpha=-2$, the toroidal component of vorticity is $\nabla \times u\cdot e_{\phi}/r=-2$, cf. Theorem 1.4 (ii). For $b\neq 0$, $u$ is with swirls, belongs to $C^{1}(\mathbb{R}^{3})$ for $\alpha<-2$, and is supported in the wedged region $\{a^{2}r^{2}-b^{2}z^{2}>0\}$. In particular, $u|_{\mathbb{S}^{2}}$ is compactly supported in $\{\theta_0<\theta<\pi-\theta,\ 0\leq \phi\leq 2\pi \}$ on $\mathbb{S}^{2}$ for $\theta_0=\arctan{|b/a|}$, cf. Theorem 1.5 (iv). Here, $\theta$ is the geodesic radial coordinate on $\mathbb{S}^{2}$ (Section 2). This solution is as follows:\\

\noindent
(D) Geodesic flows $(u\cdot \nabla u)\times u=0$.\\

Namely, streamlines are rays. The solutions of (1.1) with a constant pressure are geodesic flows. It is demonstrated in the work of Shvydkoy \cite[Proposition 5.3]{Shv} that all axisymmetric ($-\alpha$)-homogeneous solutions $u\in C^{1}(\mathbb{R}^{3}\backslash \{0\})$ with a constant pressure $p$ are this solution or the irrotational solution for $\alpha=2$. We remark on the existence of compactly supported inhomogeneous axisymmetric solutions with swirls in $\mathbb{R}^{3}$ \cite{Gav}, \cite{CLV}, \cite{DEPS21} and compactly supported vortex patch solutions in $\mathbb{R}^{2}$ \cite{GPS22}. Baldi \cite{Baldi} discusses the streamline geometry of compactly supported inhomogeneous axisymmetric solutions with swirls.

\noindent
(ii) The two-dimensional (2D) $(-\alpha)$-homogeneous solutions $u=(u^{1},u^{2},0)$ can exist for all $\alpha\in \mathbb{R}$. Luo and Shvydkoy \cite{LuoShvydkoy}, \cite[2.2]{Shv} found several explicit solutions and investigated the streamlines of $(-\alpha)$-homogeneous solutions based on the stream function's  Hamiltonian PDE. The pressure $p$ of 2D $(-\alpha)$-homogeneous solutions is constant on the circle $r=1$. In fact, for any $\alpha\in \mathbb{R}$,

\begin{align*}
u(x)=\frac{a}{r^{\alpha+1}}\left(
\begin{array}{c}
-x_2 \\
x_1\\
0
\end{array}
\right),\quad p(x)=-\left(\frac{a^{2}}{2\alpha}\right)\frac{1}{r^{2\alpha}},\quad a\in \mathbb{R},
\end{align*}\\
is a radially symmetric $(-\alpha)$-homogeneous solution (a circular flow) to (1.1) in $\mathbb{R}^{3}\backslash \{r=0\}$ for $\alpha\geq 0$ and in $\mathbb{R}^{3}$ for $\alpha< 0$, cf. Theorem 1.5 (i) and (ii). All 2D radially symmetric $(-\alpha)$-homogeneous solution for $\alpha\in \mathbb{R}\backslash \{1\}$ is this solution (Theorem A.1). The work by Shvydkoy \cite[Proposition 4.1]{Shv} demonstrates that all ($-\alpha$)-homogeneous solutions $(u,p)\in C^{1}(\mathbb{R}^{3}\backslash \{0\})$ to (1.1) are 2D radially symmetric solutions for $\alpha\leq -1$, provided that\\ 

\noindent
(E) Tangential flows $u\cdot x=0$. \\

Guo et al. \cite{GHPW}, \cite{GPW} demonstrated that the radially symmetric $1$-homogeneous solution is stable in the axisymmetric Euler equations via the Euler--Coriolis equations. 
\end{rems}

\newpage
  
Noncircular streamlines appear for the following:\\ 

\noindent
(F) 2D reflection symmetric flows $u=(u^{1},u^{2},0),$

\begin{equation*}
\begin{aligned}
u^{1}(x_1,x_2)&=u^{1}(x_1,-x_2), \\
u^{2}(x_1,x_2)&=-u^{2}(x_1,-x_2),\\
p(x_1,x_2)&=p(x_1,-x_2).
\end{aligned}
\end{equation*}\\
Irrotational 2D reflection symmetric ($-\alpha$)-homogeneous solutions $u=(u^{1},u^{2},0)\in C^{1}(\mathbb{R}^{2}\backslash \{0\})$ to (1.1) exist if and only if $\alpha \in \mathbb{Z}$ (Theorem A.2). They are constant multiples of the following:

\begin{equation}
\begin{aligned}
u&=
\frac{1}{r^{2}}
\left(
\begin{array}{c}
x_1 \\
x_2\\
0
\end{array}
\right),\quad \alpha=1,\\
\quad
u&=\left(
\begin{array}{c}
\partial_2\psi \\
-\partial_1\psi\\
0
\end{array}
\right),\quad \alpha\in \mathbb{Z}\backslash \{1\},
\end{aligned}
\end{equation}
for the stream function 

\begin{align}
\psi(x_1,x_2)=\frac{\sin(n\phi)}{r^{n}},\quad n=\alpha-1.
\end{align}\\
Figure 2 shows the level sets of $\psi$ for $n=\pm1$, $\pm2$, and $\pm3$.\\

We consider 2D reflection symmetric ($-\alpha$)-homogeneous solutions $u\in C^{1}(\mathbb{R}^{2}\backslash \{0\})$ for $\alpha\geq 0$ in $\mathbb{R}^{2}\backslash \{0\}$ and $u\in C(\mathbb{R}^{2})$ for $\alpha< 0$ in $\mathbb{R}^{2}$ satisfying (1.1) in the distributional sense. The existence and nonexistence ranges of 2D reflection symmetric ($-\alpha$)-homogeneous solutions are split into $\alpha\in \mathbb{R}\backslash[-1,1]$ and $\alpha\in [-1,1]$, cf. Theorem 1.4.

\vspace{3pt}

\begin{thm}
The following holds for rotational 2D reflection symmetric $(-\alpha)$-homogeneous solutions $u=(u^{1},u^{2},0)$ and $r^{2\alpha} p=\textrm{const.}$ to (1.1):

\noindent
(i) For $0\leq \alpha\leq 1$, no solutions $u\in C^{1}(\mathbb{R}^{2}\backslash \{0\})$ exist.\\
\noindent
(ii) For $-1\leq \alpha< 0$, no solutions $u\in C^{2}(\mathbb{R}^{2}\backslash \{0\})$ exist.\\
\noindent
(iii) For $\alpha>1$, solutions $u\in C^{2}(\mathbb{R}^{2}\backslash \{0\})$ exist. \\
\noindent
(iv) For $\alpha<-1$, solutions $u\in C^{1}(\mathbb{R}^{2}\backslash \{0\})\cap C(\mathbb{R}^{2})$ exist.
\end{thm}

\vspace{3pt}

The stream function level sets $\{\psi=\pm C\}$ for $C>0$ of the rotational 2D reflection symmetric $(-\alpha)$-homogeneous solutions for $\alpha>1$ in Theorem 1.7 (iii) are unions of the Jordan curves sharing the origin (multifoils). For $\alpha=2$, the stream function level sets of the irrotational 2D reflection symmetric $(-2)$-homogeneous solution consist of the Jordan curves $\{\psi=C\}$ in the upper half plane and the Jordan curves $\{\psi=-C\}$ in the lower half plane ($n=1$ in Figure 2). We show the existence of rotational $(-\alpha)$-homogeneous solutions for $1<\alpha<2$ whose stream function level sets are homeomorphic to those of the irrotational 2D reflection symmetric $(-2)$-homogeneous solution.

\vspace{3pt}

\begin{thm}
For $1<\alpha<2$, there exist rotational 2D reflection symmetric $(-\alpha)$-homogeneous solutions $u\in C^{2}(\mathbb{R}^{2}\backslash \{0\})$ to (1.1) whose stream function level sets are homeomorphic to those of the irrotational $(-2)$-homogeneous solution.
\end{thm}

\vspace{3pt}


\begin{figure}[p]
    \begin{tabular}{ccc}
    \hspace{-15pt}
      \begin{minipage}[t]{0.45\hsize}
        \centering
\includegraphics[scale=0.65]{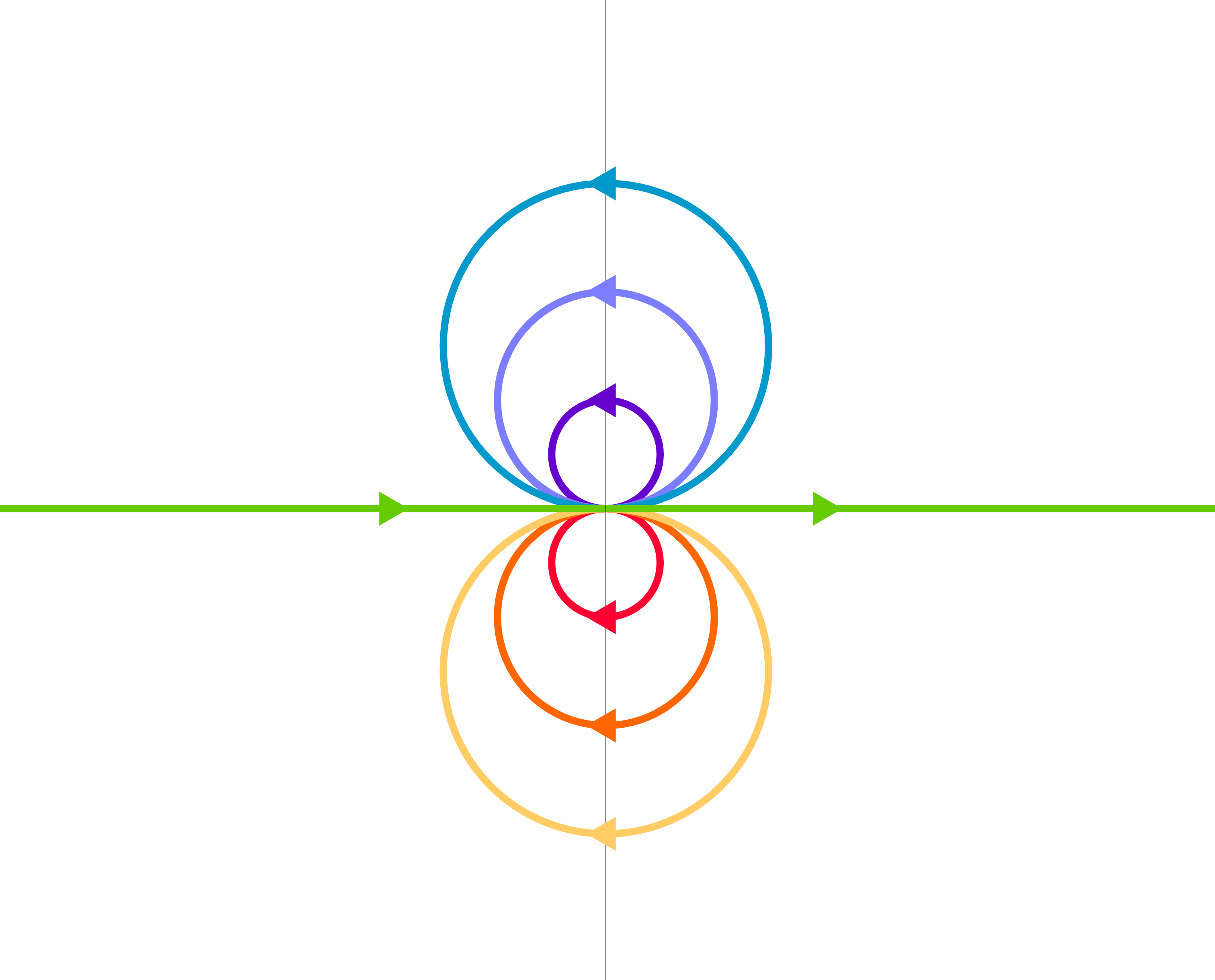}    \subcaption*{\hspace{0pt}$n=1$}
        \label{composite}
      \end{minipage} &
    \hspace{15pt}
            \vspace{-10pt}
      \begin{minipage}[t]{0.45\hsize}
        \centering
                \mbox{\raisebox{0pt}{\includegraphics[scale=0.65]{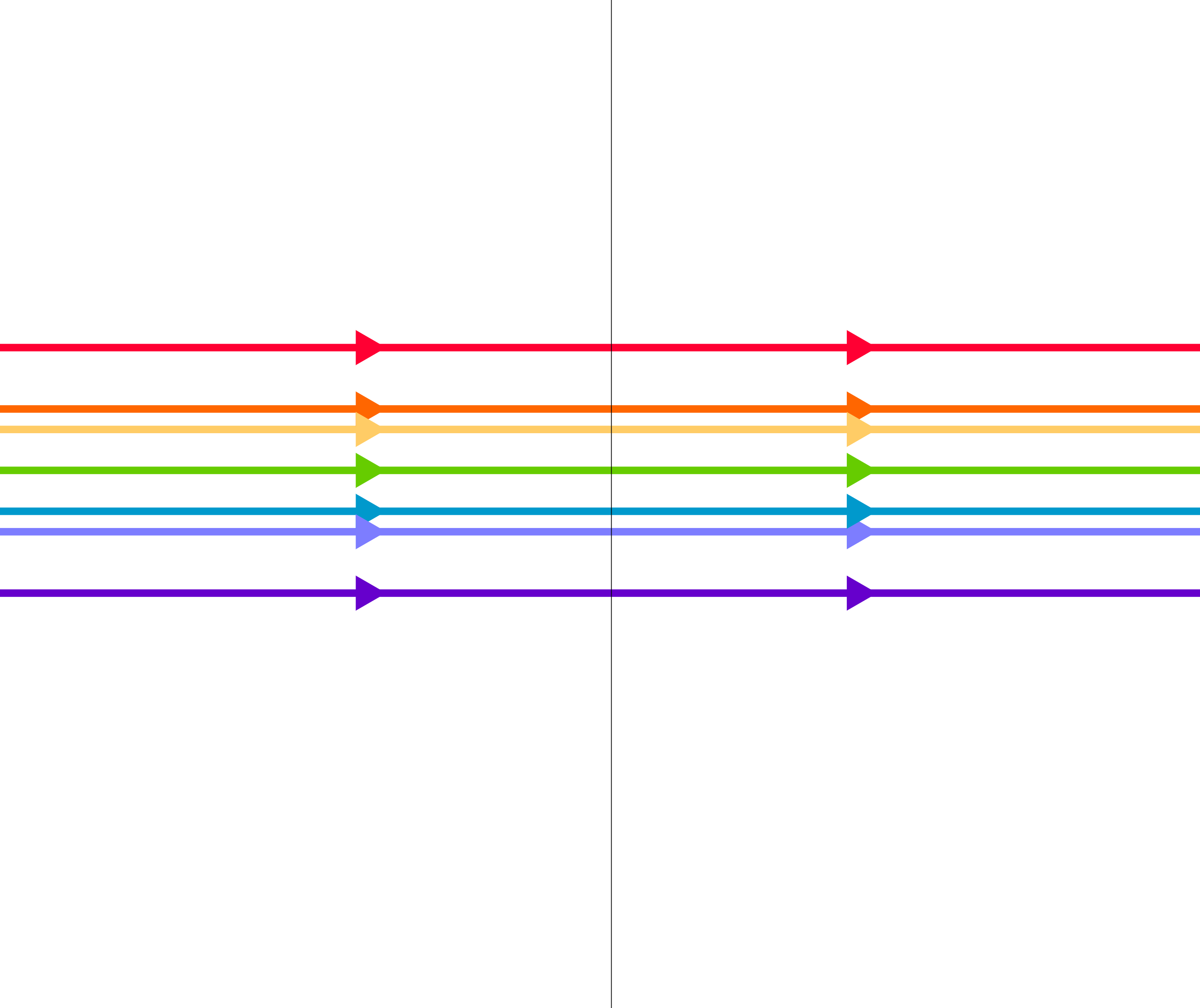}}}
           \subcaption*{\hspace{10pt}$n=-1$}
        \label{Gradation}
      \end{minipage} \\
    \hspace{-30pt}\\
      \vspace{10pt}
      \begin{minipage}[c]{0.45\hsize}
        \centering
    \hspace{-14pt}\includegraphics[scale=1.1]{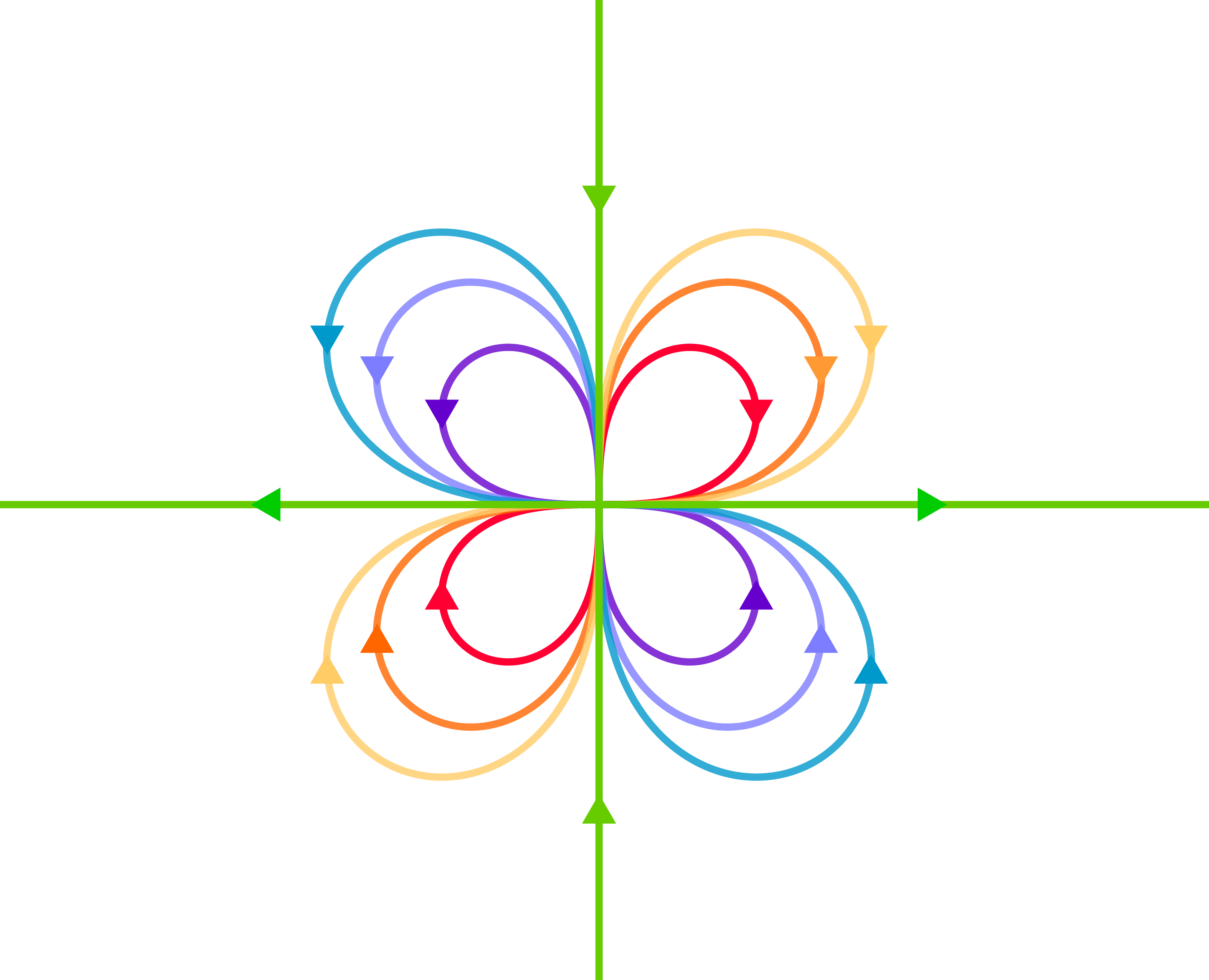} \subcaption*{\hspace{-15pt}$n=2$}
        \label{composite}
      \end{minipage} &
    \hspace{18pt}
      \begin{minipage}[c]{0.45\hsize}
         \centering
                  \mbox{\raisebox{0pt}{\includegraphics[scale=0.78]{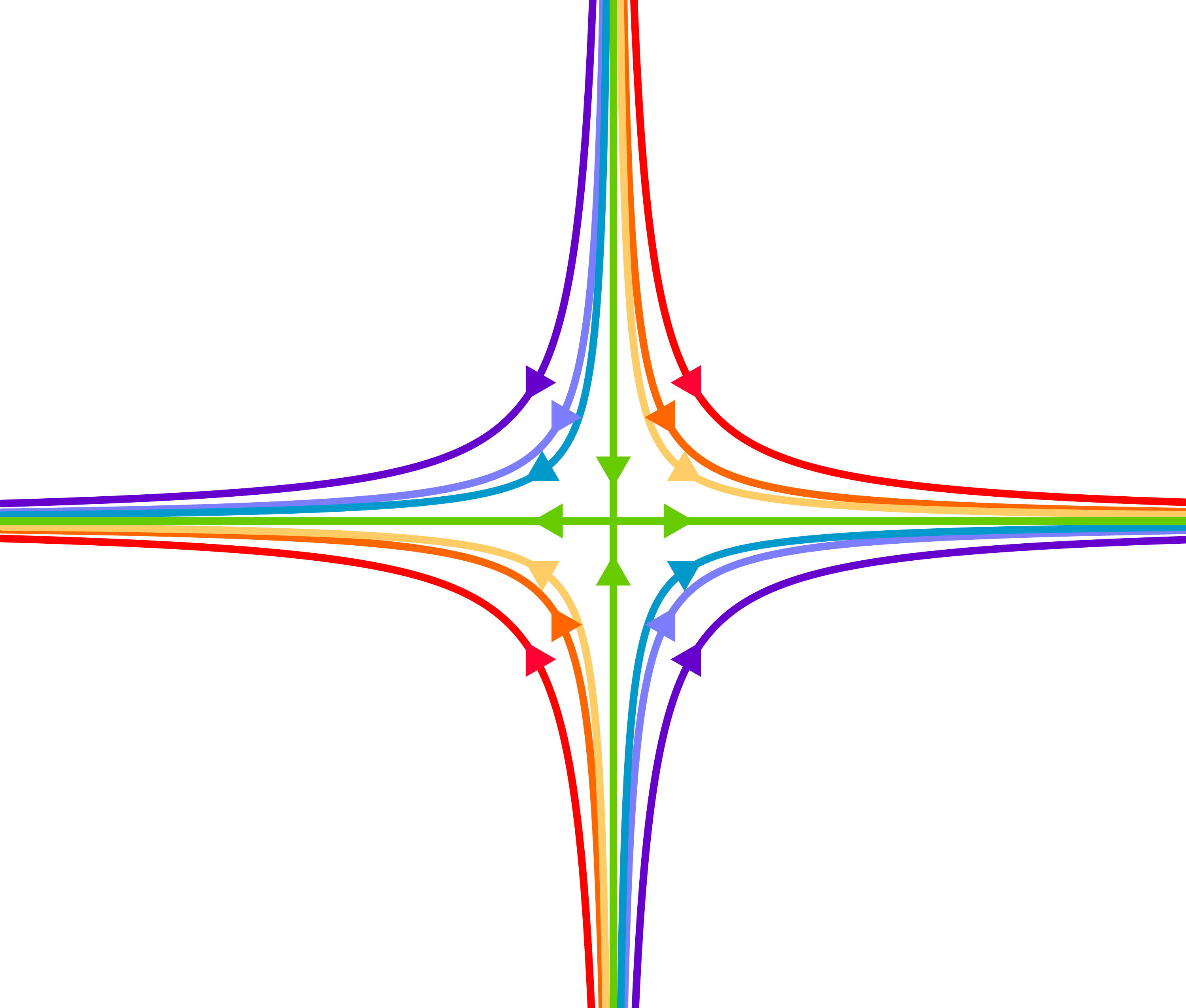} }  
                   }
                   \subcaption*{\hspace{6.5pt}$n=-2$}  
        \label{Gradation}
      \end{minipage} \\   
    \hspace{-11pt}      \begin{minipage}[c]{0.45\hsize}
        \centering
           \mbox{\raisebox{0pt}{\includegraphics[scale=1.26]{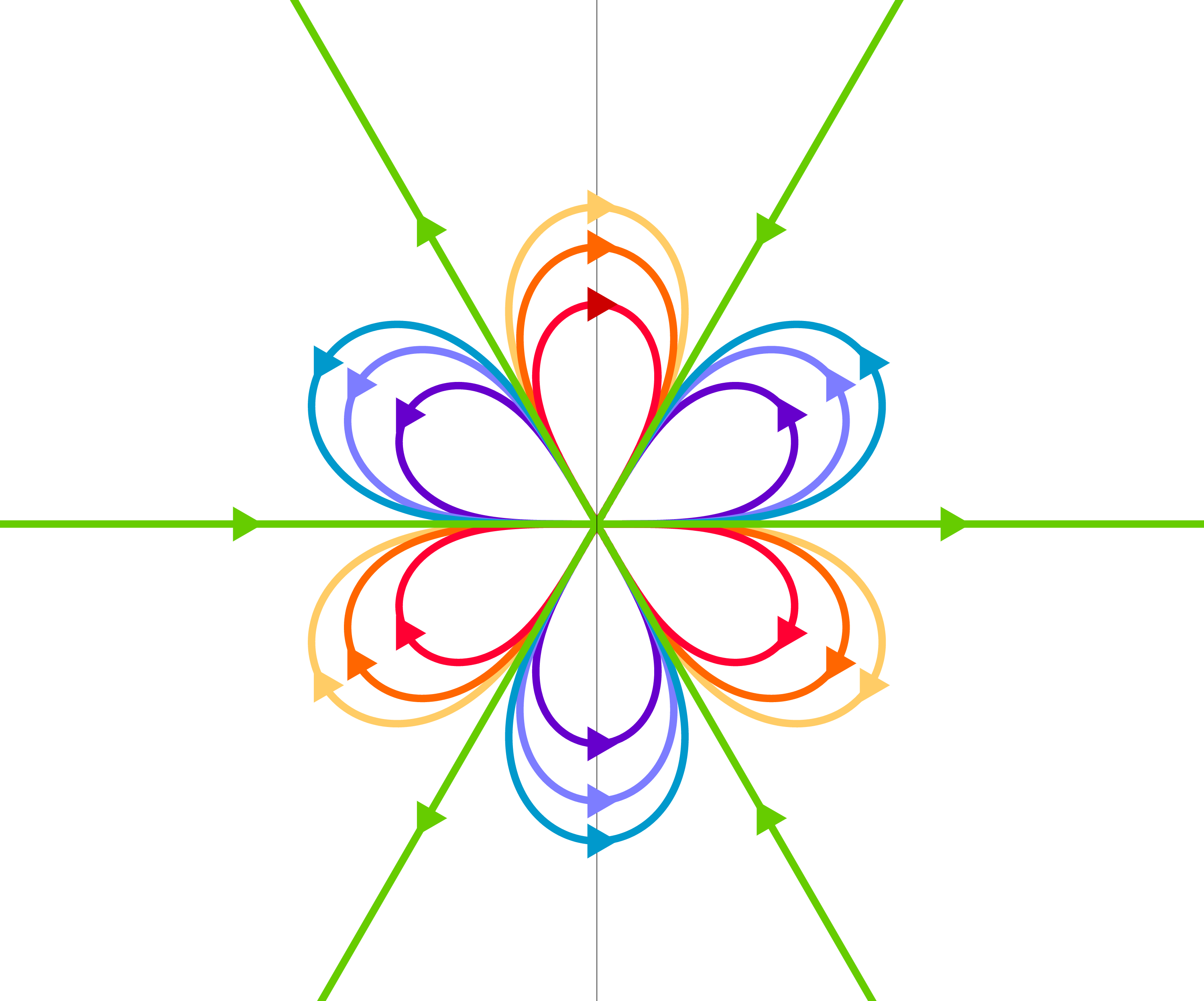} }}
    \subcaption*{\hspace{-4pt}$n=3$}
        \label{fill}
      \end{minipage} &
    \hspace{14pt}    \vspace{10pt}
          \begin{minipage}[c]{0.45\hsize}
        \centering
   \mbox{\raisebox{0pt}{\includegraphics[scale=0.7]{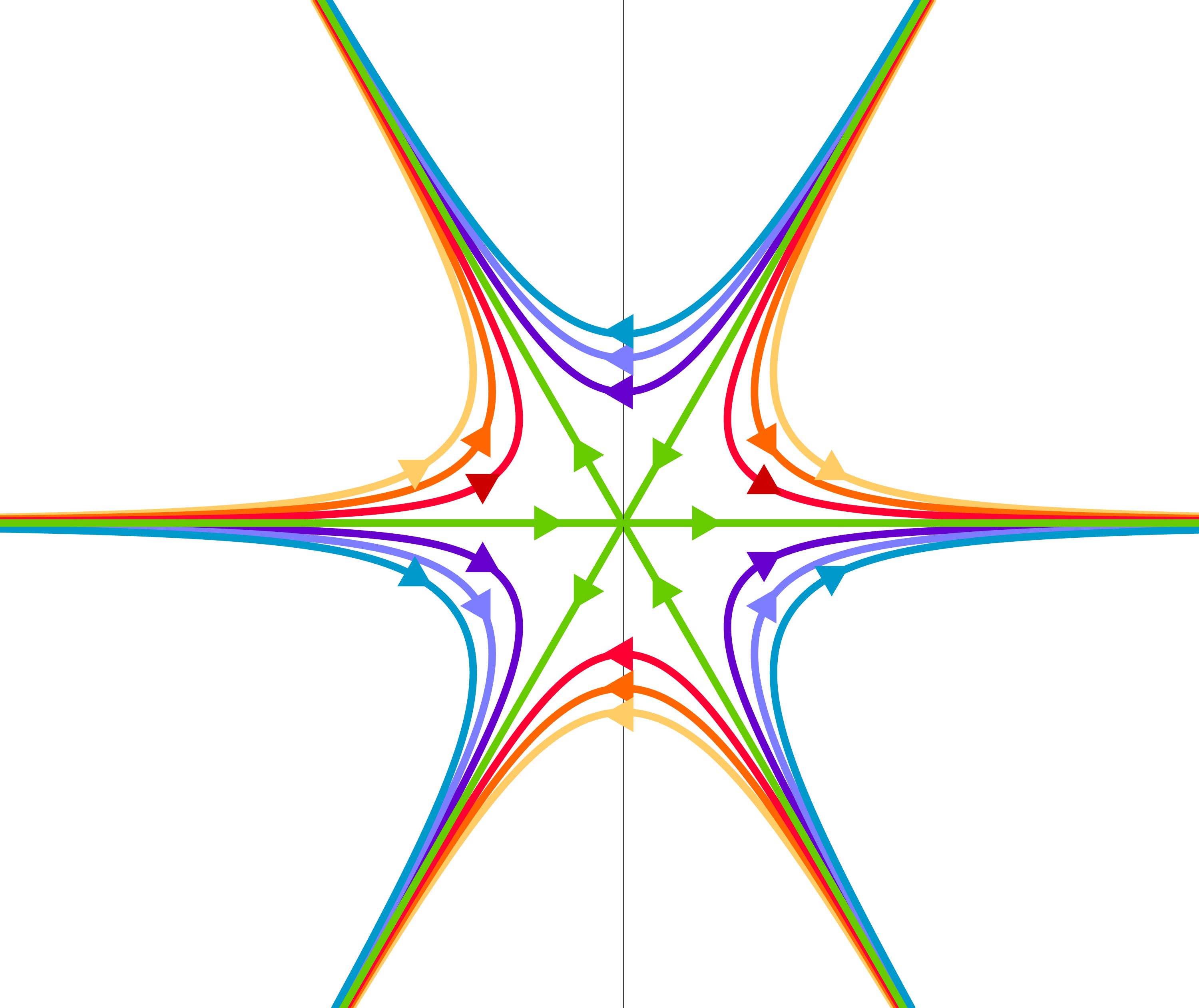} }}
        \subcaption*{\hspace{11pt}$n=-3$}
               \label{transform}
      \end{minipage}
    \end{tabular}
    \vspace{5pt}
    \caption{The level sets of $\psi(x_1,x_2)=\sin{(n\phi)}/r^{n}$ for $n=\pm1$ (dipole), $\pm2$ (quadrupole), and $\pm3$ (hexapole). The sets $\psi^{-1}(k)$ are represented in purple  ($k=1$), blue ($k=1/2$), light blue ($k=1/3$), green ($k=0$), yellow ($k=-1/3$), orange ($k=-1/2$), and red ($k=-1$). 
}
  \end{figure}


\newpage

\begin{rems}
(i) Choffrut and \v{S}ver\'{a}k \cite{CS12} investigated a local one-to-one correspondence between smooth 2D steady states $u=(u^{1},u^{2},0)$ and co-adjoint orbits of the nonstationary problem in an annulus for steady states whose stream function $\psi$ and vorticity $\omega$ have no critical points and satisfy nondegeneracy conditions. The stream function and vorticity of ($-\alpha$)-homogeneous solutions in Theorem 1.7 (iii) and (iv) are the following: 

\begin{equation}
\begin{aligned}
\psi(x_1,x_2)&=\frac{w(\phi)}{r^{\alpha-1}},\\
\omega&=c\psi|\psi|^{\frac{2}{\alpha-1}},
\end{aligned}
\end{equation}\\
for some function $w(\phi)$ on $[-\pi,\pi]$ and a positive constant $c>0$. Their gradients are the following:

\begin{align*}
|\nabla\psi|^{2}=\frac{|\alpha-1|^{2}w^{2}+|w'|^{2} }{r^{2\alpha}},\quad \nabla \omega=\frac{\alpha+1}{\alpha-1}|\psi|^{\frac{2}{\alpha-1}}\nabla \psi.
\end{align*}\\
For $\alpha>1$, $\psi$ has no critical points in $\mathbb{R}^{2}\backslash \{0\}$ because $w$ and $w'$ do not vanish at the same point (Remarks B.3 (iii)). The vorticity $\omega$ has critical points on, e.g., $\{x_2=0\}$. For $\alpha<-1$, both $\psi$ and $\omega$ have critical points at the origin. We remark that Choffrut and Sz\'{e}kelyhidi \cite{Cho14} demonstrated the existence of merely bounded steady states near a given smooth steady state in $\mathbb{T}^{d}$ for $d\geq 2$ based on the convex integration.

\noindent
(ii) Hamel and Nadirashvili \cite[Theorem 1.8]{HN22} established rigidity theorems for the 2D Euler equations in bounded annuli, exteriors of disk, punctured disks, and punctured planes. It is shown that all solutions of the 2D Euler equations in a punctured plane $u=(u^{1},u^{2},0)\in C^{2}(\mathbb{R}^{2}\backslash \{0\})$ satisfying 

\begin{equation}
\begin{aligned}
&|u|>0\quad \textrm{in}\ \mathbb{R}^{2}\backslash \{0\},\\
&\liminf_{r\to\infty}|u|>0,\\
&u\cdot e_{r}=o\left(\frac{1}{r}\right)\quad \textrm{as}\ r\to\infty,\\
&\int_{\{r=\varepsilon\}}|u\cdot e_{r}|\dd H \to 0\quad \textrm{as}\ \varepsilon\to 0,
\end{aligned}
\end{equation}\\
are circular flows, i.e., $u\cdot e_{r}=0$. The vector field $(1.4)_2$ is a noncircular irrotational ($-\alpha$)-homogeneous solution to the 2D Euler equation in $\mathbb{R}^{2}\backslash \{0\}$, violating the conditions $(1.7)_2$ and $(1.7)_4$ for $n>0$ and $(1.7)_3$ for $n<0$. The solutions in Theorem 1.7 (iii) are examples of the noncircular rotational ($-\alpha$)-homogeneous solutions for $\alpha>1$ in $\mathbb{R}^{2}\backslash \{0\}$.

\noindent
(iii) It is shown in the work of Hamel and Nadirashvili \cite[Theorem 1.1]{HN19} that all solutions of the 2D Euler equations in the plane $u=(u^{1},u^{2},0)\in C^{2}(\mathbb{R}^{2})$ satisfying 

\begin{align*}
&\sup_{\mathbb{R}^{2}}|u|<\infty,\\
&\inf_{\mathbb{R}^{2}}|u|>0,
\end{align*}\\
are shear flows, i.e., $u=(u^{1}(x_2),0,0)$, for some function $u^{1}$ with a constant strict sign by a suitable rotation. (Koch and Nadirashvili \cite{Nadirashvili2015} discuss analyticity of streamlines and Hamel and Nadirashvili \cite{HN17} discuss a rigidity theorem in a strip). The vector field $(1.4)_2$ for $n<0$ is an irrotational ($-\alpha$)-homogeneous solution to the 2D Euler equation in $\mathbb{R}^{2}$. This solution is constant (a shear flow) for $n=-1$ and has a stagnation point at $x=0$ and is growing as $|x|\to\infty$ for $n\leq -2$ (Figure 2). The solutions $u\in C^{1}(\mathbb{R}^{2}\backslash \{0\})\cap C(\mathbb{R}^{2})$ in Theorem 1.7 (iv) are examples of the  nonshear rotational ($-\alpha$)-homogeneous solutions for $\alpha<-1$ in $\mathbb{R}^{2}$.  

\noindent
(iv) It is a conjecture \cite[chapter 34]{SverakLec}, \cite{Sh13} that vorticity of the 2D nonstationary Euler equation is generically weakly compact but not strongly compact as $t\to\infty$. Glatt-Holtz et al. \cite[2.2]{GSV15} discuss the relationship between the compactness of vorticity and coherent structures at the end state. The behavior of solutions around shear and circular flows are investigated in perturbative regimes. We refer to the important works \cite{BM15}, \cite{BPM19}, \cite{DM18}, \cite{MZ20}, \cite{IJ20}, \cite{IJ22}, \cite{IJ22b} on the nonlinear asymptotic stability of the 2D Euler equations. 
\end{rems}

\vspace{3pt}

The stream function level sets of 2D reflection symmetric ($-\alpha$)-homogeneous solutions in the $(x_1,x_2)$-upper half plane model those of axisymmetric ($-\alpha$)-homogeneous solutions in the $(z,r)$-upper half plane (cross-section).

\vspace{13pt}

\begin{table}[h]
\begin{minipage}[b]{0.45\linewidth}
 \hspace{-110pt}
\begin{tabular}{|c|cl|c|c|c|cc|}
\hline
$\alpha$   & \multicolumn{1}{c|}{$-2$}     & \multicolumn{1}{c|}{$-1$} & $0$                         & $1$ & $2$          & \multicolumn{1}{c|}{$3$}        & $4$      \\ \hline
Pole       & \multicolumn{1}{l|}{Hexapole} & Quadrupole                & \multicolumn{1}{l|}{Dipole} & -   & Dipole       & \multicolumn{1}{l|}{Quadrupole} & Hexapole \\ \hline
Level set & \multicolumn{2}{c|}{Wedged curve}                         & Line                        & -   & Jordan curve & \multicolumn{2}{c|}{Multifoil}             \\ \hline
\end{tabular}
\vspace{15pt}
\subcaption{The level sets of $\psi(x_1,x_2)=\sin(n\phi)/r^{n}$ for $n=\alpha-1$ and $\alpha\in \mathbb{Z}\backslash \{1\}$ in the $(x_1,x_2)$-upper half plane.}
\end{minipage}\\
\vspace{15pt}
 \begin{minipage}[b]{0.45\linewidth}
 \hspace{-110pt}
\begin{tabular}{|c|cl|c|c|c|c|cc|}
\hline
$\alpha$   & \multicolumn{1}{c|}{$-2$}     & \multicolumn{1}{c|}{$-1$} & $0$                         & $1$ & $2$ & $3$          & \multicolumn{1}{c|}{$4$}        & $5$      \\ \hline
Pole       & \multicolumn{1}{l|}{Hexapole} & Quadrupole                & \multicolumn{1}{l|}{Dipole} & -   & -   & Dipole       & \multicolumn{1}{l|}{Quadrupole} & Hexapole \\ \hline
Level set & \multicolumn{2}{c|}{Wedged curve}                         & Line                        & -   & -   & Jordan curve & \multicolumn{2}{c|}{Multifoil}             \\ \hline
\end{tabular}
\vspace{15pt}
\subcaption{The level sets of $\psi(z,r)=w_{n+1}(\cos\theta)/\rho^{n}$ for $n=\alpha-2$ and $\alpha\in \mathbb{Z}\backslash \{1,2\}$ in the $(z,r)$-upper half plane.}
\end{minipage}
\vspace{15pt}
\caption{Integers $\alpha$ and stream function level sets for the irrotational $(-\alpha)$-homogeneous solutions to (1.1): (A) 2D reflection symmetric case and (B) axisymmetric case.}
\end{table}

The irrotational axisymmetric $(-\alpha)$-homogeneous solutions are the following:

\begin{equation}
\begin{aligned}
u&=
\frac{1}{\rho^{3}}x,\quad \alpha=2,\\
\quad
u&=\nabla \times (\psi\nabla \phi),\quad \alpha\in \mathbb{Z}\backslash \{1,2\},
\end{aligned}
\end{equation}\\
for the axisymmetric stream function, 

\begin{equation}
\begin{aligned}
\psi(z,r)&=\frac{w_{n+1}(\cos\theta) }{\rho^{n}},\quad n=\alpha-2,\\
w_{n+1}(t)&=-\int_{-1}^{t}P_{n}(s)\dd s.
\end{aligned}
\end{equation}\\
Here, $(\rho,\theta)$ is the polar coordinates in the $(z,r)$-plane and $P_{n}$ is the Legendre function (Remark 2.3). The function $w_{n+1}$ is a polynomial of degree $n+1$ whose all $n+1$ zero points are lying on $[-1,1]$. The function $w_{n+1}$ for $\alpha$ is the same as that for $3-\alpha$. For $n=1,2,3$, they are as follows: 

\begin{equation}
\begin{aligned}
w_{2}(t)&=\frac{1}{2}(1-t^{2}),\quad \alpha=3,0, \\
w_{3}(t)&=\frac{1}{2}t(1-t^{2}),\quad \alpha=4,-1,\\
w_{4}(t)&=-\frac{1}{8}(5t^{4}-6t^{2}+1),\quad \alpha=5,-2.
\end{aligned}
\end{equation}\\
Table 1 (B) shows the level sets of the stream functions (1.9) in the $(z,r)$-upper half plane for the irrotational axisymmetric $(-\alpha)$-homogeneous solutions. The stream function level sets $\{\psi=C\}$ for $C>0$ of the irrotational axisymmetric ($-3$)-homogeneous are the Jordan curves. We show the existence of rotational axisymmetric $(-\alpha)$-homogeneous solutions for $2<\alpha<3$ whose stream function level sets are homeomorphic to those of the irrotational axisymmetric $(-3)$-homogeneous solution.

\vspace{3pt}

\begin{thm}
For $2<\alpha<3$, there exist rotational axisymmetric $(-\alpha)$-homogeneous solutions to (1.1) whose stream function level sets are homeomorphic to those of the irrotational axisymmetric $(-3)$-homogeneous solution. Those solutions exist in the classes (iii) of Theorems 1.1, 1.4, and 1.5. 
\end{thm}

\vspace{3pt}

For smooth Euler flows in a bounded domain, the level sets of $\Pi$ (Bernoulli surfaces) are nested tori or nested cylinders \cite{AK21}. For smooth Beltrami flows, the level sets of $\varphi$ are not diffeomorphic to a sphere \cite{EP16}.

The solutions in Theorem 1.10 are expressed by the Clebsch representation:

\begin{align}
u=\nabla \times (\psi \nabla \phi)+\Gamma\nabla \phi.
\end{align}\\
Their Bernoulli function and toroidal component are as follows: 

\begin{align}
\Pi=C_1|\psi|^{2+\frac{4}{\alpha-2}},\quad \Gamma=C_2\psi |\psi|^{\frac{1}{\alpha-2}},
\end{align}\\
for some constants $C_1$ and $C_2$. For $C_1=0$, this solution is a Beltrami flow with the proportionality factor (1.3) for $C=C_2$. For $C_1\neq 0$, this solution is an Euler flow with a nonconstant Bernoulli function. The level sets of $\Pi$ and $\varphi$ are identified with surfaces created by the rotation of the level sets of $\psi$ in the $(z,r)$-upper half plane. They provide new examples of the Bernoulli surfaces and level sets of $\varphi$, cf. \cite{AK21}, \cite{EP16}.

\vspace{3pt}

\begin{cor}
The level sets of $\Pi$ and $\varphi$ of the rotational axisymmetric $(-\alpha)$-homogeneous solutions for $2<\alpha<3$ in Theorem 1.10 are nested surfaces created by the rotation of the sign $``\infty"$ around the $z$-axis.
\end{cor}

\vspace{3pt}

\begin{rem}
The level sets of $\Pi$ and $\varphi$ of the rotational axisymmetric $(-\alpha)$-homogeneous solutions for $\alpha \in \mathbb{R}\backslash [0,3)$ in (iii) and (iv) of Theorems 1.1, 1.4, and 1.5 are \textit{not} nested surfaces created by the rotation of the sign $``\infty"$. For $\alpha\geq 3$, they are nested surfaces created by the rotation of multifoils around the $z$-axis (Remark 4.11). 
\end{rem}

\vspace{3pt}

\subsection{Physical background to homogeneous force-free fields}

The axisymmetric homogeneous Beltrami flows (force-free fields) have been studied by astrophysicists to model magnetic fields around the Sun. Low and Lou \cite{LL90} derived a second-order nonlinear ODE describing axisymmetric homogeneous force-free fields (Subsection 1.4). Wolfson and Low \cite{Wolfson92} studied the dipole structure of axisymmetric homogeneous force-free fields. Lynden-Bell and Boily \cite{LB94} observed the quadrupole structure and derived asymptotics of axisymmetric homogeneous force-free fields toward a current sheet as $\alpha\to2$. Aly \cite{Aly94} considered its 2D model. Axisymmetric homogeneous force-free fields are also used to calculate the magnetic helicity of Sun's corona \cite{Low12}. We refer to the more recent works by Gourgouliatos \cite{Gourgouliatos}, Lerche \cite{Lerche}, Lynden-Bell and Moffatt \cite{Lynden-Bell15}, Luna et al. \cite{Luna}, and Lyutikov \cite{Lyutikov} on homogeneous solutions in astrophysics. 

The dipole is a standard model to describe magnetic fields around stars and planets, e.g., Earth and Jupiter. On the other hand, quadrupoles or multipoles also appear for the theoretical models of star's magnetic fields \cite{Brun}, \cite{Sokoloff}, \cite{Maiewski}. We remark that Jupiter's magnetic field has an asymmetric structure with a nondipole in the northern hemisphere and a dipole in the southern hemisphere, according to the results observed by the Juno spacecraft \cite{Moore18}.  

The existence (or nonexistence) of asymmetric solutions is related to Grad's conjecture for the Euler equations \cite{Grad67}, \cite{Grad85}, \cite[Conjecture 1]{CDG21b}. Asymmetric solutions to the Euler equations are constructed in the work by Constantin et al. \cite{CDG21} with force and in the works by Bruno and Laurence \cite{BL96} and Enciso et al. \cite{ELP21} with a piecewise constant Bernoulli function. Constantin and Pasqualotto \cite{Pasqualotto}, \cite{CP22} constructed smooth solutions to (1.1) with a nonconstant Bernoulli function by magnetic relaxation via the Voigt--MHD system without assuming any symmetries.  

The existence of rotational asymmetric homogeneous solutions to (1.1) is unknown. Irrotational homogeneous solutions consist of one axisymmetric solution without swirls and other $2|\alpha-3/2|-1$ nonaxisymmetric solutions for each $\alpha\in \mathbb{Z}\backslash \{1\}$ (Theorem 2.2).

\vspace{3pt}

\subsection{Self-similar solutions to the Euler equations}
Homogeneous solutions to (1.1) are \textit{self-similar} solutions to the nonstationary Euler equations. According to Chae and Wolf \cite{CW20} and Tsai \cite[Chapter 8]{Tsaibook}, we say that $u=u(x,t)$ is a self-similar solution to the Euler equations if there exists $\alpha\in \mathbb{R}$ such that  $u(x,t)=\lambda^{\alpha} u(\lambda x, \lambda^{\alpha+1}t)$ for all $\lambda>0$ and $(x,t)$. Self-similar solutions can be defined in the following domains:\\

\noindent
(i) Homogeneous ($u$ is time independent): $\mathbb{R}^{3}\backslash \{0\}$ or $\mathbb{R}^{3}$\\
(ii) Backward: $\mathbb{R}^{3}\times (-\infty,0)$ or $\mathbb{R}^{3}\times (-\infty,0]$  \\
(iii) Forward: $\mathbb{R}^{3}\times (0,\infty)$ or $\mathbb{R}^{3}\times [0,\infty)$  \\

Self-similarity of the Euler equations is \textit{second} kind \cite{Bare96}, \cite{EF09} in the sense that the scaling law has the freedom on the choice of the parameter $\alpha\in \mathbb{R}$.

Backward self-similar solutions to the Navier--Stokes equations appeared in the work of Leray \cite{Leray1934} as the first candidate of the self-similar singularity. Their nonexistence is demonstrated in the works by Necas et al. \cite{NRS} and Tsai \cite{Tsai98}. Forward self-similar solutions are related to uniqueness and large time asymptotics. The existence of large forward self-similar solutions is demonstrated in the work by Jia and Sverak \cite{JS} (also in the works by Tsai \cite{Tsai14}, Korobkov and Tsai  \cite{KT16}, and Bradshaw and Tsai \cite{BT17}). We refer to the works by Jia et al. \cite{JST18} and Tsai \cite{Tsaibook} for self-similar solutions to the Navier-Stokes equations. Jia and \v{S}ver\'{a}k \cite{JS}, \cite{JS15}, Guillod and \v{S}ver\'{a}k, \cite{GuS}, and Albritton et al. \cite{ABC} also describe the nonuniqueness of the Leray--Hopf solutions.

Chae \cite{Chae07}, \cite{Chae15}, \cite{Chae15b}, \cite{Chae15c}, Chae et al. \cite{CKL}, Chae and Shvydkoy \cite{ChaeShvydkoy13}, and Chae and Tsai \cite{CT14} investigated the nonexistence of backward self-similar solutions to the Euler equations for $\alpha\neq -1$:

\begin{align*}
u(x,t)=\frac{1}{(-t)^{\frac{\alpha}{\alpha+1}}}u\left(\frac{x}{(-t)^{\frac{1}{\alpha+1}}},-1\right).
\end{align*}\\
Their nonexistence is demonstrated in the work by Chae and Shvydkoy \cite{ChaeShvydkoy13} under the condition $u(y,-1)\in L^{p}\cap C^{1}(\mathbb{R}^{3})$ for $3\leq p\leq \infty$ and

\begin{align*}
-1<\alpha\leq \frac{3}{p}\ \textrm{or}\ \  \frac{3}{2}<\alpha\leq \infty.
\end{align*}\\
The case $\alpha=3/2$ is referred to as the energy-conserving scale. Shvydkoy \cite{Shvydkoy13} introduced the energy measure to investigate the concentration of the energy at the blowup time (also in the works by Bronzi and Shvydkoy \cite{BronziShvydkoy} and Leslie and Shvydkoy \cite{LeslieShvydkoy2018}). Chae and Wolf \cite{CW20} demonstrated the nonexistence of backward self-similar solutions at the energy-conserving scale $\alpha=3/2$ by using the energy measure. 

Elgindi \cite{Elgindi} demonstrated the existence of backward self-similar solutions to the Euler equations for some $\alpha\in \mathbb{R}$. The backward self-similar solution \cite{Elgindi} is with infinite energy and axisymmetric without swirls $u(y,-1)\in C^{1,\gamma}(\mathbb{R}^{3})$ ($0<\gamma<1$). Elgindi et al. \cite{EGM2} demonstrated the existence of blowup solutions for axisymmetric data with swirls with compactly supported vorticity in $C^{1,\gamma}\cap L^{2}(\mathbb{R}^{3})$ ($0<\gamma<1$) by showing the stability of the backward self-similar solution.

It is remarkable that backward self-similar solutions to the one-dimensional (1D) inviscid Bergers equation $u_t+u u_x=0$ exist for $-1<\alpha<0$ and admit the profile $u(y,-1)$ growing as $|y|\to\infty$ \cite{EF09}, \cite{CGM22}. More specifically, the profile $u(y,-1)$ is analytic if

\begin{align*}
\alpha=\frac{-1}{2i+1},\quad i=1,2,\dots.
\end{align*}\\
Collot et al. \cite{CGM22} classified all scale invariant solutions to the 1D inviscid Bergers equation and demonstrated their existence. In particular, backward self-similar solutions exist even for nonintegers $i>0$. The 2D Burgers equation is related to the 2D Prandtl equations. Collot et al. \cite{CGM21} demonstrated the existence of backward self-similar solutions to the inviscid 2D Prandtl equations. Collot et al. \cite{CGIM22} described the self-similar blow-up behavior for the 2D Prandtl equations.

Elling \cite{Elling} constructed forward self-similar weak solutions $u=(u^{1},u^{2},0)$ to the 2D Euler equations for $-1<\alpha<1/2$ with initial vorticity, 

\begin{align*}
\omega(x,0)=\frac{1}{|x|^{\alpha+1}}\omega\left(\frac{x}{|x|},0\right).
\end{align*}\\
The work by Elling \cite{Elling} is motivated by the numerical results on bifurcation of self-similar vortex sheets \cite{Pullin78}, \cite{Pullin89}. Bressan and Murray \cite{Bressan20} and Bressan and Shen \cite{BS21} also discuss the nonuniqueness of forward self-similar weak solutions for vortex initial data in $L^{p}$ for $1\leq p<\infty$. Vishik \cite{Vishik18}, \cite{Vishik18b} demonstrated the nonuniqueness of weak solutions with force in this class, cf. \cite{ABCDGK}. In a recent work, Mengual \cite{Mengual23} constructed admissible weak solutions for vortex initial data in $L^{p}\cap L^{1}$ for $2<p<\infty$ based on the convex integration.

Elgindi and Jeong \cite{EJ20b} constructed unique forward self-similar weak solutions for $\alpha=-1$,  

\begin{align*}
\omega(x,t)= \omega\left(\frac{x}{|x|},t\right)=h(\phi,t),
\end{align*}\\
for bounded and discretely symmetric initial vorticity. Elgindi et al. \cite{Elgindi2022} investigated the large time behavior of this solution. Jeong \cite{Jeong17} investigated 3D solutions and Drivas and Elgindi \cite{DE} described a review on symmetries and singularity formations of the Euler equations.

We finally discuss the self-similar logarithmic spiral vortex sheets of Prandtl \cite{Prandtl} and Alexander \cite{Al}. They have been important objects for forward self-similar solutions to the Birkhoff--Rott equation or the Euler equations. We refer to Mengual and Sz\'{e}kelyhidi \cite{Mengual} on the existence of admissible weak solutions to the 2D Euler equations for general vortex sheet initial data based on the convex integration. In recent works, Elling and Gnann \cite{EG} demonstrated that Alexander spirals are the solutions of the Birkhoff--Rott equation when the number of sheets is larger than 3. More recently, Cie\'{s}lak et al. \cite{Cieslak1}, \cite{Cieslak2} established a sufficient condition for a general logarithmic spiral vortex sheet including Prandtl and Alexander spirals to be weak forward self-similar solutions to the Euler equations for $\alpha\in \mathbb{R}\backslash \{-1\}$. Jeong and Said \cite{Jeong23} demonstrated the existence of forward self-similar solutions of the form 

\begin{align*}
\omega (x,t)=h(\phi-\beta \ln{r},t),\quad \beta>0,
\end{align*}\\
representing logarithmic spiral vortex sheets and investigated their large time behavior. Garc\'{i}a and G\'{o}mez-Serrano \cite{Serrano22} demonstrated the existence of forward self-similar solutions to the SQG equations.

\vspace{3pt}

\subsection{Ideas of the proof}

The main results of this study consist of two parts:\\

\noindent
(a) Rigidity: Theorems 1.3, 1.4 (ii), and 1.7 (i), (ii)  \\
(b) Existence: (iii) and (iv) of Theorems 1.1, 1.4, and 1.5 and Theorems 1.8 and 1.10\\

The proofs of the rigidity are based on the homogeneous solution's equations on the sphere \cite{Sverak2011}, \cite{Shv}. We demonstrate Theorem 1.3 by an analog of the rigidity of tangentially homogeneous solutions. We show that the integral curves of the tangential component of radially irrotational homogeneous solutions on the sphere are geodesics (Lemma 2.10). 

The proof of Theorem 1.4 (ii) is based on the first integral of $u$ for axisymmetric solutions without swirls,

\begin{align*}
u\cdot \nabla \zeta=0,\quad \zeta=\frac{\nabla \times u\cdot e_{\phi}}{r}.
\end{align*}\\
The nonexistence of rotational axisymmetric ($-\alpha$)-homogeneous solutions for $0\leq \alpha\leq 2$ \cite{Shv} is based on the first integral of $u$,

\begin{align*}
u\cdot \nabla \Pi=0,\quad \Pi=p+\frac{1}{2}|u|^{2}.
\end{align*}\\
The Bernoulli function $\Pi$ is ($-2\alpha$)-homogeneous and implies the nonexistence of rotational axisymmetric ($-\alpha$)-homogeneous solutions for $0\leq \alpha\leq 2$. The quantity $\zeta$ is ($-\alpha-2$)-homogeneous and will be used to demonstrate the nonexistence of rotational axisymmetric ($-\alpha$)-homogeneous solutions without swirls for $-2\leq \alpha < 0$. For the 2D reflection symmetric ($-\alpha$)-homogeneous solutions $u=(u^{1},u^{2},0)$, the vorticity $\omega=\partial_1u^{2}-\partial_2u^{1}$ is a first integral and $(-\alpha-1)$-homogeneous. We use it to show the nonexistence of rotational 2D reflection symmetric ($-\alpha$)-homogeneous solutions for $-1\leq \alpha\leq 1$ (Theorem 1.7 (i) and (ii)).

The main contribution of this study is the existence of axisymmetric ($-\alpha$)-homogeneous solutions to (1.1): (iii) and (iv) of Theorems 1.1, 1.4, and 1.5. We express axisymmetric ($-\alpha$)-homogeneous solenoidal vector fields $u$ by the Clebsch representation (1.11) and the axisymmetric stream function 

\begin{align*}
\psi=\frac{w(\cos\theta)}{\rho^{\alpha-2}}.
\end{align*}\\
We choose the functions $\Pi$ and $\Gamma$ by (1.12) so that they are first integrals of $u$ and homogeneous of degree $-2\alpha$ and $-\alpha+1$, respectively. Then axisymmetric ($-\alpha$)-homogeneous solutions to (1.1) can be constructed by the one-dimensional (1D) Dirichlet problem for $w(t)$ with $\beta=\alpha-2$:

\begin{equation}
\begin{aligned}
- w''&=\frac{\beta(\beta+1)}{1-t^{2}}w+c_1 w|w|^{\frac{4}{\beta}}+\frac{c_2}{1-t^{2}}w|w|^{\frac{2}{\beta}},\quad t\in (-1,1),\\
w(1)&=w(-1)=0. 
\end{aligned}
\end{equation}\\
This equation appeared in \cite{LL90} for $c_1=0$ and in \cite[Remark 5.5 (45)]{Shv} for $c_2=0$. The function $(1.9)_2$ is a solution for $c_1=c_2=0$. For the constants $\beta$, $c_1$, and $c_2$ satisfying

\begin{align*}
&\beta\in \mathbb{R}\backslash [-2,0],\quad c_1,c_2\geq 0,\quad c_1\neq 0\ \textrm{or}\  c_2\neq 0,\\
&c_1=0\quad \textrm{for}\ -4\leq \beta<-2,
\end{align*}\\
the nonlinear term is superlinear for $\beta>0$ and sublinear for $\beta<-2$.

We construct solutions to (1.13) by minimax theorems. The main task is to specify $\beta$ for which solutions exist. The equation (1.13) involves $\beta$ both in the linear operator $L_{\beta}=\partial_{t}^{2}+\beta(\beta+1)/(1-t^{2})$ and in the nonlinear power. For the 2D reflection symmetric ($-\alpha$)-homogeneous solutions to (1.1), the corresponding equation is an autonomous equation for $\beta=\alpha-1$ (in (1.15) in the following). The construction of solutions to the autonomous equation for superlinear $\beta>0$ using a linking theorem can be found in the book of Willem \cite[Corollary 2.19]{Willem}. The construction of solutions to some sublinear elliptic problem using a saddle point theorem can be found in the book of Rabinowitz \cite[THEOREM 4.12]{Rab}.

The main idea is to use the Chandrasekhar's transformation $\psi\longmapsto \psi/r^{2}$ and regard the equation (1.13) as an elliptic equation on $\mathbb{S}^{4}$. This map appears in the study of vortex rings and transforms stream functions in the cross-section $\mathbb{R}^{2}_{+}$ into axisymmetric functions in $\mathbb{R}^{5}$, e.g., \cite{Fra00}. For homogeneous stream functions, this map is expressed as follows:

\begin{align*}
w(t)\longmapsto \chi(\theta)=\frac{w(\cos\theta)}{\sin^{2}\theta},
\end{align*}\\
with the geodesic radial coordinate $\theta$ on $\mathbb{S}^{4}$. The equation (1.13) is then transformed into an elliptic equation on $\mathbb{S}^{4}$:

\begin{align}
-\Delta_{\mathbb{S}^{4}}\chi=(\beta-1)(\beta+2)\chi+c_1\chi |\chi|^{\frac{4}{\beta}}\sin^{2+\frac{8}{\beta}}\theta+c_2\chi |\chi|^{\frac{2}{\beta}}\sin^{\frac{4}{\beta}}\theta.
\end{align}\\
Here, $\Delta_{\mathbb{S}^{4}}$ is the Laplace--Beltrami operator on $\mathbb{S}^{4}$. We remark that the Brezis--Nirenberg problem on $\mathbb{S}^{n}$ is $-\Delta_{\mathbb{S}^{n}}\chi=\lambda \chi+\chi|\chi|^{4/(n-2)}$ and $\chi>0$ on $\mathbb{S}^{n}$, for example, the works by Brezis and Peletier \cite{Brezis06}, Brezis and Li \cite{BrezisLi06}, and Bandle and Wei \cite{Bandle08}.

We construct an orthonormal basis $\{e_i\}$ on $L^{2}(-1,1)$ consisting of the  eigenfunctions of the operator $-L_{\beta}$ associated with the eigenvalues $\mu_1\leq \mu_2\leq \dots\leq\mu_{N}\leq 0<\mu_{N+1}\leq \cdots$. Chandrasekhar's transformation yields the sign of the principal eigenvalue $\mu_1$:\\

\noindent
(i) $\mu_1>0$ for $0<\beta<1$ \\
(ii) $\mu_1\leq 0$ for $\beta\in \mathbb{R}\backslash [-2,1)$\\

We consider the direct sum decomposition $H^{1}_{0}(-1,1)=Y\oplus Z$ for a finite-dimensional subspace spanned by the eigenfunctions associated with the  nonpositive eigenvalues $Y=\textrm{span}(e_1,\dots,e_N)$ and its orthogonal space $Z=\{z\in H^{1}_{0}(-1,1)\ |\ \int_{-1}^{1}zy\dd t=0,\ y\in Y \}$. We construct solutions to (1.13) by seeking critical points of the associated functional $I\in C^{1}(H^{1}_{0}(\Omega); \mathbb{R})$. According to the sign of $\mu_1$ and the nonlinear power, we apply three minimax theorems:\\ 

\noindent
(i) Mountain pass theorem for $0<\beta<1$\\
(ii) Linking theorem for $1\leq \beta$\\
(iii) Saddle point theorem for $\beta<-2$\\

For each case, we check that $I$ satisfies the Palais--Smale condition and $\inf_{N}I>\max_{M_0}I$ for suitable subsets $N$ and $M_0\subset H^{1}_{0}(-1,1)$. For $0<\beta<1$, we consider a modified problem to (1.13) and construct positive solutions by a mountain pass theorem. We outline the variational construction in detail in subsection 3.2.

The existence of solutions to the 1D Dirichlet problem (1.13) implies the existence of axisymmetric ($-\alpha$)-homogeneous solutions to (1.1) outside of the nonexistence range $0\leq \alpha\leq 2$. Namely, the existence of solutions to (1.13) for $c_1\neq 0$ implies the existence of axisymmetric $(-\alpha)$-homogeneous solutions with a nonconstant Bernoulli function ((iii) and (iv) of Theorems 1.4 and 1.5). The existence of solutions to (1.13) for $c_1=0$ implies the existence of axisymmetric Beltrami $(-\alpha)$-homogeneous solutions (Theorem 1.1 (iii) and (iv)). The existence of positive solutions to (1.13) for $0<\beta<1$ implies the existence of axisymmetric ($-\alpha$)-homogeneous solutions to (1.1) whose stream function level sets are the Jordan curves (Theorem 1.10). We show the regularity of $(-\alpha)$-homogeneous solutions at poles by using the geodesic radial coordinate $\theta$ and the function $\chi$.\\

For 2D reflection symmetric $(-\alpha)$-homogeneous solutions $u=(u^{1},u^{2},0)$, we choose stream functions and vorticity by (1.6). Then the associated 1D Dirichlet problem is an autonomous equation for $w(\phi)$ with $\beta=\alpha-1$:

\begin{equation}
\begin{aligned}
-w''&=\beta^{2}w+cw|w|^{\frac{2}{\beta}},\quad \phi\in (0,\pi),\\
w(0)&=w(\pi)=0.
\end{aligned}
\end{equation}\\
The construction of solutions to (1.15) is easier than that of solutions to (1.13). The principal eigenvalue of the operator $-\partial_{\phi}^{2}$ is positive for $0<\beta<1$ and nonpositive for $\beta\in \mathbb{R}\backslash [-2,1)$ because the eigenvalues are $\mu_{n}=n^{2}-\beta^{2}$ for $n\in \mathbb{N}$. The aforementioned minimax theorems imply the existence of solutions to (1.15) for $\beta\in \mathbb{R}\backslash [-2,0]$. The solutions of (1.15) provide two-and-a-half-dimensional ($2\nicefrac{1}{2}$D) reflection symmetric ($-\alpha$)-homogeneous solutions to (1.1) (Theorem B.1) including 2D reflection symmetric solutions in Theorems 1.7 (iii) and (iv) and 1.8.

\vspace{3pt}

\subsection{Organization of the study}
The following is how this study is structured. In Section 2, we show the nonexistence of rotational axisymmetric ($-\alpha$)-homogeneous solutions without swirls for $-2\leq \alpha<0$ (Theorem 1.4 (ii)) and the rigidity of radially irrotational homogeneous solutions (Theorem 1.3) based on the equations on the sphere. In Section 3, we derive the nonautonomous 1D Dirichlet problem (1.13) for axisymmetric ($-\alpha$)-homogeneous solutions and demonstrate the existence of solutions to (1.13) by applying minimax theorems. In Section 4, we show the regularity of rotational axisymmetric ($-\alpha$)-homogeneous solutions and deduce their existence from the existence of solutions to the nonautonomous 1D Dirichlet problem ((iii) and (iv) of Theorems 1.1, 1.4, and 1.5 and Theorem 1.10). In  Appendix A, we show the nonexistence of rotational 2D reflection symmetric $(-\alpha)$-homogeneous solutions to (1.1) for $-1\leq \alpha\leq 1$ (Theorems 1.7 (i) and (ii)). In Appendix B, we show the existence of rotational $2\nicefrac{1}{2}$D reflection symmetric $(-\alpha)$-homogeneous solutions to (1.1) for $\alpha\in \mathbb{R}\backslash [-1,1]$ (Theorem B.1) and deduce Theorems 1.7 (iii) and (iv) and 1.8.

\vspace{3pt}

\subsection{Acknowledgments}
I thank In-Jee Jeong for discussing self-similar solutions to the Euler equations. This work was made possible in part by the JSPS through the Grant-in-aid for Young Scientist 20K14347 and by MEXT Promotion of Distinctive Joint Research Center Program JPMXP0619217849.

\vspace{3pt}

\section{The rigidity of homogeneous solutions}

The nonexistence of irrotational homogeneous solutions to (1.1) is based on the uniqueness of the eigenfunctions of the Laplace--Beltrami operator on the sphere. The nonexistence of rotational Beltrami homogeneous solutions is based on that of irrotational homogeneous solutions. We give a heuristic argument using the first integral condition $u\cdot \nabla \varphi=0$ when Beltrami homogeneous solutions admit the proportionality factor $\varphi$. The nonexistence of rotational axisymmetric homogeneous solutions is based on that of Beltrami homogeneous solutions and the first integral condition $u\cdot\nabla  \Pi=0$ for the Bernoulli function $\Pi$. We show the nonexistence of rotational axisymmetric homogeneous solutions without swirls based on the nonexistence of irrotational homogeneous solutions and the first integral condition $u\cdot \nabla \zeta=0$ for $\zeta=\nabla_{\mathbb{R}^{3}} \times u\cdot e_{\phi}/r$. We also demonstrate the rigidity of radially irrotational homogeneous solutions by using the geodesic property of the tangential component of solutions on the sphere and the uniqueness of the associated Legendre equation.

\subsection{Equations on the sphere}

We consider homogeneous solution's equations on the sphere to (1.1) \cite{Sverak2011}, \cite{Shv}. We use the polar coordinates 

\begin{align*}
x_1&=\rho\cos\phi\sin \theta,  \\
x_2&=\rho\sin\phi\sin \theta, \\
x_3&=\rho\cos\theta,
\end{align*}\\
and the associated orthogonal frame 
 
\begin{equation*}
e_{\rho}=\left(
\begin{array}{c}
\cos\phi\sin \theta \\
\sin\phi\sin \theta \\
\cos\theta
\end{array}
\right),\quad 
e_\theta=\left(
\begin{array}{c}
\cos\phi\cos \theta \\
\sin\phi\cos \theta \\
-\sin\theta
\end{array}
\right),\quad 
e_{\phi}=\left(
\begin{array}{c}
-\sin\phi \\
\cos\phi \\
0
\end{array}
\right).
\end{equation*}\\
The basis $e_{\theta}$ and $e_{\phi}$ are the orthogonal frame on $\mathbb{S}^{2}$. The dual basis are $\textrm{d}\theta$ and $\sin\theta\textrm{d}\phi$. For a tangential vector field $v=ae_{\theta}+be_{\phi}$ on $\mathbb{S}^{2}$, the dual 1-form is denoted by $v^{b}=ad\theta+b\sin\theta d\phi$. Conversely, for a 1-form $\beta=ad\theta+b\sin\theta d\phi$, the associated tangential vector field is denoted by $\beta^{\#}=ae_{\theta}+be_{\phi}$. We denote the $\pi/2$ counterclockwise rotation of $v=ae_{\theta}+be_{\phi}$ by $v^{\perp}=(*\ v^{b})^{\#}=-be_{\theta}+a e_{\phi}$. Here,  $*$ is the Hodge star operator. 

The gradient of a $0$-form $\psi$ is $d\psi$ for the exterior differentiation $d$. Its adjoint operator is $\delta \beta=- {*}\ d\  {*}\ \beta$ for a 1-form $\beta$. The divergence of $\beta$ is $-\delta \beta$. We define $\textrm{grad}\ \psi=(d\psi)^{\#}$ and $\textrm{div}\ v=-\delta v^{b}$ for a scholar function $\psi$ and a tangential vector field $v$. We use the operator $\nabla=e_{\theta}\partial_{\theta}+(\sin{\theta})^{-1}e_{\phi}\partial_{\phi}$ and express them as

\begin{align*}
\textrm{grad}\ \psi&=\nabla \psi=\partial_{\theta}\psi e_{\theta}+\frac{1}{\sin\theta}\partial_{\phi}\psi e_{\phi} ,\\
\textrm{div}\ v&=\nabla\cdot v=\frac{1}{\sin\theta}\left(\partial_{\theta}(a\sin\theta)+\partial_{\phi}b \right) .
\end{align*}\\
We omit the symbol $\mathbb{S}^{2}$ and use a short-hand notation $\nabla=\nabla_{\mathbb{S}^{2}}$ throughout this section. Similarly, the rotation of a 1-form $\beta$ is $*\ d \beta$. The adjoint operator of $*\ d$ is $\delta * \psi$ for a 0-form $\psi$ and $\delta=-{*}\ d\  {*}$. We set $\textrm{curl}^{*}\ \psi=-\delta *\psi=*\ d\psi$. We define $\textrm{curl}\ v=*\ dv^{b}$ and $\textrm{curl}^{*}\ \psi=(*\ d\psi)^{\#}$ for a tangential vector field $v$ and a scholar function $\psi$. We use the operator $\nabla^{\perp}=-(\sin\theta)^{-1} e_{\theta}\partial_{\phi}+e_{\phi}\partial_{\theta}$ and express them as 

\begin{align*}
\textrm{curl}\ v&=\nabla^{\perp}\cdot v=\frac{1}{\sin\theta}\left(\partial_{\theta}(b\sin\theta)-\partial_{\phi}a \right) ,\\
\textrm{curl}^{*}\ \psi&=\nabla^{\perp}\psi=-\frac{1}{\sin\theta}\partial_{\phi}\psi e_{\theta}+\partial_{\theta}\psi e_{\phi}.
\end{align*}\\
We note that $\nabla^{\perp}\psi=(\nabla\psi)^{\perp}$, $v\cdot \nabla^{\perp}=-v^{\perp}\cdot \nabla$, and $\nabla^{\perp}\cdot v^{\perp}=\nabla\cdot v$. The Laplace--Beltrami operator for a 0-form $\psi$ is $-\Delta \psi=\delta d \psi$. Its explicit form is expressed as 

\begin{align*}
\Delta \psi=\frac{1}{\sin\theta}   \partial_{\theta}(\partial_{\theta}\psi\sin\theta)+\frac{1}{\sin^{2}\theta}\partial_{\phi}^{2}\psi.
\end{align*}\\
We note that $\Delta \psi=\nabla \cdot \nabla \psi=\nabla^{\perp} \cdot \nabla^{\perp} \psi$. 

We rewrite the equations (1.1) as 

\begin{equation}
\begin{aligned}
(\nabla_{\mathbb{R}^{3}} \times u)\times u+\nabla_{\mathbb{R}^{3}}  \Pi&=0,\\
\nabla_{\mathbb{R}^{3}} \cdot u&=0,
\end{aligned}
\end{equation}\\
with the Bernoulli function $\Pi=p+|u|^{2}/2$. By multiplying $u$ by the first equation,

\begin{align}
u\cdot \nabla_{\mathbb{R}^{3}} \Pi=0.
\end{align}\\
We denote ($-\alpha$)-homogeneous solutions $u$ to (2.1) by  

\begin{align}
u=\frac{1}{\rho^{\alpha}}(v+fe_{\rho}),\quad v=ae_{\theta}+be_{\phi}, 
\end{align}\\
and the functions $a(\theta,\phi)$, $b(\theta,\phi)$, and $f(\theta,\phi)$. The rotation of $u$ in $\mathbb{R}^{3}$ is expressed as   

\begin{align}
\nabla_{\mathbb{R}^{3}} \times u=\frac{1}{\rho^{\alpha+1}}\left( (1-\alpha)v^{\perp}-\nabla^{\perp}f+(\nabla^{\perp}\cdot v) e_{\rho} \right).
\end{align}\\
Substituting (2.3) into (2.1) implies that $(v,f,p)$ satisfies the equations on $\mathbb{S}^{2}$:  

\begin{equation}
\begin{aligned}
&(1-\alpha)fv+(\nabla^{\perp}\cdot v) v^{\perp}+\nabla \left(p+\frac{1}{2}|v|^{2} \right)=0,\\
&v\cdot \nabla f=|v|^{2}+\alpha f^{2}+2\alpha p,\\
&(2-\alpha)f+\nabla \cdot v=0.
\end{aligned}
\end{equation}\\
The condition (2.2) in terms of $(v,f,p)$ on $\mathbb{S}^{2}$ is $\Pi=p+|v|^{2}/2+f^{2}/2$ and

\begin{align}
v\cdot \nabla \Pi=2\alpha f \Pi.
\end{align}

\vspace{3pt}

\begin{rem}
The identity 

\begin{align}
v\nabla v=(\nabla^{\perp}\cdot v)v^{\perp}+\nabla \frac{1}{2}|v|^{2},
\end{align}\\
holds for the covariant derivative of $v$. The tangential vector field $v$ is a $0$-homogeneous vector field in $\mathbb{R}^{3}$ and satisfies the identity \cite[Appendix]{Shv},

\begin{align*}
v\cdot \nabla_{\mathbb{R}^{3}}v&=\frac{1}{\rho} (v\nabla v-|v|^{2}e_{\rho} ),\\
\nabla_{\mathbb{R}^{3}} \times v&=\frac{1}{\rho}\left(v^{\perp}+(\nabla^{\perp}\cdot v)e_{\rho} \right).
\end{align*}\\
The identity $v\cdot \nabla_{\mathbb{R}^{3}}v=(\nabla_{\mathbb{R}^{3}} \times v)\times v+\nabla_{\mathbb{R}^{3}} |v|^{2}/2$ implies (2.7).
\end{rem}

\vspace{3pt}

\subsection{Irrotational solutions}
The nonexistence of irrotational homogeneous solutions to (1.1) is based on the fact that the eigenvalues of the Laplace--Beltrami operator $-\Delta$ are $ n(n+1)$ for $n\in \mathbb{N}\cup \{0\}$. The multiplicity of each eigenvalue is $2n+1$ with one axisymmetric eigenfunction (zonal harmonic) and $2n$ nonaxisymmetric eigenfunctions.

\vspace{3pt}

\begin{thm}(\hspace{-1pt}\cite[Proposition 3.1]{Shv})
Irrotational $(-\alpha)$-homogeneous solutions $(u,p)\in C^{1}(\mathbb{R}^{3}\backslash \{0\})$ exist if and only if $\alpha\in \mathbb{Z}\backslash \{1\}$. They are given by spherical harmonics $f$ satisfying $-\Delta f=(\alpha-2)(\alpha-1)f$ and 

\begin{align}
v=\frac{1}{1-\alpha}\nabla f,\quad p=-\frac{1}{2(1-\alpha)^{2}}|\nabla f|^{2}-\frac{1}{2}f^{2}.
\end{align}\\
The pressure $p$ is unique up to constant for $\alpha=0$. In particular, no solutions exist for $\alpha=1$. For each $\alpha\in \mathbb{Z}\backslash \{1\}$, linearly independent one axisymmetric solution without swirls and $2|\alpha -3/2|-1$ nonaxisymmetric solutions exist.  
\end{thm}

\begin{proof}
For $\alpha\neq 1$, $\nabla_{\mathbb{R}^{3}}\times u=0$ and the identity (2.4) imply that $(1-\alpha)v^{\perp}=\nabla^{\perp}f$. Thus $(2.8)_1$ holds by taking $\perp$. Substituting $(2.8)_1$ into $(2.5)_3$ implies the equation for $f$. Since the Bernoulli function $\Pi=p+|v|^{2}/2+f^{2}/2$ is zero for $\alpha\neq 0$ and constant for $\alpha=0$, the pressure representation $(2.8)_2$ follows.  

For $\alpha=1$, $\nabla_{\mathbb{R}^{3}}\times u=0$ and (2.4) imply that $f$ is constant and $\nabla^{\perp}\cdot v=0$. By integrating $(2.5)_3$ on $\mathbb{S}^{2}$, $f=0$ and $\nabla\cdot v=0$. Thus $v$ is harmonic on $\mathbb{S}^{2}$ and hence $v=0$. The pressure $p$ vanishes by $(2.5)_2$. 

For integers $\alpha\geq 2$ and $n=\alpha-2$, one zonal harmonic and $2n$ nonaxisymmetric spherical harmonics exist. Since $(\alpha-2)(\alpha-1)$ is symmetric for $\alpha=3/2$, for each $\alpha\in \mathbb{Z}\backslash \{1\}$ there exists one axisymmetric irrotational ($-\alpha$)-homogeneous solution and $2|\alpha -3/2|-1$ nonaxisymmetric solutions. The axisymmetric solution is without swirls by $(2.8)_1$. 
\end{proof}

\vspace{3pt}

\begin{rem}
The spherical harmonics of $-\Delta f=(\alpha-2)(\alpha-1)f$ for $\alpha\in \mathbb{Z}\backslash \{1\}$ are expressed by the associated Legendre functions. By the Fourier series expansion of $f=f_{n}(\theta,\phi)$ for $n=\alpha-2$,

\begin{align}
f_n(\theta,\phi)=\sum_{m=-\infty}^{\infty}f_{n,m}(\theta)e^{im\phi},\quad 
f_{n,m}(\theta)=\frac{1}{2\pi}\int_{0}^{2\pi} f_n(\theta,\phi)e^{-im\phi}\textrm{d}\phi.
\end{align}\\
The Fourier coefficients $f_{n,m}(\theta)$ satisfy the associated Legendre equation,

\begin{align*}
-f_{n,m}''-\cot{\theta}f_{n,m}'+\frac{m^{2}}{\sin^{2}\theta}f_{n,m}=n(n+1)f_{n,m}.
\end{align*}\\
For $t=\cos\theta$, $g_{n,m}(t)=f_{n,m}(\theta)$ satisfy 

\begin{align*}
-(1-t^{2})g_{n,m}''+2t g_{n,m}'+\frac{m^{2}}{1-t^{2}}g_{n,m}=n(n+1)g_{n,m}.
\end{align*}\\
Solutions of this equation for $n\in \mathbb{N}\cup \{0\}$ are given by the associated Legendre functions, e.g., \cite[12.5]{AW01},

\begin{align*}
P_{n}^{m}(t)&=\frac{(1-t^{2})^{\frac{m}{2}}}{2^{n}n!}\frac{\dd^{n+m} }{\dd t^{n+m}}(t^{2}-1)^{n},\\
P_{n}^{-m}(t)&=(-1)^{m}\frac{(n-m)!}{(n+m)!}P_{n}^{m}(t),\quad m\in \mathbb{N}\cup \{0\}.
\end{align*}\\
The function $P_{n}=P^{0}_{n}$ is the Legendre function. The function $w_{n+1}$ in (1.9$)_2$ is a solution to (1.13) for $c_1=c_2=0$ and $\beta=n$ since $(1-t^{2})P_{n}'=-n(n+1)\int_{-1}^{t}P_{n}\dd s$ by integrating the above equation for $P_{n}$.
\end{rem}

\vspace{3pt}

\subsection{The case $\alpha=1$}

For $\alpha=1$, the quantity $p+|v|^{2}/2$ is a first integral of $v$ and implies the nonexistence of homogeneous solutions, cf. \cite[Theorem 3]{TX98}.

\vspace{3pt}

\begin{thm}(\hspace{-1pt}\cite[Proposition 2.1]{Shv})
There exist no $(-1)$-homogeneous solutions $(u,p)\in C^{1}(\mathbb{R}^{3}\backslash \{0\})$ to (1.1).
\end{thm}

\vspace{3pt}

\begin{proof}
The equation $(2.5)_1$ implies that 

\begin{align*}
v\cdot \nabla \left(p+\frac{1}{2}|v|^{2}\right)=0.
\end{align*}\\
Thus $v\cdot \nabla h=0$ for $h=2p+|v|^{2}$. By $(2.5)_3$ and $(2.5)_2$, $\nabla \cdot v=-f$, $v\cdot \nabla f=f^{2}+h$, and 

\begin{align*}
\nabla\cdot (vfh)=(\nabla \cdot v)fh+(v\cdot \nabla f) h+vf\cdot \nabla  h=h^{2}.
\end{align*}\\
Integrating both sides on $\mathbb{S}^{2}$ implies that $h=0$. By $(2.5)_1$, $\nabla^{\perp}\cdot v=0$ and 

\begin{align*}
\nabla\cdot (vf^{3})=(\nabla \cdot v)f^{3}+3(v\cdot \nabla f)f^{2}=2f^{4}.
\end{align*}\\
Integrating both sides on $\mathbb{S}^{2}$ implies that $f=0$ and $\nabla\cdot v=0$. Thus $v$ is harmonic on $\mathbb{S}^{2}$ and hence $v=0$. 
\end{proof}

\vspace{3pt}

\subsection{Beltrami solutions for $\alpha\leq 2$}

The nonexistence of rotational ($-\alpha$)-homogeneous Beltrami solutions $(\nabla_{\mathbb{R}^{3}} \times u)\times u=0$ for $\alpha\leq 2$ (Theorem 1.1 (i) and (ii)) is based on the nonexistence of irrotational ($-\alpha$)-homogeneous solutions for $\alpha\notin \mathbb{Z}\backslash \{1\}$ (Theorem 2.2) and the nonexistence of ($-1$)-homogeneous solutions (Theorem 2.4). 

\vspace{3pt}

\begin{thm}(\hspace{-1pt}\cite[Proposition 3.2]{Shv})
The following holds for the rotational Beltrami ($-\alpha$)-homogeneous solutions to (1.1):

\noindent
(i) For $1\leq \alpha\leq 2$, no solutions $(u,p)\in C^{1}(\mathbb{R}^{3}\backslash \{0\})$ exist.\\
(ii) For $\alpha<1$, no solutions $(u,p)\in C^{2}(\mathbb{R}^{3}\backslash \{0\})$ exist.
\end{thm}

\vspace{3pt}

The Bernoulli function $\Pi=p+|u|^{2}/2$ of Beltrami ($-\alpha$)-homogeneous solutions vanishes for $\alpha\neq 0$ and is constant for $\alpha=0$. Thus $\alpha\Pi= 0$. By $(2.5)_2$ and $(2.5)_3$, 

\begin{align}
\nabla \cdot (vf)=(\alpha-2)f^{2}+(1-\alpha)|v|^{2}.
\end{align}\\
The case $\alpha=1$ is excluded by Theorem 2.4. For $1<\alpha< 2$, integrating both sides on $\mathbb{S}^{2}$ implies that $f=0$ and $v=0$. For $\alpha=2$, integrating (2.10) implies that $v=0$. The equations $(2.5)_1$ and $(2.5)_2$ imply that $p$ and $f$ are constants. Thus $u$ is irrotational by (2.4).

The case $\alpha<1$ is more involved. A simpler case is when $u$ admits the proportionality factor $\varphi\in C^{1}(\mathbb{R}^{3}\backslash \{0\})$ and forms 

\begin{align}
\nabla_{\mathbb{R}^{3}} \times u=\varphi u.
\end{align}\\
Taking the divergence in $\mathbb{R}^{3}$ implies that $\varphi$ is a first integral of $u$, i.e., 

\begin{align*}
u\cdot \nabla_{\mathbb{R}^{3}} \varphi=0.
\end{align*}\\
The factor $\varphi$ is ($-1$)-homogeneous and hence  

\begin{align}
v\cdot \nabla \varphi=f\varphi.
\end{align}\\
The equations (2.10) and (2.12) imply that for nonnegative integers $n$, 

\begin{align*}
\nabla \cdot (vf\varphi^{n})=(\alpha-2+n)f^{2}\varphi^{n}+(1-\alpha)|v|^{2}\varphi^{n}.
\end{align*}\\
By taking an even number $n=2k>2-\alpha$ and integrating both sides on $\mathbb{S}^{2}$, $f\varphi^{k}=0$ and $v \varphi^{k}=0$. This implies that $f\varphi=0$ and $v\varphi=0$, and hence $u$ is irrotational by (2.11).

The work by Shvydkoy \cite{Shv} demonstrated the same conclusion without using the proportionality factor $\varphi$. Instead of (2.12), one can use the equation for $\omega=\nabla^{\perp}\cdot v$, 

\begin{align}
(vf)\cdot\nabla \omega=(f^{2}+(1-\alpha)|v|^{2}) \omega.
\end{align}\\
This equation follows from the radial component of the vorticity equations $u\cdot \nabla_{\mathbb{R}^{3}} (\nabla_{\mathbb{R}^{3}}\times u)=(\nabla_{\mathbb{R}^{3}}\times u)\cdot \nabla_{\mathbb{R}^{3}} u$ (The equation (2.13) is a different form of (2.12) if $u$ admits the proportionality factor $\varphi=\omega/f$). Since $\Pi$ is constant, the equation $(2.5)_1$ forms

\begin{align*}
(1-\alpha)fv+\omega v^{\perp}-\nabla f f=0.
\end{align*}\\
By  taking rotation and using $(2.5)_3$, $(1-\alpha)\nabla^{\perp}(fv)+(\alpha-2)f\omega+v\cdot \nabla \omega=0$. By multiplying $v^{\perp}$ by the above equation, $\omega|v|^{2}=v^{\perp}\cdot \nabla f f$ and $\nabla^{\perp}\cdot (fv)=(|v|^{2}/f+f)\omega$. Thus (2.13) follows from (2.5).

Combining (2.13) and (2.10) imply for nonnegative integers $n$, 

\begin{align}
\nabla \cdot (vf \omega^{n})=(\alpha-2+n)f^{2}\omega^{n}+(n+1)(1-\alpha)|v|^{2}\omega^{n}.
\end{align}

\vspace{3pt}

\begin{proof}
For $\alpha<1$, by taking an even number $n=2k>2-\alpha$ and integrating both sides of (2.14) on $\mathbb{S}^{2}$, $f\omega^{k}=0$ and $v\omega^{k}=0$. Thus 

\begin{align*}
f\omega=0,\quad v\omega=0.
\end{align*}\\
The equation $(2.5)_1$ forms $((1-\alpha)v-\nabla f)f=-\omega v^{\perp}$ and hence  

\begin{align*}
\left((1-\alpha)v^{\perp}-\nabla^{\perp}f\right)f=\omega v=0.
\end{align*}\\
If $f\neq 0$, $(1-\alpha)v^{\perp}-\nabla^{\perp}f=0$ and $\omega=0$. If $f=0$, $v=0$ by (2.10). In particular, $(1-\alpha)v^{\perp}-\nabla^{\perp}f=0$ and $\omega=0$. Thus $u$ is irrotational by (2.4). 
\end{proof}

\vspace{3pt}

\begin{proof}[Proof of Theorem 1.1 (i) and (ii)]
The results follow from Theorem 2.5. 
\end{proof}

\vspace{3pt}

\subsection{Axisymmetric solutions for $0\leq \alpha\leq 2$}

The nonexistence of rotational axisymmetric ($-\alpha$)-homogeneous solutions for $0\leq \alpha\leq 2$ (Theorem 1.5 (i) and (ii)) is based on the nonexistence of rotational Beltrami ($-\alpha$)-homogeneous solutions (Theorem 2.5) and the first integral condition (2.6).

\vspace{3pt}

\begin{thm}(\hspace{-1pt}\cite[Propositions 5.1 and 5.2]{Shv})
The following holds for rotational axisymmetric ($-\alpha$)-homogeneous solutions to (1.1):

\noindent
(i) For $1\leq \alpha\leq 2$, no solutions $u\in C^{1}(\mathbb{R}^{3}\backslash \{0\})$ exist. \\
(ii) For $0\leq \alpha<1$, no solutions $u\in C^{2}(\mathbb{R}^{3}\backslash \{0\})$ exist. 
\end{thm}

\vspace{3pt}

\begin{proof}
We first consider the case $0<\alpha<2$. The equations $(2.5)_3$ and (2.6) for axisymmetric solutions are expressed as 

\begin{align*}
&(2-\alpha)f+\frac{1}{\sin\theta}\partial_{\theta}(a\sin\theta)=0,\\
&a\partial_{\theta}\Pi=2\alpha f\Pi.
\end{align*}\\
We show that $a\Pi=0$ on $(0,\pi)$ by a contradiction argument. Suppose that $a\Pi\neq 0$ on some interval $J\subset (0,\pi)$. Then, by eliminating $f$ from the above equations and integrating an equation for $\Pi$ and $a$,  

\begin{align*}
|\Pi|^{2-\alpha}|a|^{2\alpha}\sin^{2\alpha}{\theta}=C,
\end{align*}\\
on $J$ for some constant $C$. If $C\neq 0$, $J$ is extendable to $(0,\pi)$. However, the condition $0<\alpha<2$ implies that the left-hand side vanishes as $\theta\to0$. Thus $C=0$. This contradicts the assumption $a\Pi\neq 0$ on $J$. Hence $a\Pi=0$ on $(0,\pi)$.

We show that $\Pi=0$ on $(0,\pi)$ by a contradiction argument. Suppose that $\Pi\neq 0$ on some interval $J\subset (0,\pi)$. Then, $a=0$ and $v=be_{\phi}$. By $(2.5)_3$, (2.5$)_1$, and $\nabla^{\perp}\cdot v=\partial_{\theta}b+b\cot{\theta}$, $f=0$ and $b^{2}\cot\theta-\partial_{\theta}p=0$. By $(2.5)_2$, $b^{2}+2\alpha p=0$. Thus $b$ and $p$ satisfy 

\begin{align*}
b^{2}\cot\theta-\partial_{\theta}p&=0,\\
b^{2}+2\alpha p&=0,
\end{align*}\\
on $J$. Since $0\neq \Pi=p+|b|^{2}/2=(1-\alpha)p$, by eliminating $b$ from the above equations and integrating an equation for $p$,

\begin{align*}
|p|\sin^{2\alpha}{\theta}=C,
\end{align*}\\
on $J$ for some constant $C$. If $C\neq 0$, $J$ is extendable to $(0,\pi)$. However, the left-hand side vanishes as $\theta\to0$. Thus $C=0$ and $\Pi=(1-\alpha)p=0$ on $J$. This contradicts $\Pi \neq 0$ on $J$. Thus $\Pi=0$ on $(0,\pi)$ and $u$ is a Beltrami flow. We apply Theorem 2.5 and conclude that no rotational solutions exist for $0<\alpha<2$. 

For $\alpha=2$, the equation $(2.5)_3$ implies that 

\begin{align*}
0=\nabla \cdot v=\frac{1}{\sin\theta}\partial_{\theta}(a\sin\theta)
\end{align*}\\
Thus $a\sin\theta=C$ for some constant $C$. Since $a$ is bounded, $C=0$ and hence $a=0$. By the equations $(2.5)_1$ and $(2.5)_2$ for $v=be_{\phi}$

\begin{align*}
fb&=0,\\
b^{2}\cot\theta-\partial_{\theta}p&=0,\\
b^{2}+2f^{2}+4p&=0.
\end{align*}\\
If $f= 0$ on $(0,\pi)$, the second and third equations are a system for $b$ and $p$. If $p\neq0$ on some interval $J\subset (0,\pi)$, 

\begin{align*}
|p|\sin^{4}\theta=C.
\end{align*}\\
If $C\neq 0$, $J$ is extendable to $(0,\pi)$. However, the left-hand side vanishes as $\theta\to0$. Thus $C=0$. This contradicts $p\neq 0$ on $J$. Thus $p= 0$ and $b=0$. We conclude that no solutions exist.

If $f\neq 0$ on some interval $J\subset (0,\pi)$, the above equations imply that $b=0$, and $p$ and $f$ are constants on $J$. Thus $J$ is extendable to $(0,\pi)$ and $f$ is constant on $(0,\pi)$. Since $b=0$ on $(0,\pi)$, $v=0$ and $u$ is irrotational.

For $\alpha=0$, the equation (2.6) implies that $a\partial_{\theta}\Pi=0$. Suppose that $\partial_{\theta}\Pi\neq 0$ on some interval $J\subset (0,\pi)$. Then $a=0$ and the equation $(2.5)_2$ implies that $v=be_{\phi}=0$. By $(2.5)_3$ and $(2.5)_1$, $f=0$ and $p$ is constant. Thus $\partial_{\theta}\Pi= 0$ on $J$. This is a contradiction. Thus $\partial_{\theta}\Pi=0$ in $(0,\pi)$ and $u$ is a Beltrami flow. We apply Theorem 2.5 and conclude that no rotational solutions exist.
\end{proof}

\vspace{3pt}

\begin{proof}[Proof of Theorem 1.5 (i) and (ii)]
The results follow from Theorem 2.6. 
\end{proof}

\vspace{3pt}

\subsection{Axisymmetric solutions without swirls for $-2\leq \alpha< 0$}

We now demonstrate the nonexistence of rotational axisymmetric ($-\alpha$)-homogeneous solutions without swirls for $-2\leq \alpha<0$ (Theorem 1.4 (ii)). 

\vspace{3pt}

\begin{thm}
The following holds for rotational axisymmetric ($-\alpha$)-homogeneous solutions without swirls to (1.1):

\noindent
(i) For $1\leq \alpha\leq 2$, no solutions $(u,p)\in C^{1}(\mathbb{R}^{3}\backslash \{0\})$ exist. \\
(ii) For $-2< \alpha<1$, no solutions $(u,p)\in C^{2}(\mathbb{R}^{3}\backslash \{0\})$ exist. For $\alpha=-2$, no solutions $(u,p)\in C^{2}(\mathbb{R}^{3}\backslash \{0\})$ exist provided that $\nabla_{\mathbb{R}^{3}} \times u\cdot e_{\phi}/r$ vanishes on the $z$-axis. 
\end{thm}

\vspace{3pt}

We show that axisymmetric $(-\alpha)$-homogeneous solutions without swirls $u\in C^{2}(\mathbb{R}^{3}\backslash \{0\})$ for $-2\leq \alpha <0$ are irrotational. Their vorticity in the cylindrical coordinates $(r,\phi,z)$ is expressed as 

\begin{align*}
\nabla_{\mathbb{R}^{3}}\times u=\omega^{\phi}(z,r)e_{\phi}.
\end{align*}\\
By taking the rotation in $\mathbb{R}^{3}$ to (2.1), 

\begin{align*}
u\cdot \nabla_{\mathbb{R}^{3}} (\nabla_{\mathbb{R}^{3}} \times u)=(\nabla_{\mathbb{R}^{3}} \times u) \cdot \nabla_{\mathbb{R}^{3}} u.
\end{align*}\\
Multiplying $e_{\phi}$ by both sides,

\begin{align*}
u\cdot \nabla_{\mathbb{R}^{3}} \omega^{\phi}=\frac{u^{r}}{r}\omega^{\phi}.
\end{align*}\\
By dividing both sides by $r$, 

\begin{align*}
u\cdot \nabla_{\mathbb{R}^{3}}\zeta=0,\quad \zeta=\frac{\omega^{\phi}(z,r)}{r}.
\end{align*}\\
Thus $\zeta$ is a first integral of $u$ \cite{UI}. Since $\zeta$ is ($-\alpha-2$)-homogeneous, this condition implies that  

\begin{align}
v\cdot \nabla \zeta=(\alpha+2)f\zeta.
\end{align}

\vspace{3pt}

\begin{proof}[Proof]
It suffices to show (ii) for $-2\leq \alpha<0$ by Theorem 2.6. For axisymmetric $(-\alpha)$-homogeneous solutions without swirls $u\in C^{2}(\mathbb{R}^{3}\backslash \{0\})$, $v=ae_{\theta}$ and 

\begin{align}
\nabla_{\mathbb{R}^{3}} \times u=\omega^{\phi}(z,r)e_{\phi}=
\frac{1}{\rho^{\alpha+1}}\left( (1-\alpha)a-\partial_{\theta}f \right)e_{\phi}.
\end{align}\\
Since $\nabla_{\mathbb{R}^{3}}\times u$ is continuously differentiable on $\mathbb{S}^{2}$, $(1-\alpha)a-\partial_{\theta}f\in C^{1}[0,\pi]$ vanishes at $\theta=0$ and $\pi$. We set 

\begin{align*}
\zeta=\frac{\omega^{\phi}(z,r) }{r}
=\frac{1}{\rho^{\alpha+2}}\left( \frac{(1-\alpha)a-\partial_{\theta}f}{\sin\theta}\right).
\end{align*}\\
The function $\zeta$ on $\mathbb{S}^{2}$ is continuous on $[0,\pi]$ and continuously differentiable in $(0,\pi)$. 

We first consider the case $-2<\alpha<0$. We show that $a\zeta=0$ on $(0,\pi)$. Suppose that $a\zeta\neq 0$ on some interval $J\subset (0,\pi)$. Then, by eliminating $f$ from $(2.5)_3$ and (2.15), and integrating an equation for $a$ and $\zeta$, 

\begin{align*}
|\zeta|^{2-\alpha}|a|^{\alpha+2}\sin^{\alpha+2}\theta=C,
\end{align*}\\
for some constant $C$ on $J$. If $C\neq 0$, $J$ is extendable to $(0,\pi)$. However, the condition $-2<\alpha<0$ implies that the left-hand side vanishes as $\theta\to0$. Thus $C=0$. This contradicts $a\zeta\neq 0$ on $J$. Thus $a\zeta=0$ on $(0,\pi)$.

We show that $\zeta=0$ on $(0,\pi)$. Suppose that $\zeta\neq0$ on some interval $J\subset (0,\pi)$. Then $a=0$ on $J$. The equation (2.5$)_3$ implies that $v=ae_\phi=0$ and $f=0$. Hence $\zeta=0$ on $J$. This is a contradiction and hence $\zeta=0$ on $(0,\pi)$. Thus $u$ is irrotational.

For $\alpha=-2$,  the equation (2.15) implies that $a\partial_{\theta}\zeta=0$. If $\partial_{\theta}\zeta\neq 0$ on some interval $J\subset (0,\pi)$, $v=ae_{\theta}=0$ and the equation $(2.5)_3$ implies that $f=0$. Hence $\zeta=0$ on $J$. This is a contradiction. Thus $\partial_{\theta}\zeta= 0$ on $(0,\pi)$. By the assumption, $\zeta$ vanishes at $\theta=0$ and $\pi$. Thus $\zeta=0$ on $(0,\pi)$ and $u$ is irrotational. 
\end{proof}

\vspace{3pt}

\begin{proof}[Proof of Theorem 1.4 (i) and (ii)]
The results follow from Theorem 2.7.
\end{proof}

\vspace{3pt}

\subsection{Radially irrotational solutions}

We show that all radially irrotational ($-\alpha$)-homogeneous solutions to (1.1) for $\alpha\in \mathbb{R}$ with a nonconstant Bernoulli function are axisymmetric without swirls (Theorem 1.3). Radially irrotational ($-\alpha$)-homogeneous solutions $u=(v+fe_{\rho})/\rho^{\alpha}$ satisfy

\begin{align*}
\nabla_{\mathbb{R}^{3}}\times u\cdot x=0.
\end{align*}\\
The rotation expression (2.4) implies that $\nabla^{\perp}\cdot v=0$ and

\begin{align}
\nabla_{\mathbb{R}^{3}} \times u=\frac{1}{\rho^{\alpha+1}}\left( (1-\alpha)v^{\perp}-\nabla^{\perp}f \right).
\end{align}\\
The equations (2.5) form

\begin{equation}
\begin{aligned}
&(1-\alpha)fv+\nabla \left(p+\frac{1}{2}|v|^{2} \right)=0,\\
&v\cdot \nabla f=|v|^{2}+\alpha f^{2}+2\alpha p,\\
&(2-\alpha)f+\nabla \cdot v=0,\\
&\nabla^{\perp}\cdot v=0.
\end{aligned}
\end{equation}\\
Taking the rotation on $\mathbb{S}^{2}$ to the first equation implies that $v\cdot \nabla^{\perp} f=0$. Multiplying $v^{\perp}$ by the first equation yields that $v^{\perp}\cdot \nabla (p+|v|^{2}/2)=0$. Thus 

\begin{equation}
\begin{aligned}
v^{\perp}\cdot \nabla f&=0,\\
v^{\perp}\cdot \nabla \Pi&=0.
\end{aligned}
\end{equation}

\vspace{3pt}

A key fact to demonstrate the rigidity of radially irrotational homogeneous solutions is that integral curves of $v$ on $\mathbb{S}^{2}$ are geodesic on $\{\Pi\neq 0\}$ (Lemma 2.10). We first exclude the case $\alpha=0$.

\vspace{3pt}

\begin{prop}
There exist no radially irrotational $0$-homogeneous solutions $(u,p)\in C^{2}(\mathbb{R}^{3}\backslash \{0\})$ with a nonconstant Bernoulli function to (1.1).
\end{prop}

\vspace{3pt}

\begin{proof}
By the irrotational condition $(2.18)_4$, there exists a potential function $\psi$ such that $v=\nabla \psi$. The equation $(2.18)_2$ for $\alpha=0$ implies that

\begin{align*}
\nabla \psi\cdot \nabla (f-\psi)=0.
\end{align*}\\
The condition $(2.19)_1$ implies that $\nabla^{\perp} \psi\cdot \nabla (f-\psi)=0$. Thus $\nabla (f-\psi)=0$. By $v^{\perp}=\nabla^{\perp}\psi=\nabla^{\perp}f$ and (2.17) for $\alpha=0$, $u$ is irrotational. Since $\Pi$ is not constant, no solutions exist by the equation $(2.1)_1$.
\end{proof}

\vspace{3pt}

\begin{prop}
Let $(u,p)\in C^{2}(\mathbb{R}^{3}\backslash \{0\})$ be a radially irrotational $(-\alpha)$-homogeneous solution for $\alpha\in \mathbb{R}\backslash \{0\}$ with a nonconstant Bernoulli function $\Pi$ to (1.1). Then

\begin{equation}
\begin{aligned}
&v\neq 0,\\
&v^{\perp}\cdot \nabla |v|^{2}=0,
\end{aligned}
\end{equation}
on the set

\begin{align*}
\{\Pi\neq 0\}=\left\{e_{\rho}(\theta,\phi)\in \mathbb{S}^{2}\ \middle|\ \Pi(\theta,\phi)\neq 0,\ 0\leq \theta\leq \pi,\ 0\leq \phi\leq 2\pi\ \right\}\neq \emptyset.
\end{align*}
\end{prop}

\vspace{3pt}

\begin{proof}
We write (2.18$)_2$ as $v\cdot \nabla f=(1-\alpha)|v|^{2}+2\alpha \Pi$ and observe that $v$ does not vanish on the set $\{\Pi\neq 0\}$. Thus $(2.20)_1$ holds. We take a function $\psi$ such that $v=\nabla \psi$. We show that there exists a $C^{1}$-function $F$ such that $f$ is locally expressed as 

\begin{align}
F(\psi(\theta,\phi))=f(\theta,\phi).
\end{align}\\
This equation and $(2.18)_2$ imply 

\begin{align*}
2\alpha \Pi=v\cdot (\nabla f-(1-\alpha)v)=|v|^{2}\left(F'(\psi)-(1-\alpha) \right).
\end{align*}\\
By $\alpha\neq 0$ and $v\neq 0$ on $\{\Pi\neq 0\}$, $F'(\psi)-(1-\alpha)\neq 0$. Since $\Pi$ and $F'$ are constants on level sets of $\psi$ by $(2.19)_2$, so is $|v|^{2}$. Thus $(2.20)_2$ holds.

We demonstrate (2.21). For the function $\psi(\theta,\phi)$, we denote its 0-homogeneous extension to $\mathbb{R}^{3}\backslash \{0\}$ by 

\begin{align*}
\tilde{\psi}(x)=\psi(\theta,\phi),\quad x=\rho e_{\rho}(\theta,\phi).
\end{align*}\\
We consider a level set,

\begin{align*}
\left\{\tilde{\psi}=C,\ \tilde{\Pi}\neq 0\right\}=\left\{x=e_{\rho}(\theta,\phi)\in \mathbb{S}^{2}\ |\   \tilde{\psi}(x)=C,\ \tilde{\Pi}(x)\neq 0\ \right\},
\end{align*}\\
for a constant $C$. By $\nabla \tilde{\psi}=\nabla \psi\neq 0$ on $\{\tilde{\psi}=C,\ \tilde{\Pi}\neq 0\}$ and the implicit function theorem, there exists a function $\phi(\theta)$ and an open interval $J$ such that the level set $\{\tilde{\psi}=C,\ \tilde{\Pi}\neq 0\}$ is locally expressed by a curve $\{e_{\rho}(\theta,\phi(\theta))\ |\ \theta\in J\ \}$. We take an integral curve $x(t,\theta)$ of $\nabla \tilde{\psi}/| \nabla \tilde{\psi}|^{2}$ starting from the point $x(0,\theta)=e_{\rho}(\theta,\phi(\theta))$. Since $\nabla \tilde{\psi}/| \nabla \tilde{\psi}|^{2}$ is a $C^{2}$-function, $x(t,\theta)$ is a $C^{1}$-function for $(t,\theta)$. By $\partial_t x\cdot \nabla \tilde{\psi}=1$,

\begin{align*}
\tilde{\psi}(x(t,\theta) )=C+t.
\end{align*}\\
By differentiating both sides by $\theta$, $\partial_{\theta}x\cdot \nabla \tilde{\psi}=0$. Thus $\partial_{\theta}x$ is parallel to $\nabla^{\perp} \tilde{\psi}$. Since $\nabla^{\perp} \tilde{\psi}\cdot \nabla \tilde{f}=0$ by $(2.19)_1$, 

\begin{align*}
0=\partial_{\theta}x(t,\theta) \cdot \nabla \tilde{f}(x(t,\theta))=\partial_{\theta} \tilde{f}(x(t,\theta)).
\end{align*}\\
Thus $\tilde{f}(x(t,\theta))$ is independent of $\theta$. We denote $\tilde{f}(x(t,\theta))$ by $\tilde{f}(x(t,\theta))=F(C+t)$ and a function $F$. By using the inverse function $t(x)$ of $(t,\theta)\longmapsto x(t,\theta)$, $\tilde{f}(x)=F(C+t(x))=F(\tilde{\psi}(x))$. Thus (2.21) holds.
\end{proof}

\vspace{3pt}

\begin{lem}
Integral curves of $v$ on $\mathbb{S}^{2}$ are geodesics on $\{\Pi\neq 0\}$.
\end{lem}

\vspace{3pt}

\begin{proof}
The property (2.20) implies that 

\begin{align*}
\nabla |v|^{2}=(v\cdot \nabla |v|^{2})\frac{v}{|v|^{2}},
\end{align*}\\
on $\{\Pi\neq 0\}$. By the identity (2.7), 

\begin{align*}
v\nabla v=(\nabla^{\perp}\cdot v)v^{\perp}+\nabla \frac{1}{2}|v|^{2}=\frac{v\cdot \nabla |v|^{2}}{2|v|^{2}}v.
\end{align*}\\
Since the covariant derivative $v\nabla v$ is the tangential component of the second derivative $y''_{\textrm{tan}}(t)$ for the integral curve $y(t)$ of $v$ on $\mathbb{S}^{2}$, 

\begin{align*}
y''_{\textrm{tan}}(t)=c(t)y'(t),\quad
c(t)=\left(\frac{v\cdot \nabla |v|^{2}}{2|v|^{2}}\right)(y(t)).
\end{align*}\\
The curve $z(t)$ defined by

\begin{align*}
z'(t)=y'(t)\exp\left(-\int_{0}^{t}c(s)ds\right),
\end{align*}\\
satisfies $z''_{\textrm{tan}}(t)=0$. Thus $z(t)$ is geodesic and so is $y(t)$.
\end{proof}

\vspace{3pt}

In the sequel, we say that $(v,f,p)$ is axisymmetric without swirls if $v=ae_{\theta}$ and $(a,f,p)$ is independent of $\phi$. We show that the geodesic property of integral curves of $v$ on $\{\Pi\neq 0\}$ implies that $(v,f,p)$ is axisymmetric without swirls in the zonal region 

\begin{align}
D(\theta_1,\theta_2)=\left\{e_{\rho}(\theta,\phi)\in \mathbb{S}^{2}\ \middle|\ 0\leq \phi\leq 2\pi,\ \theta_1<\theta<\theta_2\right\}\subset \{\Pi\neq 0\},
\end{align}\\
for some $0<\theta_1<\theta_2<\pi$. 

\vspace{3pt}

\begin{prop}
There exists $\theta_1,\theta_2\in (0,\pi)$ such that up to rotation, $(v,f,p)$ is axisymmetric without swirls in $D(\theta_1,\theta_2)\subset \{\Pi\neq 0\}$.
\end{prop}

\vspace{3pt}

\begin{proof}
We take an open set $D_0\subset \mathbb{S}^{2}$ such that $D_0\subset \{\Pi\neq 0\}$. By Lemma 2.10, integral curves of $v$ are geodesics on $D_0$. We may assume that $v=ae_{\theta}$ by rotation and  

\begin{align*}
D_0=\left\{e_{\rho}(\theta,\phi)\in \mathbb{S}^{2}\ \middle|\ 0<\phi<\phi_1,\ \theta_1<\theta<\theta_2\right\}
\end{align*}\\
for some $0<\phi_1<2\pi$ and $0<\theta_1<\theta_2<\pi$. The conditions (2.19) and $(2.20)_2$ imply that  

\begin{align*}
\partial_{\phi}f=0,\quad \partial_{\phi}\Pi=0,\quad  \partial_{\phi}a=0.
\end{align*}\\
Thus $(v,f,p)$ is axisymmetric without swirls on $D_0$. For each $\theta\in (\theta_1,\theta_2)$, $\Pi$ is nonzero constant on the circle of the latitude

\begin{align*}
\left\{e_{\rho}(\theta,\phi)\in \mathbb{S}^{2}\ \middle|\ 0\leq \phi\leq 2\pi\ \right\}.
\end{align*}\\
Thus the set $\{\Pi\neq 0\}$ includes the zonal region $D(\theta_1,\theta_2)$.
\end{proof}

\vspace{3pt}

\begin{prop}
Suppose that $(v,f,p)$ is axisymmetric without swirls in $D(\theta_1,\theta_2)$. If there exists $\theta_3\in (\theta_2,\pi)$ such that $\Pi=0$ in $D(\theta_2,\theta_3)$. Then, $(v,f,p)$ is axisymmetric without swirls in $D(\theta_1,\theta_3)$.
\end{prop}

\vspace{3pt}

\begin{proof}
By $(2.18)_1$ and $\Pi=0$ on $D(\theta_2,\theta_3)$, $(1-\alpha)vf=\nabla ff$. If $f\neq  0$, $(1-\alpha)v=\nabla f$. If $f=0$, $(2.18)_2$ and $\Pi=0$ imply that  $v=0$. Thus,

\begin{align*}
(1-\alpha)v=\nabla f\quad \textrm{in}\ D(\theta_2,\theta_3). 
\end{align*}\\
By $(2.18)_3$, 
 
\begin{align*}
-\Delta f=(\alpha-2)(\alpha-1)f\quad \textrm{in}\ D(\theta_2,\theta_3).
\end{align*}\\
Since $(v,f,p)$ is axisymmetric without swirls in $D(\theta_1,\theta_2)$, 

\begin{align*}
\textrm{$f$ and $\partial_{\theta}f$ are constants on $\{\theta=\theta_2\}$.}
\end{align*}\\
By the Fourier series expansion (2.9) for $f=f_{\alpha-2}$,

\begin{align*}
f_{\alpha-2}(\theta,\phi)=\sum_{m=-\infty}^{\infty}f_{\alpha-2,m}(\theta)e^{im\phi},\quad f_{\alpha-2,m}(\theta)=\frac{1}{2\pi}\int_{0}^{2\pi}f_{\alpha-2}(\theta,\phi)e^{-im\phi}\dd \phi.
\end{align*}\\
The Fourier coefficients $f_{\alpha-2,m}(\theta)$ satisfy the associated Legendre equation,

\begin{align*}
-f_{\alpha-2,m}''-\cot \theta f_{\alpha-2,m}'+\frac{m^{2}}{\sin^{2}\theta}f_{\alpha-2,m}=(\alpha-2)(\alpha-1)f_{\alpha-2,m},\quad \theta_2<\theta<\theta_3.
\end{align*}\\
Since $f$ and $\partial_{\theta}f$ are constants on $\{\theta=\theta_2\}$, the coefficients $f_{\alpha-2,m}$ for $m\in \mathbb{Z}\backslash \{0\}$ satisfy 

\begin{align*}
f_{\alpha-2,m}(\theta_2)=f_{\alpha-2,m}'(\theta_2)=0.
\end{align*}\\
By the uniqueness of the ODE, $f_{\alpha-2,m}(\theta)=0$ for all $m\in \mathbb{Z}\backslash \{0\}$ and 

\begin{align*}
f_{\alpha-2}(\theta,\phi)=\sum_{m=-\infty}^{m}f_{\alpha-2,m}(\theta)e^{im\phi}=f_{\alpha-2,0}(\theta).
\end{align*}\\
Thus $(v,f,p)$ is axisymmetric without swirls in $D(\theta_1,\theta_3)$.
\end{proof}

\vspace{3pt}

\begin{prop}
Suppose that $(v,f,p)$ is axisymmetric without swirls in $D(\theta_1,\theta_2)$. Then, there exists $\theta_3\in (\theta_2,\pi)$ such that $(v,f,p)$ is axisymmetric without swirls in $D(\theta_1,\theta_3)$.
\end{prop}

\vspace{3pt}

\begin{proof}
If there exists $\tilde{\phi}_1\in (0,2\pi)$ and $\theta_3\in (\theta_2,\pi)$ such that by rotation around the $z$-axis, 

\begin{align*}
\tilde{D}_0=\left\{e_{\rho}(\theta,\phi)\in \mathbb{S}^{2}\ \middle|\ 0<\phi<\tilde{\phi}_1,\ \theta_2<\theta<\theta_3\right\}\subset \{\Pi\neq 0\},
\end{align*}\\
then repetition of the same argument in the proof of Proposition 2.11 implies that level sets of $\Pi$ on $\tilde{D}_0$ are parallel to circles of the latitude and $(v,f,p)$ is axisymmetric without swirls in $D(\theta_1,\theta_3)$. 

If there exist no such points $\tilde{\phi}_1$ and $\theta_3$, $\Pi=0$ in $D(\theta_2,\theta_3)$ for some $\theta_3\in (\theta_2,\pi)$. By Proposition 2.12, $(v,f,p)$ is axisymmetric without swirls in $D(\theta_1,\theta_3)$.
\end{proof}

\vspace{3pt}

\begin{proof}[Proof of Theorem 1.3]
By Proposition 2.8, we may assume that $\alpha\neq 0$. By Proposition 2.11, there exist $\theta_1,\theta_2\in(0,\pi)$ such that by rotation $(v,f,p)$ is axisymmetric without swirls in $D(\theta_1,\theta_2)$. We set 

\begin{align*}
\theta_{*}=\sup\left\{\ \theta_3\in (\theta_2,\pi)\ \middle|\  \textrm{$(v,f,p)$ is axisymmetric without swirls}\ \textrm{in}\ D(\theta_1,\theta_3)\  \right\}.
\end{align*}\\
By Proposition 2.13, $\theta_*= \pi$. By applying the same argument for $\theta<\theta_1$, we conclude that $(v,f,p)$ is axisymmetric without swirls on $\mathbb{S}^{2}$. Thus $u$ is axisymmetric without swirls in $\mathbb{R}^{3}\backslash \{0\}$. 
\end{proof}

\vspace{3pt}

\begin{rem}
Tangentially irrotational ($-\alpha$)-homogeneous solutions $u\in C^{1}(\mathbb{R}^{3}\backslash \{0\})$ to (1.1) are irrotational solutions for $\alpha\in \mathbb{Z}\backslash \{1\}$. In fact, the condition

\begin{align*}
(\nabla_{\mathbb{R}^{3}}\times u)\times x=0,
\end{align*}\\
implies that $(1-\alpha)v=\nabla f$ and 

\begin{align*}
\nabla_{\mathbb{R}^{3}}\times u=\frac{1}{\rho^{\alpha+1}}(\nabla^{\perp}\cdot v)e_{\rho}.
\end{align*}\\
For $\alpha\neq 1$, taking the rotation to $(1-\alpha)v=\nabla f$ implies that $\nabla^{\perp}\cdot v=0$. Thus $u$ is irrotational. For $\alpha=1$, no solutions exist by Theorem 2.4.
\end{rem}

\vspace{3pt}

\section{The existence of solutions to the one-dimensional Dirichlet problem}

We construct axisymmetric homogeneous solutions to the Euler equations (1.1) by solutions of the one-dimensional (1D) Dirichlet problem (1.13). Axisymmetric solutions $u=\nabla \times (\psi\nabla \phi)+\Gamma \nabla \phi$ to the Euler equations can be constructed by solutions of the Grad--Shafranov equation in the cross-section for a prescribed Bernoulli function $\Pi=\Pi(\psi)$ and a function $\Gamma=\Gamma(\psi)$ \cite{Gav}, \cite{CLV}, \cite{DEPS21}. We consider $(-\alpha+2)$-homogeneous solutions $\psi=w(\cos\theta)/\rho^{\alpha-2}$ to the Grad--Shafranov equation for the $(-2\alpha)$-homogeneous $\Pi(\psi)=C_1|\psi|^{2+4/(\alpha-2)}$ and the $(-\alpha+1)$-homogeneous $\Gamma(\psi)=C_2\psi|\psi|^{1/(\alpha-2)}$, and derive the 1D Dirichlet problem for the function $w$ in $(-1,1)$. We outline the construction of solutions to the 1D Dirichlet problem in subsection 3.2.

\subsection{The nonautonomous Dirichlet problem}

We consider axisymmetric solenoidal vector fields expressed by the Clebsch representation 

\begin{align*}
u=\nabla \times (\psi\nabla \phi)+\Gamma \nabla \phi,
\end{align*}\\
for the axisymmetric function $\Gamma$ and stream function $\psi$ satisfying 

\begin{align*}
\psi(z,0)=0.
\end{align*}\\
Here, $\nabla=\nabla_{\mathbb{R}^{3}}$ is the gradient in $\mathbb{R}^{3}$ and $(r,\phi,z)$ are the cylindrical coordinates. The rotation of $u$ is expressed as 

\begin{align*}
\nabla \times u&=\nabla \times (\Gamma\nabla \phi)+(-L\psi)\nabla \phi,\\
L&=\partial_{z}^{2}-\frac{1}{r}\partial_r+\partial_{r}^{2}.
\end{align*}\\
For axisymmetric solutions to the Euler equations (2.1), the Bernoulli function $\Pi=p+|u|^{2}/2$ and the function $\Gamma$ are first integrals of $u$, i.e., 

\begin{align*}
u\cdot \nabla \Pi=0,\quad u\cdot \nabla \Gamma=0.
\end{align*}\\
They imply that $\Pi$ and $\Gamma$ are locally functions of $\psi$. We assume that they are globally functions of $\psi$, i.e., $\Pi=\Pi(\psi)$ and $\Gamma=\Gamma(\psi)$. By the triple product $(\nabla \psi\times \nabla \phi) \times \nabla \phi=-\nabla \psi/r^{2}$, 

\begin{align*}
(\nabla \times u)\times u&=-\frac{1}{r^{2}} \left(\Gamma'(\psi)\Gamma(\psi) +L\psi  \right)\nabla \psi,\\
\nabla \Pi&=\Pi'(\psi)\nabla \psi.
\end{align*}\\
Thus the stream function $\psi$ satisfies the Grad--Shafranov equations \cite{Grad}, \cite{Shafranov}:  

\begin{equation}
\begin{aligned}
-\frac{1}{r^{2}}L\psi &=-\Pi'(\psi)+\frac{1}{r^{2}}\Gamma'(\psi)\Gamma(\psi),\quad (z,r)\in \mathbb{R}^{2}_{+}, \\
\psi(z,0)&=0,\quad z\in \mathbb{R}.
\end{aligned}
\end{equation}\\
Here, $\Pi'(\psi)$ denotes the derivative for the variable $\psi$. Solutions to the above 2D Dirichlet problem for prescribed $\Pi(\psi)$ and $\Gamma(\psi)$ provide axisymmetric solutions to the Euler equations (1.1).

We choose particular $\Pi(\psi)$ and $\Gamma(\psi)$ to construct axisymmetric homogeneous solutions to (1.1). We use the polar coordinates $(\rho,\theta)$ in the cross-section $\mathbb{R}^{2}_{+}$ and the coordinate $t=\cos \theta$ to express the operator $L$ as 

\begin{align*}
&z=\rho\cos\theta,\quad r=\rho\sin\theta,\\
&L=\partial_{\rho}^{2}+\frac{1-t^{2}}{\rho^{2}}\partial_t^{2}.
\end{align*}\\
For ($-\alpha$)-homogeneous solutions to (1.1), stream functions are ($-\alpha+2$)-homogeneous and expressed as

\begin{align*}
\psi(z,r)=\frac{w(t)}{\rho^{\beta}},\quad \beta=\alpha-2.
\end{align*}\\
The left-hand side of $(3.1)_1$ is expressed as

\begin{align*}
-\frac{1}{r^{2}}L\psi=-\frac{1}{\rho^{\beta+4}}\left(\frac{\beta(\beta+1)}{1-t^{2}}w+w''\right).
\end{align*}\\
We choose the functions $\Pi(\psi)$ and $\Gamma(\psi)$ by 

\begin{equation}
\begin{aligned}
\Pi(\psi)&=C_1|\psi|^{2+\frac{4}{\beta}}=C_1\frac{|w|^{2+\frac{4}{\beta}}}{\rho^{2\beta+4}},\\
\Gamma(\psi)&=C_2\psi|\psi|^{\frac{1}{\beta}}=C_2\frac{w|w|^{\frac{1}{\beta}}}{\rho^{\beta+1}},
\end{aligned}
\end{equation}\\
for constants $C_1,C_2\in \mathbb{R}$. The right-hand side of $(3.1)_1$ is expressed as 

\begin{align*}
-\Pi'(\psi)+\frac{1}{r^{2}}\Gamma'(\psi)\Gamma(\psi)
=\frac{1}{\rho^{\beta+4}}\left(c_1w|w|^{\frac{4}{\beta}}+\frac{c_2}{1-t^{2}}w|w|^{\frac{2}{\beta}} \right),
\end{align*}\\
for the constants

\begin{align}
c_1=-2C_1\left(1+\frac{2}{\beta}\right),\quad c_2=C_2^{2}\left(1+\frac{1}{\beta}\right).
\end{align}\\
Then the function $w$ satisfies the $1$D Dirichlet problem:

\begin{equation}
\begin{aligned}
- w''&=\frac{\beta(\beta+1)}{1-t^{2}}w+c_1 w|w|^{\frac{4}{\beta}}+\frac{c_2}{1-t^{2}}w|w|^{\frac{2}{\beta}},\quad t\in (-1,1),\\
w(1)&=w(-1)=0. 
\end{aligned}
\end{equation}\\
We demonstrate the existence of axisymmetric ($-\alpha$)-homogeneous solutions to (1.1) by constructing solutions to (3.4) for $\beta$, $c_1$, and $c_2$ satisfying 

\begin{equation}
\begin{aligned}
&\beta\in \mathbb{R}\backslash [-2,0],\quad c_1,c_2\geq 0,\quad c_1\neq 0\ \textrm{or}\ c_2\neq 0,\\
&c_1=0\quad \textrm{for}\ -4\leq \beta<-2.
\end{aligned}
\end{equation}

\vspace{3pt}

\begin{thm}
For constants $\beta$, $c_1$, and $c_2$ satisfying (3.5), there exists a solution $w\in C^{2}(-1,1)\cap C^{1}[-1,1]$ to (3.4). For $0<\beta<1$, there exists a positive solution to (3.4). 
\end{thm}

\vspace{3pt}

\begin{rems}
(i) The function $w_{n+1}$ in $(1.9)_2$ is a solution to (3.4) for $c_1=c_2=0$ and $\beta=n$.

\noindent
(ii) No positive solutions to (3.4) exist for $\beta\in \mathbb{R}\backslash [-2,1)$. In fact, solutions $w$ to (3.4) and $w_2=(1-t^{2})/2$ satisfy 

\begin{align*}
\left(1-\frac{\beta(\beta+1)}{2}\right)\int_{-1}^{1}w\dd t=\int_{-1}^{1}\left(c_1w|w|^{\frac{4}{\beta}}+\frac{c_2}{1-t^{2}}w|w|^{\frac{2}{\beta}}\right)w_2\dd t.
\end{align*}\\
Thus $0<\beta<1$ for positive $w$.

\noindent
(iii) The number of zero points are finite for solutions to (3.4) for $\beta>0$ by the uniqueness of the ODE. In fact, if zero points $\{t_n\}\subset (-1,1)$ accumulate at some $t_0\in (-1,1)$, critical points $\{s_n\}\subset (-1,1)$ also  accumulate at the same point by Rolle's theorem, and hence $w(t_0)=w'(t_0)=0$. By the uniqueness of the ODE (3.4) for $\beta>0$, such $w$ vanishes in $[-1,1]$.
\end{rems}

\vspace{3pt}

\subsection{The minimax method outline}

We demonstrate Theorem 3.1 in the rest of this section. We recast (3.4) as 

\begin{equation}
\begin{aligned}
-L_{\beta}w&=g(t,w)\quad \textrm{in}\ \Omega,\\
w&=0 \quad \textrm{on}\ \partial\Omega,
\end{aligned}
\end{equation}\\
with the symbols 

\begin{equation}
\begin{aligned}
L_{\beta}&=\partial_{t}^{2}+\beta(\beta+1)\frac{1}{1-t^{2}},\\
g(t,w)&=c_1w|w|^{\frac{4}{\beta}}+\frac{c_2}{1-t^{2}}w|w|^{\frac{2}{\beta}},\\
\Omega&=(-1,1).
\end{aligned}
\end{equation}

\vspace{3pt}

We seek variational solutions to (3.6) in the following steps (I)-(V).\\

\noindent
(I) The direct sum decomposition. 

We construct an orthonormal basis on $L^{2}(\Omega)$ consisting of the eigenfunctions of the operator $-L_{\beta}$ associated with the eigenvalues $\mu_1\leq \mu_2\leq \cdots\leq \mu_N\leq 0< \mu_{N+1}\leq \cdots$ (Lemma 3.10). The Chandrasekhar's transform 

\begin{align}
w(t)\longmapsto \chi(\theta)=\frac{w(\cos\theta)}{\sin^{2}\theta},
\end{align}\\
is an isometry between $H^{1}_{0}(\Omega)$ and the space of all radially symmetric functions $H^{1}_{\textrm{rad}}(\mathbb{S}^{4})$ (Proposition 3.12), and yields Rayleigh's formula for the principal eigenvalue $\mu_1$ in terms of radially symmetric functions on $\mathbb{S}^{4}$:

\begin{align}
\mu_1=\inf\left\{\frac{1}{2\pi^{2}}\int_{\mathbb{S}^{4} }\left( |\nabla_{\mathbb{S}^{4}}\chi |^{2}+(2+\beta)(1-\beta)|\chi|^{2} \right)  \dd H\ \middle|\ \chi \in H^{1}_{\textrm{rad}}(\mathbb{S}^{4}),\  \frac{1}{2\pi^{2}}\int_{\mathbb{S}^{4}}\sin^{2}\theta |\chi|^{2}\dd H=1\  \right\}.
\end{align}\\
This formula implies the principal eigenvalue sign (Lemma 3.13): 

\begin{equation}
\begin{aligned}
\mu_1&>0,\quad 0<\beta<1\\
\mu_1&\leq 0,\quad \beta\in \mathbb{R}\backslash [-2,1)
\end{aligned}
\end{equation}\\
We consider the direct sum decomposition on $H^{1}_{0}(\Omega)$ by a finite-dimensional subspace associated with the nonpositive eigenvalues and its complement (Lemma 3.14):

\begin{equation}
\begin{aligned}
&H^{1}_{0}(\Omega)=Y\oplus Z,\\
&Y=\textrm{span}\ (e_1,\cdots,e_N),\\
&Z=\left\{z\in H^{1}_{0}(\Omega)\ \middle|\ (z,y)_{L^{2}}=0,\ y\in Y\ \right\}.
\end{aligned}
\end{equation}\\
For $0<\beta<1$, $Y=\emptyset$ and the operator $-L_{\beta}$ is positive on $Z=H^{1}_{0}(\Omega)$. For $\beta\in \mathbb{R}\backslash [-2,1)$, $-L_{\beta}$ is nonpositive on $Y$ and positive on $Z$. We show a coercive estimate for the bilinear form associated with $L_{\beta}$ on $Z$ (Lemma 3.16).

\noindent
(II) Regularity of critical points.  

The functional associated with (3.6) is as follows:

\begin{equation}
\begin{aligned}
I[w]&=\frac{1}{2}\int_{\Omega}\left( |w'|^{2}-\frac{\beta(\beta+1)}{1-t^{2}}w^{2}\right)\dd t-\int_{\Omega}G(t,w)\dd t, \\
G(t,w)&=\int_{0}^{w}g(t,s)\dd s=\frac{c_1\beta}{2(\beta+2)}|w|^{2+\frac{4}{\beta}}+\frac{c_2\beta}{2(\beta+1)(1-t^{2})}|w|^{2+\frac{2}{\beta}}.
\end{aligned}
\end{equation}\\
This functional is Fr\'echet differentiable and $I'\in C(H^{1}_{0}(\Omega); H^{-1}(\Omega))$. We seek a critical point $w\in H^{1}_{0}(\Omega)$ of $I$ in the sense that  

\begin{align}
<I'[w],\eta>=0,\quad \eta\in H^{1}_0(\Omega).
\end{align}\\
We show that critical points are classical solutions $w\in C^{2}(\Omega)\cap C^{1}(\overline
{\Omega})$ (Lemma 3.19). \\

\noindent
(III) The three minimax theorems.

We find critical points by applying the three minimax theorems (Lemmas 3.4, 3.6, and 3.7) according to $\beta\in \mathbb{R}\backslash [-2,0]$: \\

\noindent
(i) Mountain pass theorem: $0<\beta<1$\\
(ii) Linking theorem: $1\leq \beta$\\
(iii) Saddle-point theorem: $\beta<-2$ \\

\noindent
In order to apply minimax theorems, we check two conditions for the functional $I$:\\  

\noindent
(a) The Palais--Smale condition \\
(b) The functional estimates on subsets: $\inf_{N}I>\max_{M_0}I$\\

\noindent
(IV) The Palais--Smale condition.

We say that a sequence $\{w_n\}\subset H^{1}_{0}(\Omega)$ is a Palais--Smale sequence at level $c\in \mathbb{R}$ if 

\begin{align*}
I[w_n]&\to c,\\ 
I'[w_n]&\to 0\quad \textrm{in}\ H^{-1}(\Omega).
\end{align*}\\
We say that $I$ satisfies the $(\textrm{PS})_c$ condition if any Palais--Smale sequence at level $c\in \mathbb{R}$ has a convergent subsequence. We show that $I$ satisfies the $(\textrm{PS})_c$ condition for any $c\in \mathbb{R}$ (Lemma 3.22) by showing the boundedness of a Palais--Smale sequence in $H^{1}_{0}(\Omega)$ both for superlinear $\beta>0$ and for sublinear $\beta<-2$.\\  

\noindent
(v) The functional estimates on subsets: $\inf_{N}I>\max_{M_0}I$.

According to $\beta\in \mathbb{R}\backslash [-2,0]$, we choose a set $N\subset H^{1}_{0}(\Omega)$ on which the infimum of $I$ is bounded from below and a set in a finite-dimensional subspace $M_0\subset H^{1}_{0}(\Omega)$ on which the maximum of $I$ is small. The parameter $\beta$ affects both the linear operator $L_{\beta}$ and the power of the nonlinear terms in (3.6).

The simplest case is $0<\beta<1$ for which the operator $-L_{\beta}$ is positive on $H^{1}_{0}(\Omega)$ and the nonlinear term is superlinear. The functional $I[w]$ is positive near $w=0$ in $H^{1}_{0}(\Omega)$ and decreases for large $w$ in $H^{1}_0(\Omega)$. We choose

\begin{equation*}
\begin{aligned}
N&=\left\{w\in H^{1}_{0}(\Omega)\ \middle|\ ||w||_{H^{1}_{0}}=r_0\ \right\},\\
M_0&=\{0,w_0\},
\end{aligned}
\end{equation*}\\
for some $r_0>0$ and $w_0\in H^{1}_{0}(\Omega)$ such that $||w_0||_{H^{1}_{0}}>r_0$, and apply a mountain pass theorem (Lemma 3.4). 

For $\beta\in \mathbb{R}\backslash [-2,1)$, the operator $-L_{\beta}$ is nonpositive on $Y$ and positive on $Z$. The nonlinear term is superlinear for $\beta\geq 1$. The functional $I[w]$ is positive near $w=0$ in $Z$ and decreases for large $w$ in $H^{1}_0(\Omega)$. We choose

\begin{equation*}
\begin{aligned}
N&=\left\{z\in Z\ \middle|\ ||z||_{H^{1}_{0}}=r_0\ \right\},\\
M_0&=\left\{w=y+\lambda z_0\in Y\oplus \mathbb{R}z_0\ \middle|\ \lambda=0\ \textrm{and}\ ||y||_{H^{1}_{0}}\leq \rho_0,\ \textrm{or}\ \lambda>0\ \textrm{and}\ ||w||_{H^{1}_{0}}= \rho_0\   \right\},
\end{aligned}
\end{equation*}\\
for some $z_0\in Z$ and $\rho_0>r_0>0$ such that $||z_0||_{H^{1}_{0}}=r_0$, and apply a linking theorem (Lemma 3.6).

The nonlinear term is sublinear for $\beta <-2$. The functional $I$ is bounded from below on $Z$ and decreases on $Y$. We choose 

\begin{equation*}
\begin{aligned}
N&=Z,\\
M_0&=\left\{y\in Y\ \middle|\  ||y||_{H^{1}_{0}}= \rho_0\   \right\},
\end{aligned}
\end{equation*}\\
for some $\rho_0>0$, and apply a saddle point theorem (Lemma 3.7).\\

The existence of positive solutions to (3.6) for $0<\beta<1$ follows from the existence of solutions to the modified problem 

\begin{equation}
\begin{aligned}
-w''&=\frac{\beta(\beta+1)}{1-t^{2}}w_{+}+c_1w_{+}^{1+\frac{4}{\beta}}+c_2\frac{1}{1-t^{2}}w_{+}^{1+\frac{2}{\beta}},\quad t\in(-1,1),\\
w(1)&=w(-1)=0.
\end{aligned}
\end{equation}\\
Here, $s_{+}=\max\{s,0\}$ for $s\in \mathbb{R}$. The associated functional of this problem is the following: 

\begin{align*}
\tilde{I}[w]&=\frac{1}{2}\int_{\Omega}\left(|w'|^{2}-\frac{\beta(\beta+1)}{1-t^{2}}w_{+}^{2}\right)\dd t-\int_{\Omega}\tilde{G}(t,w)\dd t,\\
\tilde{G}(t,w)&=\int_{0}^{w}g(t,s)_{+}\dd s=\frac{c_1\beta}{2(\beta+2)}w_{+}^{2+\frac{4}{\beta}}+\frac{c_2\beta}{2(\beta+1)(1-t^{2})}w_{+}^{2+\frac{2}{\beta}}.
\end{align*}\\
We apply a mountain pass theorem to $\tilde{I}$ and find a solution $w\in C^{2}(\Omega)\cap C^{1}(\overline{\Omega})$ to (3.14) by using the pointwise estimate $w_{+}\leq |w|$ and modifying the argument for $I$ to $\tilde{I}$. The solution $w$ to (3.14) satisfies $-w''\geq 0$ in $\Omega$ and hence is a positive solution to (3.6) by the strong maximum principle \cite[6.4. Theorem 3]{E}. \\

In the sequel, we first state the minimax theorems (III) for a general Banach space $X$, and then proceed to the other steps (I), (II), (IV), and (V) for $X=H^{1}_{0}(\Omega)$ in order.

\vspace{3pt}

\subsection{The three minimax theorems}

According to \cite{Willem}, we define the Palais--Smale condition and state three minimax theorems for a Banach space $X$ equipped with the norm $||\cdot||_{X}$. Let $C^{1}(X;\mathbb{R})$ be the space of all functionals $I[\cdot]: X\to \mathbb{R}$ such that the Fr\'echet derivative $I': X\to X^{*}$ exists and is continuous on $X$. We say that $w\in X$ is a critical point of $I$ if 

\begin{align*}
<I'[w],\eta>=0\quad \textrm{for all}\ \eta\in X.
\end{align*}\\
We say that the constant $c\in \mathbb{R}$ is a critical value if there exists a critical point $w\in X$ such that $c=I[w]$. 

\begin{defn}[Palais--Smale condition]
We say that a sequence $\{w_n\}\subset X$ is a Palais--Smale sequence at level $c\in \mathbb{R}$ if 

\begin{align*}
I[w_n]&\to c,\\ 
I'[w_n]&\to 0\quad \textrm{in}\ X^{*}.
\end{align*}\\
We say that $I$ satisfies the $(\textrm{PS})_c$ condition if any Palais--Smale sequence at level $c\in \mathbb{R}$ has a convergent subsequence.
\end{defn}

The following three minimax theorems are due to Ambrosetti and Rabinowitz \cite[Theorems 2.10, 2.11, 2.12]{Willem}.

\begin{lem}[Mountain pass theorem]
Assume that there exist $w_0\in X$ and $r_0>0$ such that $||w_0||_{X}>r_0$ and 

\begin{align*}
\inf_{||w||_{X}=r_0}I[w]>I[0]\geq I[w_0]. 
\end{align*}\\
Assume that $I$ satisfies the $(\textrm{PS})_c$ condition with 

\begin{align*}
c&=\inf_{\gamma\in \Lambda}\max_{0\leq t\leq 1}I[\gamma(t)], \\
\Lambda&=\left\{\gamma\in C([0,1]; X )\ \middle|\ \gamma(0)=0,\ \gamma(1)=w_0\  \right\}.
\end{align*}\\
Then, $c$ is a critical value of $I$.
\end{lem}

\vspace{3pt}

\begin{rem}
For the functional $I$ satisfying $I[0]\geq I[w_0]$, the functional estimate in Lemma 3.4 can be expressed as $\inf_{N}I>\max_{M_0}I$ for 

\begin{align*}
N&=\left\{w\in X\ \middle|\ ||w||_{X}=r_0  \right\}, \\
M_0&=\left\{0,w_0\right\}.
\end{align*}
\end{rem}

\vspace{3pt}

We apply a linking theorem and a saddle point theorem when $X=Y\oplus Z$ with the subspaces $Y$ and $Z$ such that $0<\textrm{dim}\ Y<\infty$.

\vspace{3pt}

\begin{lem}[Linking theorem]
For $\rho_0>r_0>0$ and $z_0\in Z$ such that $||z_0||_{X}=r_0$, set 

\begin{align*}
M&=\left\{w=y+\lambda z_0\in Y\oplus \mathbb{R}z_0\ |\ ||w||_{X}\leq \rho_0,\ \lambda\geq 0,\ y\in Y\right\},\\
M_0&=\left\{w=y+\lambda z_0\in Y\oplus \mathbb{R}z_0\ |\ \lambda=0\ \textrm{and}\ ||y||_{X}\leq \rho_0,\ \textrm{or}\ \lambda> 0\ \textrm{and}\ ||w||_{X}= \rho_0 \right\},\\
N&=\left\{z\in Z\ |\ ||z||_{X}=r_0\  \right\}.
\end{align*}\\
Assume that there exist $\rho_0>r_0>0$ and $z_0\in Z$ satisfying $||z_0||_{X}=r_0$ such that 

\begin{align*}
\inf_{N}I>\max_{M_0}I.
\end{align*}\\
Assume that $I$ satisfies the $(\textrm{PS})_c$ condition with 

\begin{align*}
c&=\inf_{\gamma\in \Lambda}\max_{w\in M}I[\gamma[w]],\\
\Lambda&=\left\{\gamma\in C(M; X)\ \middle|\ \gamma|_{M_0}=id\  \right\}.
\end{align*}\\
Then $c$ is a critical value of $I$.
\end{lem}

\vspace{3pt}

\begin{lem}[Saddle-point theorem]
For $\rho_0>0$, set 

\begin{align*}
M&=\left\{w\in Y\ \middle|\ ||w||_{X}\leq \rho_0\ \right\},\\
M_0&=\left\{w\in Y\ \middle|\ ||w||_{X}= \rho_0\ \right\},\\
N&=Z.
\end{align*}\\
Assume that there exists $\rho_0>0$ such that 

\begin{align*}
\inf_{N}I>\max_{M_0}I.
\end{align*}\\
Assume that $I$ satisfies the $(\textrm{PS})_c$ condition with 

\begin{align*}
c&=\inf_{\gamma\in \Lambda}\max_{w\in M}I[\gamma[w]],\\
\Lambda&=\left\{\gamma\in C(M; X)\ \middle|\ \gamma|_{M_0}=id\  \right\}.
\end{align*}\\
Then $c$ is a critical value of $I$.
\end{lem}

\vspace{3pt}

\subsection{The direct sum decomposition}

We show the direct sum decomposition $H^{1}_{0}(\Omega)=Y\oplus Z$ in (3.11) by using basis on $L^{2}(\Omega)$ consisting of the eigenfunctions of the operator $-L_{\beta}$ and prepare a coercive estimate for a bilinear form on $Z$.

\vspace{3pt}

\begin{prop}
\begin{align}
&\int_{\Omega}\frac{w^{2}}{(1-t)^{2}}\dd t\leq 4 \int_{\Omega}|w'|^{2}\dd t,\quad w\in H^{1}_{0}(\Omega), \\
&\left|\int_{\Omega}\frac{w\eta}{1-t^{2}}\dd t\right| \leq 2 ||w'||_{L^{2}}  ||\eta||_{L^{2}},\quad w,\eta\in H^{1}_{0}(\Omega).
\end{align}
\end{prop}

\vspace{3pt}

\begin{proof}
The inequality (3.15) is nothing but Hardy's inequality \cite[(1.3.4)]{Mazya}

\begin{align*}
\int_{0}^{2}\frac{\tilde{w}^{2}}{s^{2}}\dd s\leq 4 \int_{0}^{2}|\tilde{w}'|^{2}\dd s,\quad \tilde{w}\in H^{1}_{0}(0,2),
\end{align*}\\
for $w(t)=\tilde{w}(1-t)$. We apply it to estimate 

\begin{align*}
\left|\int_{\Omega}\frac{w\eta }{1-t^{2}}\dd t\right|\leq \frac{1}{2}\int_{\Omega}\left(\frac{1}{1-t}+\frac{1}{1+t}\right)|w||\eta| \dd t\leq 2||w'||_{L^{2}}||\eta ||_{L^{2}}. 
\end{align*}
\end{proof}

\vspace{3pt}

\begin{prop}
The bilinear form 

\begin{equation}
\begin{aligned}
B(w,\eta)&=\int_{\Omega} \left(w'\eta'-a(t)w\eta\right)\dd t,\\
a(t)&=\beta(\beta+1)\frac{1}{1-t^{2}},\quad \beta\in \mathbb{R}\backslash [-2,0],
\end{aligned}
\end{equation}\\
is bounded on $H^{1}_{0}(\Omega)\times H^{1}_{0}(\Omega)$. 
\end{prop}

\begin{proof}
The boundedness follows from (3.16).
\end{proof}

\vspace{3pt}

\begin{lem}
There exists an orthonormal basis $e_1$, $e_2$, $e_3$, $\cdots$ on $L^{2}(\Omega)$ consisting of the eigenfunctions of $-L_{\beta}$ corresponding to the eigenvalues $\mu_1\leq \mu_2\leq \mu_3\leq \cdots$ satisfying $\mu_n\to\infty$. These eigenvalues are denoted by repeating its finite multiplicity.  
\end{lem}

\vspace{3pt}

\begin{proof}
For $\beta\in \mathbb{R}\backslash [-2,0]$, $\beta(\beta+1)$ is positive. We take an arbitrary $\varepsilon>0$ and apply (3.16) and Young's inequality to estimate

\begin{align*}
B(w,w)=\int_{\Omega}|w'|^{2}\dd t-\beta(\beta+1)\int_{\Omega}\frac{|w|^{2}}{1-t^{2}}\dd t
&\geq ||w'||_{L^{2}}^{2}-2\beta(\beta+1)||w'||_{L^{2}}||w||_{L^{2}} \\
&\geq \left(1-\varepsilon \beta^{2} (\beta+1)^{2} \right)||w'||_{L^{2}}^{2}-\frac{1}{\varepsilon}||w||_{L^{2}}^{2}.
\end{align*}\\
Then the bilinear form

\begin{align*}
B_{\varepsilon}(w,\eta)=B(w,\eta)+\frac{1}{\varepsilon} (w,\eta)_{L^{2}}
\end{align*}\\
is coercive for $0<\varepsilon<\beta^{-2}(\beta+1)^{-2}$.

By Lax--Milgram theorem, for $f\in H^{-1}(\Omega)$, there exists a unique $w\in H^{1}_{0}(\Omega)$ such that $B_{\varepsilon}(w,\eta)=<f,\eta>$ for all $\eta\in H^{1}_{0}(\Omega)$. By the coercivity of $B_{\varepsilon}$, the operator 

\begin{align*}
K_{\varepsilon} :H^{-1}(\Omega)\ni f\longmapsto w \in H^{1}_{0}(\Omega),
\end{align*}\\
is a bounded operator. By the embedding $H^{1}_{0}(\Omega)\subset \subset L^{2}(\Omega)\subset H^{-1}(\Omega)$, $K_{\varepsilon}$ is a compact operator on $L^{2}(\Omega)$. By symmetry of $B_{\varepsilon}$, $K_{\varepsilon}$ is a symmetric operator. Thus the spectrum of $K_{\varepsilon}$ consists of zero and countable eigenvalues lying on the real line. By the theory of compact and symmetric operators \cite[APPENDIX D, Theorem 7]{E}, there exists a countable orthogonal basis of $L^{2}(\Omega)$ consisting of the eigenfunction $\{f_n\}$ of $K_{\varepsilon}$. All eigenvalues are positive since the eigenvalue $\lambda$ and the eigenfunction $f$ of $K_{\varepsilon}$ satisfy

\begin{align*}
0<B_{\varepsilon}(K_{\varepsilon}f,K_{\varepsilon}f)=(f,K_{\varepsilon}f)_{L^{2}}=\lambda ||f||_{L^{2}}^{2}.
\end{align*}\\
Since the multiplicity of each eigenvalue is finite by Fredholm alternative \cite[APPENDIX D, Theorem 5]{E}, the eigenvalues are accumulating at zero. By repeating finite multiplicity, the eigenfunctions of $K_{\varepsilon}$ are expressed as 

\begin{align*}
0<\cdots\leq \cdots\leq \lambda_2\leq  \lambda_1,\quad \lambda_n\to0.
\end{align*}\\
Since $w_n=K_{\varepsilon}f_n=\lambda_n f_n$ satisfies 

\begin{align*}
B(w_n,\eta)+\frac{1}{\varepsilon}(w_n,\eta)_{L^{2}}=B_{\varepsilon}(w_n,\eta)=(f_n,\eta)_{L^{2}}=\frac{1}{\lambda_n}(w_n,\eta)_{L^{2}},\quad  \eta\in H^{1}_{0}(\Omega),
\end{align*}\\
the function $w_n$ is the eigenfunction of the operator $-L_{\beta}=-\partial_t^{2}-a(t)$ with the eigenvalue $\mu_n=\lambda_n^{-1}-\varepsilon^{-1}$. Since $\{f_n\}$ is orthogonal basis on $L^{2}(\Omega)$, by normalizing $\{w_n\}$ we obtain the orthonormal basis $\{e_n\}$ and the eigenvalues $\mu_1\leq \mu_2\leq \mu_3\leq \cdots$ satisfying $\mu_n\to\infty$.  
\end{proof}

\vspace{3pt}

\begin{prop}
\begin{align}
\mu_1=\inf\left\{B(w,w)\ \middle|\ w\in H^{1}_{0}(\Omega),\  \int_{\Omega}|w|^{2}\dd t=1\  \right\}.
\end{align}
\end{prop}

\vspace{3pt}

\begin{proof}
The bilinear form $B(\cdot,\cdot)$ satisfies 

\begin{align*}
B(e_n,e_m)=\mu_{n}(e_n,e_m)_{L^{2}}=\mu_{n}\delta_{n,m}.
\end{align*}\\
By the eigenfunction expansion of $w\in H^{1}_{0}(\Omega)$ in $L^{2}(\Omega)$,
  
\begin{align*}
w&=\sum_{n=1}^{\infty}(w,e_n)_{L^{2}}e_{n},\\
B(w,w)&=B\left(\sum_{n=1}^{\infty}(w,e_n)_{L^{2}}e_n, \sum_{m=1}^{\infty}(w,e_m)_{L^{2}}e_m\right)=\sum_{n=1}^{\infty}\mu_n |(w,e_n)_{L^{2}}|^{2}\geq \mu_1||w||_{L^{2}}^{2}.
\end{align*}\\
The equality holds for $w=e_1$. Thus the formula (3.18) holds.
\end{proof}

\vspace{3pt}

\begin{prop}
The transform (3.8) is an isometry between $H^{1}_{0}(\Omega)$ and $H^{1}_{\textrm{rad}}(\mathbb{S}^{4})$ in the sense that 

\begin{equation}
\begin{aligned}
\int_{-1}^{1}w'\eta'\dd t&=\frac{1}{2\pi^{2}}\int_{\mathbb{S}^{4} }(\nabla_{\mathbb{S}^{4}}\chi \cdot \nabla_{\mathbb{S}^{4}} \xi+2\chi \xi) \dd H,   \\
\int_{-1}^{1}\frac{w\eta}{1-t^{2}}\dd t&=\frac{1}{2\pi^{2}}\int_{\mathbb{S}^{4} }\chi  \xi \dd H,\quad \chi=\frac{w}{\sin^{2}\theta},\ \xi=\frac{\eta}{\sin^{2}\theta}.
\end{aligned}
\end{equation}\\
Here, $\nabla_{\mathbb{S}^{4}}\chi=\tau \chi'$ is the gradient on $\mathbb{S}^{4}$ and $\tau$ is the unit tangential vector field on $\mathbb{S}^{4}$ whose streamlines are geodesic. 
\end{prop}

\vspace{3pt}

\begin{proof}
The surface element $\dd H$ on $\mathbb{S}^{4}$ is the surface element on $\mathbb{S}^{3}$ times $\sin^{3}\theta\dd \theta$. The surface area of $\mathbb{S}^{3}$ is $2\pi^{2}$. The identity $(3.19)_2$ follows from the coordinate transform $t=\cos\theta$. The Laplace--Beltrami operator for radially symmetric functions on $\mathbb{S}^{4}$ is expressed as 

\begin{align*}
\Delta_{\mathbb{S}^{4}}\chi=\frac{1}{\sin^{3}\theta}\partial_{\theta}(\sin^{3}\theta \partial_{\theta}\chi)=\chi''+\frac{3}{\tan\theta}\chi'.
\end{align*}\\
By the transform (3.8),

\begin{align*}
\Delta_{\mathbb{S}^{4}}\chi=w''+\frac{2}{1-t^{2}}w.
\end{align*}\\
The identity $(3.19)_1$ follows from integration by parts.  
\end{proof}

\vspace{3pt}

\begin{lem}
The principal eigenvalue $\mu_1$ is positive for $0<\beta<1$ and nonpositive for $\beta\in \mathbb{R}\backslash [-2,1)$.
\end{lem}

\vspace{3pt}

\begin{proof}
By the isometry (3.8), the formula (3.18) is transformed into (3.9). The constraint in (3.9) is satisfied for $\chi=\sqrt{15}/4$. Thus for $\beta\in \mathbb{R}\backslash [-2,1)$, 

\begin{align*}
\mu_1\leq \frac{5}{4}(2+\beta)(1-\beta)\leq 0. 
\end{align*}\\
The infimum (3.9) is achieved by the function $\chi_1(\theta)=e_1(\cos\theta)/\sin^{2}\theta$. Thus for $0<\beta<1$, 

\begin{align*}
\mu_1=\frac{1}{2\pi^{2}}\int_{\mathbb{S}^{4} }\left( |\nabla_{\mathbb{S}^{4}}\chi_1 |^{2}+(2+\beta)(1-\beta)|\chi_1|^{2}\right) \dd H>0.
\end{align*}
\end{proof}

\vspace{3pt}

\begin{lem}
The direct sum decomposition (3.11) holds. The bilinear form (3.17) satisfies 

\begin{equation}
\begin{aligned}
B(w,w)&=B(y,y)+B(z,z),\quad w=y+z\in  Y\oplus Z,\\
\mu_N ||y||_{L^{2}}^{2}&\geq B(y,y)\geq \mu_1 ||y||_{L^{2}}^{2}, \quad y\in Y,\\
 B(z,z)&\geq \mu_{N+1} ||z||_{L^{2}}^{2},\quad z\in Z.
\end{aligned}
\end{equation}
\end{lem}

\vspace{3pt}

\begin{proof}
By the eigenfunction expansion of $w\in H^{1}_{0}(\Omega)$ in $L^{2}(\Omega)$, we set 

\begin{align*}
w=\sum_{n=1}^{N}(w,e_n)_{L^{2}}e_n+\sum_{n=N+1}^{\infty}(w,e_n)_{L^{2}}e_n=:y+z.
\end{align*}
Since $y\in Y$, $z\in Z$ and $H^{1}_{0}(\Omega)=Y\oplus Z$ holds. The property $(3.20)_1$ follows from $B(e_n,e_m)=\mu_n\delta_{n,m}$ and 

\begin{align*}
B(y,z)&=\sum_{n=1}^{N}\sum_{m=N+1}^{\infty}(w,e_n)_{L^{2}}(w,e_m)_{L^{2}}B(e_n,e_m)=0,\\
B(w,w)&=B(y+z,y+z)=B(y,y)+B(z,z)=\sum_{n=1}^{N}\mu_n |(w,e_n)_{L^{2}}|^{2}+\sum_{n=N+1}^{\infty}\mu_n |(w,e_n)_{L^{2}}|^{2}.
\end{align*}\\
Since $||y||_{L^{2}}^{2}=\sum_{n=1}^{N}|(w,e_n)_{L^{2}}|^{2}$ and $||z||_{L^{2}}^{2}=\sum_{n=N+1}^{\infty}|(w,e_n)_{L^{2}}|^{2}$, $(3.20)_2$ and $(3.20)_3$ follow.
\end{proof}

\vspace{3pt}

\begin{prop}
The functional 

\begin{align}
w\longmapsto \int_{\Omega}a(t)w^{2}\dd t
\end{align}\\
is weakly continuous on $H^{1}_{0}(\Omega)$.
\end{prop}

\vspace{3pt}

\begin{proof}
Suppose that (3.21) is not weakly continuous. Then, there exists $w\in H^{1}_{0}(\Omega)$ and a sequence $\{w_n\}\subset H^{1}_{0}(\Omega)$ such that $w_n\rightharpoonup w$ in $H^{1}_{0}(\Omega)$ and for some $\varepsilon_0>0$,

\begin{align*}
\left|\int_{\Omega}\frac{1}{1-t^{2}}w^{2}_{n}\dd t-\int_{\Omega}\frac{1}{1-t^{2}}w^{2}\dd t\right|\geq \varepsilon_0.
\end{align*}\\
Since a weakly convergent sequence in $H^{1}_{0}(\Omega)$ is uniformly bounded, by choosing a subsequence we may assume that $w_n\to w$ in $L^{2}(\Omega)$. By the estimate (3.16), 

\begin{align*}
0<\varepsilon_0\leq \left|\int_{\Omega}\frac{1}{1-t^{2}}(w_n+w)(w_n-w)\dd t\right|
\leq 2||w'_n+w'||_{L^{2}}||w_n-w||_{L^{2}}\to 0\quad \textrm{as}\ n\to\infty.
\end{align*}\\
This is a contradiction. Thus the functional (3.21) is weakly continuous on $H^{1}_{0}(\Omega)$. 
\end{proof}

\vspace{3pt}

\begin{lem}
There exists $\delta>0$ such that 

\begin{align}
B(z,z)\geq \delta ||z||_{H^{1}_{0}}^{2},\quad z\in Z.
\end{align}
\end{lem}

\vspace{3pt}

\begin{proof}
The constant

\begin{align*}
\delta= \inf\left\{ B(z,z)\ \middle|\ z\in Z,\ ||z||_{H^{1}_{0}}=1  \right\},
\end{align*}\\
is nonnegative by $(3.20)_3$. We show that $\delta$ is positive. We take a sequence $\{z_n\}\subset Z$ such that $||z_n||_{H^{1}_{0}}=1$ and $B(z_n,z_n)\to \delta$. By choosing a subsequence, $z_n\rightharpoonup z$ in $H^{1}_{0}(\Omega)$ and $z_n\to z$ in $L^{2}(\Omega)$ for some $z\in Z$. By the weak continuity (3.21),

\begin{align*}
\delta=\lim_{n\to\infty}B(z_n,z_n)=\lim_{n\to\infty}\left(1-\int_{\Omega}a(t)z_n^{2}\dd t\right)=1-\int_{\Omega}a(t)z^{2}\dd t.
\end{align*}\\
By the lower semicontinuity of the norm $||z_n||_{H^{1}_{0}}$ in the weak convergence and $(3.20)_3$,   

\begin{align*}
\delta=\lim_{n\to\infty}B(z_n,z_n)\geq B(z,z)\geq \mu_{N+1}||z||_{L^{2}}^{2}.
\end{align*}\\
If $z=0$, the first equation implies that $\delta=1$. If $z\neq0$, the second equation implies that $\delta>0$. Thus $\delta$ is positive.
\end{proof}

\vspace{3pt}

\subsection{Regularity of critical points}

We show that the functional $I$ in (3.12) is Fr\'echet differentiable and its critical points are classical solutions to (3.6). In the sequel, we assume that the constants $\beta$, $c_1$, and $c_2$ satisfy (3.5). 

\vspace{3pt}

\begin{prop}
\begin{align}
|w(t)|&\leq (1-t^{2})^{\frac{1}{4}}||w||_{H^{1}_{0}},\quad t\in \Omega,\ w\in H^{1}_{0}(\Omega),\\
|g(t,w)|&\leq c_1 ||w||_{H^{1}_{0} }^{1+\frac{4}{\beta}}+c_2\frac{1}{(1-t^{2})^{\frac{3}{4}-\frac{1}{2\beta} }}||w||_{H^{1}_{0} }^{1+\frac{2}{\beta}},\quad t\in \Omega,\ w\in H^{1}_{0}(\Omega).
\end{align}
\end{prop}

\vspace{3pt}

\begin{proof}
The estimate (3.23) follows from $w(1)=w(-1)=0$ and H\"older's inequality. The estimate (3.24) follows from (3.23).  
\end{proof}

\vspace{3pt}

\begin{prop}
The functional $I \in C^{1}(H^{1}_{0}(\Omega);\mathbb{R})$ satisfies 

\begin{align}
<I'[w],\eta>=B(w,\eta)-(g(\cdot ,w),\eta)_{L^{2}},\quad \eta\in H^{1}_0(\Omega),
\end{align}

\begin{equation}
\begin{aligned}
I[w]-\sigma <I'[w],w> 
&=\left(\frac{1}{2}-\sigma \right)B(w,w)
+c_1\left(\sigma-\frac{\beta}{2(\beta+2)}\right)\int_{\Omega}|w|^{2+\frac{4}{\beta}}\dd t\\
&+c_2\left(\sigma-\frac{\beta}{2(\beta+1)}\right)\int_{\Omega}\frac{1}{1-t^{2}}|w|^{2+\frac{2}{\beta}}\dd t,\quad \sigma\in \mathbb{R}.
\end{aligned}
\end{equation}
\end{prop}

\vspace{3pt}

\begin{proof}
We set 

\begin{equation*}
\begin{aligned}
I[w]=\frac{1}{2}\int_{\Omega}\left( |w'|^{2}-a(t)w^{2}\right)\dd t-\int_{\Omega}G(t,w)\dd t=:I_0[w]-J[w].
\end{aligned}
\end{equation*}\\
The functional $I_0[\cdot]=B(\cdot,\cdot)/2\in C(H^{1}_{0}(\Omega); \mathbb{R} )$ satisfies for $\varepsilon>0$,

\begin{align*}
\frac{1}{\varepsilon}\left(I_0[w+\varepsilon\eta]-I_0[w]\right)=B(w,\eta)+\frac{\varepsilon}{2}B(\eta,\eta),\quad \eta\in H^{1}_0(\Omega).
\end{align*}\\
Thus the Gateaux derivative $D_GI_0: H^{1}_{0}(\Omega)\to H^{-1}(\Omega) $ exists and $<D_GI_0[w],\eta>=B(w,\eta)$ for $\eta\in H^{1}_{0}(\Omega)$. The Gateaux derivative is continuous on $H^{1}_{0}(\Omega)$ by continuity of the bilinear form on $H^{1}_{0}(\Omega)\times H^{1}_{0}(\Omega)$. Thus the Fr\'echet derivative $I_0'=D_GI_0$ exists and $I_0\in C^{1}(H^{1}_{0}(\Omega); \mathbb{R})$.

The function $g$ satisfies the pointwise estimate (3.24). The second term on the right-hand side of (3.24) is integrable since $3/4-1/(2\beta)<1$ for $\beta\in \mathbb{R}\backslash [-2,0]$. Thus $J\in C(H^{1}_{0}(\Omega); \mathbb{R} )$. For $w,\eta\in H^{1}_{0}(\Omega)$ and $\varepsilon>0$, 

\begin{align*}
\frac{\dd}{\dd s}G(t,w+\varepsilon s\eta)=\varepsilon \eta g(t,w+\varepsilon s \eta)
\end{align*}\\
is integrable for $(s,t)\in [0,1]\times \Omega$. Thus 

\begin{align*}
\frac{1}{\varepsilon}\left(J[w+\varepsilon \eta]-J[w]\right)
=\frac{1}{\varepsilon}\int_{0}^{1}\frac{\dd}{\dd s}\left(\int_{\Omega}G(t,w+\varepsilon s\eta)\dd t \right)\dd s=\int_{\Omega}\eta\int_{0}^{1}g(t,w+\varepsilon s \eta)\dd s\dd t.
\end{align*}\\
By (3.24), letting $\varepsilon\to0$ implies the existence of the Gateaux derivative $D_GJ: H^{1}_{0}(\Omega)\to H^{-1}(\Omega) $ and 

\begin{align*}
<D_GJ[w],\eta>=\int_{\Omega} g(t,w)\eta\dd t.
\end{align*}\\
The Gateaux derivative $D_GJ$ is continuous on $H^{1}_{0}(\Omega)$ by (3.24). Thus the Fr\'echet derivative $J'=D_GJ$ exists and $J\in C^{1}(H^{1}_{0}(\Omega); \mathbb{R})$. 

We demonstrated that $I\in C^{1}(H^{1}_{0}(\Omega); \mathbb{R})$ and (3.25). The identity (3.26) follows from (3.12$)_1$ and $(3.25)$.
\end{proof}

\vspace{3pt}

\begin{lem}
Critical points $w\in H^{1}_{0}(\Omega)$ of the functional $I$ are classical solutions $w\in C^{2}(\Omega)\cap C^{1}(\overline{\Omega})$ to (3.6). 
\end{lem}

\vspace{3pt}

\begin{proof}
By (3.25), a critical point $w\in H^{1}_{0}(\Omega)$ satisfies 

\begin{align*}
\int_{\Omega}w'\eta'\dd t=\int_{\Omega}a(t)w\eta\dd t+\int_{\Omega}g \eta\dd t,\quad \eta\in H^{1}_{0}(\Omega).
\end{align*}\\
Since $a(t)w+g\in C(\Omega)$, $w''\in L^{2}_{\textrm{loc}}(\Omega)$ and 

\begin{align*}
-w''=a(t)w+g.
\end{align*}\\
Thus $w''\in C(\Omega)$. By the pointwise estimates (3.23) and (3.24), $w''\in L^{1}(\Omega)$ and hence $w'\in C(\overline{\Omega})$. 
\end{proof}

\vspace{3pt}

\subsection{The Palais--Smale condition}

We show (PS$)_c$ for the functional $I$ and for any $c\in \mathbb{R}$. The main step is to demonstrate the boundedness of sequences $\{w_n\}\subset H^{1}_{0}(\Omega)$ satisfying 

\begin{equation}
\begin{aligned}
&\sup_{n}I[w_n]<\infty, \\
&I'[w_n]\to 0 \quad \textrm{in}\ H^{-1}(\Omega).
\end{aligned}
\end{equation}\\
We show the boundedness both for the superliner $\beta>0$ and for the sublinear $\beta<-2$.

The superlinear case $\beta>0$ is based on the identity (3.26). For simplicity of the explanation, we consider the case $c_1\neq 0$. For a suitable choice of $\sigma>0$, the identity (3.26) implies that 

\begin{align*}
I[w_n]-\sigma <I'[w_n], w_n>\ \gtrsim B(w_n,w_n)+||w_n||_{L^{p}}^{p}.
\end{align*}\\
The bilinear form is quadratic for $w_n$. The last term is the $p$-th power of $w_n$ for $p=2+4/\beta>2$. The left-hand side is at most linear growth for $w_n$ under the boundedness (3.27). If $0<\beta<1$, $Y=\emptyset$ and $B(w_n,w_n)$ is estimated from below by $||w_n||_{H^{1}_{0}}^{2}$ by Lemmas 3.14 and 3.16. Hence $\{w_n\}$ is bounded in $H^{1}_{0}(\Omega)$.

In the case $\beta\geq 1$, the quadratic term is decomposed into the nonpositive term and the positive term,

\begin{align*}
B(w_n,w_n)=B(y_n,y_n)+B(z_n,z_n),\quad w_n=y_n+z_n\in Y\oplus Z.
\end{align*}\\
We show that $\{w_n\}$ remains bounded in $H^{1}_{0}(\Omega)$ by using the $p$-th power of $w_n$ and the finite dimensionality of $Y$.

\vspace{3pt}

\begin{prop}
For $\beta>0$, any sequences $\{w_n\}\subset H^{1}_{0}(\Omega)$ satisfying (3.27) are bounded.
\end{prop}

\vspace{3pt}

\begin{proof}
We argue by contradiction. Suppose on the contrary that $M_n=||w_n||_{H^{1}_{0}}$ diverges. By (3.26), 

\begin{equation}
\begin{aligned}
I[w_n]-\sigma <I'[w_n],w_n> 
&=\left(\frac{1}{2}-\sigma \right)B(w_n,w_n)
+c_1\left(\sigma-\frac{\beta}{2(\beta+2)}\right)\int_{\Omega}|w_n|^{2+\frac{4}{\beta}}\dd t\\
&+c_2\left(\sigma-\frac{\beta}{2(\beta+1)}\right)\int_{\Omega}\frac{1}{1-t^{2}}|w_n|^{2+\frac{2}{\beta}}\dd t.
\end{aligned}
\end{equation}
We take $\sigma>0$ satisfying

\begin{align*}
\frac{\beta}{2(\beta+1)}<\sigma<\frac{1}{2},
\end{align*}\\
so that all the coefficients on the right-hand side of (3.28) are positive. We set $\tilde{w}_{n}=w_n/M_n$. By choosing a subsequence, there exists $\tilde{w}\in H^{1}_{0}(\Omega)$ such that $\tilde{w}_n\to \tilde{w}$ in $C(\overline{\Omega})$. By (3.5), either $c_1$ or $c_2$ is positive. We divide (3.28) by $M_{n}^{2+4/\beta}$ if $c_1>0$ and by $M_{n}^{2+2/\beta}$ if $c_2>0$. By the pointwise estimate (3.23), we let $n\to\infty$ and conclude $\tilde{w}=0$ in both cases.

By the direct sum decomposition $\tilde{w}_n=\tilde{y}_n+\tilde{z}_n\in Y\oplus Z$, 

\begin{align*}
1=||\tilde{w}_n||_{H^{1}_{0}}^{2}&=||\tilde{y}_n||_{H^{1}_{0}}^{2}+2(\tilde{y}_n',\tilde{z}_n')_{L^{2}}+||\tilde{z}_n||_{H^{1}_{0}}^{2},\\
||\tilde{w}_n||_{L^{2}}^{2}&=||\tilde{y}_n||_{L^{2}}^{2}+||\tilde{z}_n||_{L^{2}}^{2}.
\end{align*}\\
Since $\tilde{w}_{n}\to 0$ in $C(\overline{\Omega})$, we have $\tilde{y}_n\to 0$ and $\tilde{z}_n\to 0$ in $L^{2}(\Omega)$. Since $Y$ is finite dimensional,  $\tilde{y}_n\to 0$ in $H^{1}_{0}(\Omega)$ and $\{\tilde{z}_{n}\}$ is bounded in $H^{1}_{0}(\Omega)$. By letting $n\to \infty$ to the first equation, $\lim_{n\to\infty}||\tilde{z}_n||_{H^{1}_{0}}=1$. 

By the bilinear estimates $(3.20)_2$ and (3.22), 

\begin{align*}
\frac{1}{M_n^{2}}\left( I[w_n]-\sigma<I'[w_n],w_n>\right)\ \geq \left(\frac{1}{2}-\sigma\right)\left(\mu_1||\tilde{y}_n||_{L^{2}}^{2}+\delta ||\tilde
{z}_n||_{H^{1}_{0}}^{2}  \right).
\end{align*}\\
Letting $n\to\infty$ implies that 

\begin{align*}
0\geq \left(\frac{1}{2}-\sigma\right)\delta >0.
\end{align*}\\
This is a contradiction. We conclude that $\{w_n\}$ is bounded in $H^{1}_{0}(\Omega)$.
\end{proof}

\vspace{3pt}

For the sublinear case $\beta<-2$, the boundedness of $\{I'[w_n]\}$ in $H^{-1}(\Omega)$ implies the boundedness of the sequence $\{w_n\}$ in $H^{1}_{0}(\Omega)$.

\vspace{3pt}

\begin{prop}
For $\beta<-2$, any sequences $\{w_n\}\subset H^{1}_{0}(\Omega)$ satisfying (3.27) are bounded. 
\end{prop}

\vspace{3pt}

\begin{proof}
Suppose that $M_n=||w_n||_{H^{1}_{0}}$ diverges. We set $\tilde{w}_{n}=w_n/M_n$. For $\tilde{w}_n=\tilde{y}_n+\tilde{z}_n\in Y\oplus Z$, $\tilde{\eta}_n=-\tilde{y}_n+\tilde{z}_n$ is bounded in $H^{1}_{0}(\Omega)$ since $Y$ is finite dimensional. By substituting $\tilde{\eta}_n$ into (3.25), 

\begin{align*}
\frac{1}{M_n}<I'[w_n],\tilde{\eta}_n>=B(\tilde{w}_n,\tilde{\eta}_n)
-\frac{c_1}{M_n^{-\frac{4}{\beta}}}\int_{\Omega}\tilde{w}_n|\tilde{w}_n|^{\frac{4}{\beta}}\tilde{\eta}_{n}\dd t-\frac{c_2}{M_n^{-\frac{2}{\beta}}}\int_{\Omega}\frac{1}{1-t^{2}}\tilde{w}_n|\tilde{w}_n|^{\frac{2}{\beta}}\tilde{\eta}_{n}\dd t.
\end{align*}\\
The left-hand side vanishes as $n\to\infty$. By (3.23), the last term on the right-hand side vanishes as $n\to\infty$. For $-4\leq \beta<-2$, $c_1=0$. For $\beta<-4$, the second term on the right-hand side vanishes as $n\to\infty$. Thus we have 

\begin{align*}
\lim_{n\to\infty}B(\tilde{w}_n,\tilde{\eta}_n)=0.
\end{align*}\\
By the bilinear form estimates $(3.20)_2$ and (3.22),

\begin{align*}
B(\tilde{w}_n,\tilde{\eta}_n)
=B(\tilde{y}_n+\tilde{z}_n,-\tilde{y}_n+\tilde{z}_n)
=-B(\tilde{y}_n,\tilde{y}_n)+B(\tilde{z}_n,\tilde{z}_n)
\geq -\mu_{N}||\tilde{y}_n||_{L^{2}}^{2}+\delta ||\tilde{z}_n||_{H^{1}_{0}}^{2}
\geq 0. 
\end{align*}\\
Since $Y$ is finite dimensional, $\tilde{y}_n\to 0$ and $\tilde{z}_n \to 0$ in $H^{1}_{0}(\Omega)$ as $n\to\infty$. Thus $1=\lim_{n\to\infty}||\tilde{w}_n||_{H^{1}_{0}}=0$. This is a contradiction. We conclude that $\{w_n\}$ is bounded in $H^{1}_{0}(\Omega)$.
\end{proof}

\vspace{3pt}

\begin{lem}
The functional $I\in C^{1}(H^{1}_{0}(\Omega); \mathbb{R} )$ satisfies the $(\textrm{PS})_{c}$ condition for any $c\in \mathbb{R}$. 
\end{lem}

\vspace{3pt}

\begin{proof}
By Propositions 3.20 and 3.21, any sequences $\{w_n\}\subset H^{1}_{0}(\Omega)$ satisfying (3.27) are bounded in $H^{1}_{0}(\Omega)$. By choosing  a subsequence, there exists $w\in H^{1}_{0}(\Omega)$ such that $w_n \rightharpoonup w$ in $H^{1}_{0}(\Omega)$ and $w_n\to w$ in $C(\overline{\Omega})$. By the weak convergence of $\{w_n\}$ in $H^{1}_{0}(\Omega)$,

\begin{align*}
\lim_{n\to\infty}<I'[w],w-w_n>=0.
\end{align*}\\
By the convergence of $\{I'[w_n]\}$ in $H^{-1}(\Omega)$, 

\begin{align*}
\lim_{n\to\infty}<I'[w_n],w-w_n>=0.
\end{align*}\\
By (3.25) for $\eta\in H^{1}_{0}(\Omega)$, 

\begin{align*}
<I'[w]-I'[w_n],\eta>=B(w-w_n,\eta)-(g(\cdot,w)-g(\cdot,w_n),\eta)_{L^{2}}. 
\end{align*}\\
The above two convergence imply that the left-hand side for $\eta=w-w_n$ vanishes as $n\to\infty$. By the pointwise estimate (3.24), $\{g(t,w_n)\}$ is bounded in $L^{1}(\Omega)$ and 

\begin{align*}
\left|(g(\cdot,w)-g(\cdot,w_n),w-w_n )_{L^{2}}\right|\leq ||g(\cdot,w)-g(\cdot,w_n)||_{L^{1}}||w-w_n||_{L^{\infty}}\to 0\quad \textrm{as}\ n\to\infty.
\end{align*}\\
By the pointwise estimate (3.23), 

\begin{align*}
0=\lim_{n\to\infty}B(w-w_n,w-w_n)
&=\lim_{n\to\infty}\left(\int_{\Omega}|w'-w_n'|^{2}\textrm{d}t-\int_{\Omega}a(t)|w-w_n|^{2}\textrm{d}t\right) \\
&=\lim_{n\to\infty}\int_{\Omega}|w'-w_n'|^{2}\textrm{d}t.
\end{align*}\\
Thus $\{w_n\}$ strongly converges in $H^{1}_{0}(\Omega)$. 
\end{proof}

\vspace{3pt}

\subsection{The functional estimates on subsets}

We complete the proof of Theorem 3.1.

\vspace{3pt}

\begin{lem}
The functional $I\in C^{1}(H^{1}_{0}(\Omega); \mathbb{R} )$ satisfies 

\begin{align}
\inf_NI>\max_{M_0}I, 
\end{align}\\
for the following $\beta$ and sets $N, M_0$:

\noindent
(i) 
\begin{equation*}
\begin{aligned}
0&<\beta<1,\\
N&=\left\{w\in H^{1}_{0}(\Omega)\ \middle|\ ||w||_{H^{1}_{0}}=r_0\ \right\},\\
M_0&=\{0,w_0\},
\end{aligned}
\end{equation*}\\
for some $w_0\in H^{1}_{0}(\Omega)$ and $r_0>0$ such that $||w_0||_{H^{1}_{0}}>r_0$.

\noindent
(ii) 
\begin{equation*}
\begin{aligned}
1&\leq \beta <\infty,\\
N&=\left\{z\in Z\ \middle|\ ||z||_{H^{1}_{0}}=r_0\ \right\},\\
M_0&=\left\{w=y+\lambda z_0\in Y\oplus \mathbb{R}z_0\ \middle|\ \lambda=0\ \textrm{and}\ ||y||_{H^{1}_{0}}\leq \rho_0,\ \textrm{or}\ \lambda>0\ \textrm{and}\ ||w||_{H^{1}_{0}}= \rho_0\   \right\},
\end{aligned}
\end{equation*}\\
for some $z_0\in Z$ and $\rho_0>r_0>0$ such that $||z_0||_{H^{1}_{0}}=r_0$.

\noindent
(iii)  
\begin{equation*}
\begin{aligned}
\beta &<-2,\\
N&=Z,\\
M_0&=\left\{y\in Y\ \middle|\  ||y||_{H^{1}_{0}}= \rho_0\   \right\},
\end{aligned}
\end{equation*}
for some $\rho_0>0$. 
\end{lem}

\vspace{3pt}

\begin{proof}
We show (i). For $0<\beta<1$, (3.22) holds for $w\in H^{1}_0(\Omega)$. By the pointwise estimate (3.23), 

\begin{align*}
\int_{\Omega}|w|^{2+\frac{4}{\beta}}\dd t&\lesssim ||w||_{H^{1}_{0}}^{2+\frac{4}{\beta}},\\
\int_{\Omega}\frac{1}{1-t^{2}}|w|^{2+\frac{2}{\beta}}\dd t&\lesssim ||w||_{H^{1}_{0}}^{2+\frac{2}{\beta}}.
\end{align*}\\
There exists $C>0$ such that 

\begin{equation*}
\begin{aligned}
I[w]
&=\frac{1}{2}B(w,w)-\frac{c_1\beta}{2(\beta+2)} \int_{\Omega}|w|^{2+\frac{4}{\beta}}\dd t-\frac{c_2\beta}{2(\beta+1)} \int_{\Omega}\frac{1}{1-t^{2}}|w|^{2+\frac{2}{\beta}}\dd t\\
&\geq ||w||_{H^{1}_{0}}^{2}\left(\frac{\delta}{2}-C\left( \frac{c_1\beta}{2(\beta+2)} ||w||_{H^{1}_{0}}^{\frac{4}{\beta}}+\frac{c_2\beta}{2(\beta+1)} ||w||_{H^{1}_{0}}^{\frac{2}{\beta}} \right)  \right). 
\end{aligned}
\end{equation*}\\
Thus there exists $r_0>0$ such that 

\begin{align*}
\inf_{N}I=\inf\left\{I[w]\ \middle|\ w\in H^{1}_{0}(\Omega),\ ||w||_{H^{1}_{0}}=r_0 \right\}>0.
\end{align*}\\
For arbitrary $w\in H^{1}_{0}(\Omega)$ such that $||w||_{H^{1}_{0}}=r_0$,

\begin{align*}
I[\rho w]
&=\rho^{2}\left(\frac{1}{2}B(w,w)-\rho^{\frac{4}{\beta}}\frac{c_1\beta}{2(\beta+2)} \int_{\Omega}|w|^{2+\frac{4}{\beta}}\dd t-\rho^{\frac{2}{\beta}}\frac{c_2\beta}{2(\beta+1)} \int_{\Omega}\frac{1}{1-t^{2}}|w|^{2+\frac{2}{\beta}}\dd t  \right)\\
&\to -\infty\quad \textrm{as}\ \rho\to\infty.
\end{align*}\\
Thus there exists $w_0\in  H^{1}_{0}(\Omega)$ such that $||w_0||_{H^{1}_{0}}>r_0$ and $I[w_0]\leq 0$. Thus (3.29) holds. 

We next show (ii). For $\beta\geq 1$, (3.22) holds for $z\in Z$ and in the same way as (i), there exists $r_0>0$ such that 

\begin{align*}
\inf_{N}I=\inf\left\{I[z]\ \middle|\ z\in Z,\ ||z||_{H^{1}_{0}}=r_0\  \right\}>0.
\end{align*}\\
We consider a finite-dimensional subspace and a semicircle

\begin{align*}
&Y\oplus \mathbb{R}z_0\subset H^{1}_{0}(\Omega),\quad z_0=r_0\frac{e_{N+1}}{||e_{N+1} ||_{H^{1}_{0}}}, \\
&M_0=\left\{w=y+\lambda z_0\in Y\oplus \mathbb{R}z_0\ \middle|\ \lambda=0\ \textrm{and}\ ||y||_{H^{1}_{0}}\leq \rho_0,\ \textrm{or}\ \lambda>0\ \textrm{and}\ ||w||_{H^{1}_{0}}= \rho_0\   \right\}.
\end{align*}\\
For $w=y+\lambda z_0\in Y\oplus \mathbb{R}z_0$ and $\lambda>0$, $I[w]\to -\infty$ as $||w||_{H^{1}_{0}}\to\infty$. Since the set $\{w=y+\lambda z_0 \in M_0\ |\ \lambda\geq 0,\ ||w||_{H^{1}_{0}}=\rho\ \}$ is compact, there exists $\rho_0>0$ such that 

\begin{align*}
\max\left\{ I[w]\ \middle|\ w=y+\lambda z_0 \in M_0,\ \lambda\geq 0,\ ||w||_{H^{1}_{0}}=\rho_0\ \right\}\leq 0.
\end{align*}\\
The functional $I$ is nonpositive on $Y$ and hence $\max_{M_0}I\leq 0$. Thus (3.29) holds. 

It remains to show (iii). For $\beta<-2$, we apply Young's inequality to estimate 

\begin{align*}
||z||_{H^{1}_{0}}^{2+\frac{4}{\beta}}
&\leq \varepsilon \left(1+\frac{2}{\beta}\right)||z||_{H^{1}_{0}}^{2}+\frac{2}{(-\beta) \varepsilon^{-\left(\frac{\beta}{2}+1\right)}},\\
||z||_{H^{1}_{0}}^{2+\frac{2}{\beta}}
&\leq \varepsilon \left(1+\frac{1}{\beta}\right)||z||_{H^{1}_{0}}^{2}+\frac{1}{(-\beta) \varepsilon^{-(\beta+1)}},\quad \varepsilon>0.
\end{align*}\\
Since the estimate (3.22) holds for $z\in Z$,

\begin{align*}
I[z]
&\geq \frac{\delta}{2}||z||_{H^{1}_{0}}^{2} 
-C\left( \frac{c_1\beta}{2(\beta+2)} ||z||_{H^{1}_{0}}^{2+\frac{4}{\beta}}+\frac{c_2\beta}{2(\beta+1)} ||z||_{H^{1}_{0}}^{2+\frac{2}{\beta}} \right)\\
&\geq \frac{1}{2}\left(\delta-\varepsilon C(c_1+c_2) \right)||z||_{H^{1}_{0}}^{2}
+\frac{C}{2}\left(\frac{c_1}{\left(\frac{\beta}{2}+1\right)\varepsilon^{-\left(\frac{\beta}{2}+1\right)}}+\frac{c_2}{(\beta+1)\varepsilon^{-\left(\beta+1\right)}} \right).
\end{align*}\\
Thus for small $\varepsilon>0$,

\begin{align*}
\inf_{N}I=\inf_{Z}I \geq \frac{C}{2}\left(\frac{c_1}{\left(\frac{\beta}{2}+1\right)\varepsilon^{-\left(\frac{\beta}{2}+1\right)}}+\frac{c_2}{(\beta+1)\varepsilon^{-\left(\beta+1\right)}} \right)>-\infty.
\end{align*}\\
For arbitrary $y\in Y$, we set $\rho=||y||_{H^{1}_{0}}$ and $\tilde{y}=y/\rho$. By $(3.20)_2$,

\begin{align*}
I[y]
&=\frac{1}{2}B(y,y)-\frac{c_1\beta}{2(\beta+2)} \int_{\Omega}|y|^{2+\frac{4}{\beta}}\dd t-\frac{c_2\beta}{2(\beta+1)} \int_{\Omega}\frac{1}{1-t^{2}}|y|^{2+\frac{2}{\beta}}\dd t\\
&\leq -\rho^{2+\frac{4}{\beta}}\frac{c_1\beta}{2(\beta+2)} \int_{\Omega}|\tilde{y}|^{2+\frac{4}{\beta}}\dd t-\rho^{2+\frac{2}{\beta}}\frac{c_2\beta}{2(\beta+1)} \int_{\Omega}\frac{1}{1-t^{2}}|\tilde{y}|^{2+\frac{2}{\beta}}\dd t \\
&\to -\infty\quad \textrm{as}\ \rho\to\infty.
\end{align*}\\
Since $Y$ is finite-dimensional, there exists $\rho_0>r_0$ such that 

\begin{align*}
 \frac{C}{2}\left(\frac{c_1}{\left(\frac{\beta}{2}+1\right)\varepsilon^{-\left(\frac{\beta}{2}+1\right)}}+\frac{c_2}{(\beta+1)\varepsilon^{-\left(\beta+1\right)}} \right)>
\max\left\{I[y]\ \middle|\ y\in Y,\ ||y||_{H^{1}_{0}}=\rho_0 \right\}=\max_{M_0}I.
\end{align*}\\
Thus (3.29) holds. 
\end{proof}

\vspace{3pt}

\begin{proof}[Proof of Theorem 3.1]
By minimax theorems (Lemmas 3.4, 3.6, and 3.7), there exist critical points $w\in H^{1}_{0}(\Omega)$ of $I$. The critical points are classical solutions by Lemma 3.19. For $0<\beta<1$, we apply a mountain pass theorem for a functional associated with (3.14) and obtain positive solutions to (3.6).
\end{proof}

\vspace{3pt}

\section{The existence of axisymmetric homogeneous solutions}

We construct axisymmetric ($-\alpha$)-homogeneous solutions to (1.1) in (iii) and (iv) of Theorems 1.1, 1.4, and 1.5 and Theorem 1.10 by using solutions of (3.4) constructed in Theorem 3.1. We first set homogeneous vector fields by solutions of (3.4) and show their regularity at poles by using the geodesic radial coordinate $\theta$ and the function $\chi(\theta)=w(\cos\theta)/\sin^{2}\theta$. In the sequel, we show that the homogeneous vector fields are classical solutions to the Euler equations (1.1) in $\mathbb{R}^{3}\backslash \{0\}$ for $\alpha>2$ and distributional solutions to (1.1) in $\mathbb{R}^{3}$ for $\alpha<0$.

\subsection{Regularity of homogeneous solutions}

We choose constants $\alpha$, $C_1$, and $C_2$ satisfying 

\begin{equation}
\begin{aligned}
&\alpha\in \mathbb{R}\backslash [0,2], \quad C_1\leq 0,\ C_2\in \mathbb{R},\ C_1\neq0\ \textrm{or}\ C_2\neq 0,\\
&C_1=0\quad\textrm{for}\ -2\leq \alpha<0,
\end{aligned}
\end{equation}\\
so that the constants

\begin{align}
\beta=\alpha-2,\quad c_1=-2C_1\left(1+\frac{2}{\beta}\right),\quad c_2=C_2^{2}\left(1+\frac{1}{\beta}\right),
\end{align}\\
satisfy the condition (3.5).

\vspace{3pt}

\begin{thm}
Let $\alpha$, $C_1$, and $C_2$ satisfy (4.1). Let $w\in C^{2}(-1,1)\cap C^{1}[-1,1]$ be a solution to (3.4) for $\beta$, $c_1$, and $c_2$ in (4.2). Set  

\begin{equation}
\begin{aligned}
\psi&=\frac{w(\cos\theta)}{\rho^{\beta}},\quad \Pi=C_1|\psi|^{2+\frac{4}{\beta}},\quad \Gamma=C_2\psi|\psi|^{\frac{1}{\beta}}, \\
u&=\nabla \times (\psi\nabla \phi)+\Gamma \nabla \phi=-\frac{1}{r}\partial_z\psi e_r+\frac{1}{r}\partial_r\psi e_z+\frac{\Gamma}{r}e_{\phi}=u^{P}+u^{\phi}e_{\phi},\\
p&=\Pi-\frac{1}{2}|u|^{2}.
\end{aligned}
\end{equation}\\
Then $u$ is $(-\alpha)$-homogeneous and $p$ is $(-2\alpha)$-homogeneous. They satisfy the following regularity properties:\\
\noindent
(i) For $\alpha<0$, 

\begin{align}
u^{P},\ p\in C^{1}(\mathbb{R}^{3}\backslash \{r=0\})\cap C(\mathbb{R}^{3}),\quad u^{\phi}e_{\phi}\in C(\mathbb{R}^{3}).
\end{align}\\
(ii) For $\alpha<-2$ and $C_2=0$, 

\begin{align}
u^{P},\ p\in C^{1}(\mathbb{R}^{3}\backslash \{0\}),\quad u^{\phi}=0.
\end{align}\\
(iii) For $\alpha>2$, 

\begin{align}
u^{P},\ p\in C^{2}(\mathbb{R}^{3}\backslash \{0\}),\quad u^{\phi}e_{\phi}\in C^{1}(\mathbb{R}^{3}\backslash \{0\}).
\end{align}
\end{thm}

\vspace{3pt}

We demonstrate Theorem 4.1 by representing $(u,p)$ by the polar coordinates and by a function $\chi$ in the geodesic radial coordinate. 

\vspace{3pt}
  
\begin{prop}
For $t=\cos\theta$ and 

\begin{align}
\chi(\theta)=\frac{w(\cos {\theta})}{\sin^{2}{\theta}},
\end{align}\\
the functions $(u,p)$ in (4.3) are expressed as

\begin{equation}
\begin{aligned}
u^{P}&=\frac{1}{\rho^{\beta+2}}\left(\beta\chi \sin\theta e_{\theta}+(\chi'\sin\theta+2\chi \cos\theta )e_{\rho} \right),\\
u^{\phi}e_{\phi}&=\frac{1}{\rho^{\beta+2}}C_2\chi|\chi|^{\frac{1}{\beta}}\sin^{1+\frac{2}{\beta}}\theta e_{\phi},\\
p+\frac{1}{2}|u|^{2}&=\frac{1}{\rho^{2\beta+4}}C_1|\chi|^{2+\frac{4}{\beta}}\sin^{4+\frac{8}{\beta}}\theta.
\end{aligned}
\end{equation}
\end{prop}

\vspace{3pt}

\begin{proof}
For $\psi=\chi \sin^{2}\theta/\rho^{\beta}$,

\begin{align*}
\nabla \psi=\left(e_{\rho}\partial_{\rho}+\frac{1}{\rho}\nabla_{\mathbb{S}^{2}} \right)\left(\frac{\chi \sin^{2}\theta}{\rho^{\beta} } \right)
=\frac{1}{\rho^{\beta+1}}\left(-\beta \chi \sin^{2}\theta e_{\rho}
+\left(\chi'\sin^{2}\theta+2\chi \sin\theta\cos\theta \right)e_{\theta}\right).
\end{align*}\\
Thus 

\begin{align*}
u^{P}=\frac{1}{r}\nabla \psi \times e_{\phi}=\frac{1}{\rho^{\beta+2}}\left(\beta\chi \sin\theta e_{\theta}+(\chi'\sin\theta+2\chi \cos\theta )e_{\rho} \right).
\end{align*}\\
The representations (4.8$)_2$ and (4.8$)_3$ follow by substituting $\psi=\chi \sin^{2}\theta/\rho^{\beta}$ into $\Pi$ and $\Gamma$ in (4.3$)_1$.
\end{proof}

\vspace{3pt}

We show that the desired regularity for the poloidal component $u^{P}$ follows from regularity of the function $\chi$.

\vspace{3pt}

\begin{lem}
Assume that

\begin{equation}
\begin{aligned}
&\chi\in C^{2}(0,\pi),\\
&\chi,\ \chi'\sin\theta\in C[0,\pi].
\end{aligned}
\end{equation}\\
Then, $u^{P}\in C^{1}(\mathbb{R}^{3}\backslash \{r=0\})\cap C(\mathbb{R}^{3}\backslash \{0\})$. If in addition that 

\begin{align}
\chi \in C^{2}[0,\pi],\ \chi'(0)=\chi'(\pi)=0,
\end{align}\\
then $u^{P}\in C^{1}(\mathbb{R}^{3}\backslash \{0\})$. If in addition that 

\begin{align}
\chi'''\sin\theta \in C_{0}[0,\pi],
\end{align}\\
then $u^{P}\in C^{2}(\mathbb{R}^{3}\backslash \{0\})$.
\end{lem}

\vspace{3pt}

\begin{proof}
We set 

\begin{align*}
u^{P}=\frac{1}{\rho^{\beta+2}}\left(ae_{\theta}+fe_{\rho}\right),\quad a=\beta\chi \sin\theta,\ f=\chi'\sin\theta+2\chi \cos\theta.
\end{align*}\\
The condition $\chi\in C^{2}(0,\pi)$ implies $u^{P}\in C^{1}(\mathbb{R}^{3}\backslash \{r=0\})$. The condition (4.9$)_2$ implies $u^{P}\in C(\mathbb{R}^{3}\backslash \{0\})$.

The condition (4.10) implies that 

\begin{align}
\frac{\chi'}{\sin\theta},\ \frac{f'}{\sin\theta}=\chi''+3\chi' \cot \theta-2\chi \in C[0,\pi].
\end{align}\\
We show $u^{P}\in C^{1}(\mathbb{R}^{3}\backslash \{0\})$. By $\rho^{2}\sin\theta e_{\theta}=(x_1x_3,x_2x_3,-r^{2})$ and $\rho e_{\rho}=(x_1,x_2,x_3)$, 

\begin{equation*}
\begin{aligned}
\rho^{2} ae_{\theta}=
\beta \chi 
\left(
\begin{array}{c}
x_1x_3 \\
x_2x_3\\
-r^{2}
\end{array}
\right),\quad 
\rho f e_{\rho}=f
\left(
\begin{array}{c}
x_1 \\
x_2\\
x_3
\end{array}
\right).
\end{aligned}
\end{equation*}\\
We show that $\rho^{2} ae_{\theta}\in C^{1}(\mathbb{R}^{3}\backslash \{0\})$ and $\rho f e_{\rho}\in C^{1}(\mathbb{R}^{3}\backslash \{0\})$. The first and second components of $\rho^{2} ae_{\theta}$ have the worst regularity on the $x_3$-axis. It suffices to show that $\chi x_1x_3\in C^{1}(\mathbb{R}^{3}\backslash \{0\})$. For the canonical basis $e_1$, $e_2$, $e_3$ on $\mathbb{R}^{3}$,

\begin{align}
\nabla (\chi x_1x_3)=\frac{1}{\rho^{3}}\left(\frac{\chi'}{\sin\theta}\right)
\left(
\begin{array}{c}
x_1x_3 \\
x_2x_3\\
-r^{2}
\end{array}
\right)x_1x_3
+\chi(x_3e_1+x_1e_3).
\end{align}\\
By (4.12$)_1$, $\nabla (\chi x_1x_3)\in C(\mathbb{R}^{3}\backslash \{0\})$. 

The last component of $\rho f e_{\rho}$ has the worst regularity on the $x_3$-axis. We show that $fx_3\in C^{1}(\mathbb{R}^{3}\backslash \{0\})$. By 

\begin{align}
\nabla (fx_3)=\frac{1}{\rho^{3}}\left(\frac{f'}{\sin\theta}\right)
\left(
\begin{array}{c}
x_1x_3 \\
x_2x_3\\
-r^{2}
\end{array}
\right)x_3
+fe_3,
\end{align}\\
and (4.12$)_2$, $\nabla (fx_3)\in C(\mathbb{R}^{3}\backslash \{0\})$. We conclude that $u^{P}\in C^{1}(\mathbb{R}^{3}\backslash \{0\})$.

It remains to show the $C^{2}$-regularity of $u^{P}$ under the condition (4.11). By (4.10), 

\begin{align*}
\left(\frac{\chi'}{\sin\theta}\right)'\sin\theta=\chi''-\chi'\cot\theta \in C_{0}[0,\pi].
\end{align*}\\
By (4.11),

\begin{align}
\left(\frac{f'}{\sin\theta}\right)'\sin\theta=\chi''' \sin\theta+3\chi'' \cos\theta-3\frac{\chi'}{\sin\theta}-2\chi' \sin\theta\in C_0[0,\pi].
\end{align}\\
We show that $\nabla (\chi x_1x_3)\in C^{1}(\mathbb{R}^{3}\backslash \{0\})$. The last term of (4.13) belongs to $C^{1}(\mathbb{R}^{3}\backslash \{0\})$ by (4.12$)_1$. The first and second components of the first term in (4.13) have the worst regularity on the $x_3$-axis. It suffices to show that $\chi' (x_1x_3)^{2}/\sin\theta\in C^{1}(\mathbb{R}^{3}\backslash \{0\})$. By 

\begin{align*}
\nabla \left(\frac{\chi'}{\sin\theta}(x_1x_3)^{2} \right)
=\frac{1}{\rho}\left(\frac{\chi'}{\sin\theta}\right)'\sin\theta
\left(
\begin{array}{c}
x_1x_3 \\
x_2x_3\\
-r^{2}
\end{array}
\right)\cos^{2}\phi x_3^{2}
+2\left(\frac{\chi'}{\sin\theta}\right)(x_1x_3^{2}e_1+x_1^{2}x_3e_3),
\end{align*}\\
we have $\chi' (x_1x_3)^{2}/\sin\theta\in C^{1}(\mathbb{R}^{3}\backslash \{0\})$.

The last term of (4.14) belongs to $C^{1}(\mathbb{R}^{3}\backslash \{0\})$ by (4.12$)_2$. The first and second components of the first term in (4.14) have the worst regularity on the $x_3$-axis. It suffices to show that $f'x_1x_3^{2}/\sin\theta \in C^{1}(\mathbb{R}^{3}\backslash \{0\})$. By 

\begin{align*}
\nabla \left(\frac{f'}{\sin\theta}x_1x_3^{2} \right)
=\left(\frac{f'}{\sin\theta}\right)'\sin\theta 
\cos\phi x_3^{2}e_{\theta}
+\left(\frac{f'}{\sin\theta}\right) \left(x_3^{2}e_1+2x_1x_3e_3\right),
\end{align*}\\
and (4.15), we have $f'x_1x_3^{2}/\sin\theta \in C^{1}(\mathbb{R}^{3}\backslash \{0\})$. We conclude that $u^{P}\in C^{2}(\mathbb{R}^{3}\backslash \{0\})$.
\end{proof}

\subsection{Regularity in the geodesic radial coordinate}

We show the conditions for $\chi$ in Lemma 4.3. The coordinate $\theta$ is also the geodesic radial coordinate on $\mathbb{S}^{4}$. The Laplace--Beltrami operator on $\mathbb{S}^{4}$ acting on the radial function $\chi$ is expressed as 

\begin{align}
\Delta_{\mathbb{S}^{4}}\chi =\frac{1}{\sin^{3}\theta}\partial_{\theta}\left(\sin^{3}{\theta}\ \partial_{\theta}\chi \right)=\chi''+\frac{3}{\tan\theta}\chi'.
\end{align}\\
By the transform (4.7),  

\begin{align}
\Delta_{\mathbb{S}^{4}}\chi=w''+\frac{2}{1-t^{2}}w,
\end{align}\\ 
and the equation $(3.4)$ is expressed as 

\begin{equation}
\begin{aligned}
-\Delta_{\mathbb{S}^{4}}\chi&=(\beta-1)(\beta+2)\chi+g,\quad \theta\in (0,\pi),\\
\chi'(0)&=\chi'(\pi)=0,
\end{aligned}
\end{equation}\\
for 

\begin{equation}
g=c_1\chi |\chi|^{\frac{4}{\beta}}\sin^{2+\frac{8}{\beta}}\theta+c_2\chi |\chi|^{\frac{2}{\beta}} \sin^{\frac{4}{\beta}}\theta. 
\end{equation}

\vspace{3pt}

\begin{prop}
The condition (4.9) holds for $w\in C^{2}(-1,1)\cap C^{1}[-1,1]\cap C_0[-1,1]$.
\end{prop}

\vspace{3pt}

\begin{proof}
The conditions $w\in C^{2}(-1,1)\cap C^{1}[-1,1]$ and $w(1)=w(-1)=0$ imply 

\begin{align*}
\chi(\theta)=\frac{w(\cos\theta)}{\sin^{2}\theta}\in C^{2}(0,\pi)\cap C[0,\pi].
\end{align*}\\
By differentiating $w(\cos\theta)=\chi \sin^{2}\theta$, 

\begin{align*}
\chi'\sin\theta=-w'(\cos\theta)-2\chi \cos\theta\in C[0,\pi].
\end{align*}\\
Thus (4.9) holds.
\end{proof}

\vspace{3pt}

\begin{prop}
If in addition that $w\in C^{2}[-1,1]$, the condition (4.10) holds.
\end{prop}

\vspace{3pt}

\begin{proof}
By $w(1)=w(-1)=0$ and $w\in C^{1}[-1,1]$, 

\begin{align*}
\chi'  \sin^{3}\theta=- w'(\cos\theta)\sin^{2}\theta-2 w(\cos\theta)\cos\theta,
\end{align*}\\
vanishes at $\theta=0$ and $\pi$. By integrating the first equation of (4.16),

\begin{align*}
\chi'\sin^{3}\theta=\int_{0}^{\theta}\Delta_{\mathbb{S}^{4}}\chi \sin^{3}\phi\dd \phi.
\end{align*}\\
By changing the variable,

\begin{align*}
\frac{1}{\sin\theta}\chi'=\frac{1}{4}\left(\frac{\theta}{\sin\theta}\right)^{4}\fint_{0}^{\theta^{4}}\left(\frac{\sin\phi^{\frac{1}{4}}}{\phi^{\frac{1}{4}}}\right)^{3}(\Delta_{\mathbb{S}^{4}}\chi)(\phi^{\frac{1}{4}}) \dd \phi.
\end{align*}\\
Since $\Delta_{\mathbb{S}^{4}}\chi\in C[0,\pi]$ by the identity (4.17), $\chi'/\sin\theta\in C[0,\pi]$. In particular, $\chi'(0)=\chi'(\pi)=0$. The second equation of (4.16) implies $\chi''\in C[0,\pi]$. 
\end{proof}

\vspace{3pt}

\begin{prop}
The condition (4.10) holds for solutions $w\in  C^{2}(-1,1)\cap C^{1}[-1,1]\cap C_0[-1,1]$ of (3.4) for $\beta<-4$ and $c_2=0$ and for $\beta>0$ and $c_2\geq 0$.
\end{prop}

\vspace{3pt}

\begin{proof}
By Proposition 4.4, (4.9) holds for solutions to (3.4). Thus,

\begin{align*}
\frac{w}{1-t^{2}}\in C[-1,1].
\end{align*}\\
The equation $(3.4)_1$ implies that $w''\in C[-1,1]$ for $\beta<-4$ and $c_2=0$ or for $\beta>0$ and $c_2\geq 0$. Thus (4.10) holds by Proposition 4.5.
\end{proof}

\vspace{3pt}

\begin{prop}
The condition (4.11) holds for solutions $w\in C^{2}[-1,1]\cap C_0[-1,1]$ of (3.4) for $\beta>0$ and $c_2\geq 0$. 
\end{prop}

\vspace{3pt}

\begin{proof}
By differentiating $(4.16)$,

\begin{align*}
(\Delta_{\mathbb{S}^{4}} \chi)'\sin\theta
=\chi'''\sin\theta+3\chi''\cos\theta-\frac{3}{\sin\theta}\chi'.
\end{align*}\\
The sum of the last two terms is continuous in $[0,\pi]$ by (4.10) and vanishes at $\theta=0$ and $\pi$. The function $g$ in $(4.19)$ satisfies $g'\sin\theta\in C_0[0,\pi]$. By the equation $(4.18)_1$, $(\Delta_{\mathbb{S}^{4}} \chi)'\sin\theta\in C_0[0,\pi]$. Thus, $\chi'''\sin\theta \in C_0[0,\pi]$ and the condition (4.11) holds.
\end{proof}

\vspace{3pt}

\begin{proof}[Proof of Theorem 4.1]
We show (i). For $\alpha<0$, $\beta<-2$ by (4.2). By Proposition 4.4, the solution $w$ of (3.4) satisfies the condition (4.9). Thus $u^{P}\in C^{1}(\mathbb{R}^{3}\backslash \{r=0\})\cap C(\mathbb{R}^{3}\backslash \{0\})$ by Lemma 4.3. By (4.8$)_2$ and (4.8$)_3$, 

\begin{align*}
\rho^{\beta+2} u^{\phi}e_{\phi}&=C_2\chi |\chi|^{\frac{1}{\beta}}\sin^{1+\frac{2}{\beta}}\theta e_{\phi},\\
\rho^{2\beta+4} \left(p+\frac{1}{2}|u|^{2} \right)&=C_1 |\chi|^{2+\frac{4}{\beta}}\sin^{4+\frac{8}{\beta}}\theta,\\
\rho^{2\beta+4} |u^{\phi}|^{2}&=C_2^{2} |\chi|^{2+\frac{2}{\beta}}\sin^{2+\frac{4}{\beta}}\theta.
\end{align*}\\
The property (4.9) implies that

\begin{align*}
\chi |\chi|^{\frac{1}{\beta}}\sin^{1+\frac{2}{\beta}}\theta,\ |\chi|^{2+\frac{4}{\beta}}\sin^{4+\frac{8}{\beta}}\theta \in C_0[0,\pi].
\end{align*}\\
Thus $u^{\phi}e_{\phi}\in C(\mathbb{R}^{3}\backslash \{0\})$ and $p+|u|^{2}/2\in  C(\mathbb{R}^{3}\backslash \{0\})$. Hence $p\in  C(\mathbb{R}^{3}\backslash \{0\})$. Since $(u,p)$ vanish at the origin, $(u,p)\in C(\mathbb{R}^{3})$.

For $-2\leq \alpha <0$, $C_1=0$ and $-4\leq \beta <-2$. By 

\begin{align*}
|\chi|^{2+\frac{2}{\beta}}\sin^{2+\frac{4}{\beta}}\theta \in C^{1}(0,\pi),
\end{align*}\\
$|u^{\phi}|^{2}\in C^{1}(\mathbb{R}^{3}\backslash \{r=0\})$. By $p+|u|^{2}/2=0$ and $u^{P} \in C^{1}(\mathbb{R}^{3}\backslash \{r=0\})$, $p\in C^{1}(\mathbb{R}^{3}\backslash \{r=0\})$.

For $\alpha<-2$,  $\beta<-4$. The function $|\chi|^{2+4/\beta}\sin^{4+8/\beta}\theta$ is continuously differentiable and vanishes at $\theta=0$ and $\pi$ up to the first derivative. Namely, 

\begin{align*}
|\chi|^{2+\frac{4}{\beta}}\sin^{4+\frac{8}{\beta}}\theta\in C^{1}_{0}[0,\pi].
\end{align*}\\
Thus $p+|u|^{2}/2\in C^{1}(\mathbb{R}^{3}\backslash \{0\})$. By

\begin{align*}
|\chi|^{2+\frac{2}{\beta}}\sin^{2+\frac{4}{\beta}}\theta\in C^{1}_{0}[0,\pi], 
\end{align*}\\
$|u^{\phi}|^{2}\in C^{1}(\mathbb{R}^{3}\backslash \{0\})$. Since $u^{P}\in C^{1}(\mathbb{R}^{3}\backslash \{r=0\})$, we have $p\in C^{1}(\mathbb{R}^{3}\backslash \{r=0\})$.

We show (ii). For $\alpha<-2$ and $C_2=0$, $\beta<-4$ and $c_2=0$. Thus $u^{P}\in C^{1}(\mathbb{R}^{3}\backslash \{0\})$ by Proposition 4.6 and Lemma 4.3. Since $u^{\phi}=0$, $p\in C^{1}(\mathbb{R}^{3}\backslash \{0\})$.

It remains to show (iii). For $\alpha>2$, $\beta>0$ and $c_2\geq 0$ by (4.2). By Propositions 4.6, 4.7, and Lemma 4.3, $u^{P}\in C^{2}(\mathbb{R}^{3}\backslash \{0\})$. Since 

\begin{align*}
&\chi |\chi|^{\frac{1}{\beta}}\sin^{1+\frac{2}{\beta}}\theta e_{\phi}\in C^{1}_{0}[0,\pi],\\
&|\chi|^{2+\frac{4}{\beta}}\sin^{4+\frac{8}{\beta}}\theta,\ |\chi|^{2+\frac{2}{\beta}}\sin^{2+\frac{4}{\beta}}\theta\in C^{2}_{0}[0,\pi],
\end{align*}\\
$u^{\phi}e_{\phi}\in C^{1}(\mathbb{R}^{3}\backslash \{0\})$, $p+|u|^{2}/2\in C^{2}(\mathbb{R}^{3}\backslash \{0\})$, and $|u^{\phi}|^{2}\in C^{2}(\mathbb{R}^{3}\backslash \{0\})$. Thus $p \in C^{2}(\mathbb{R}^{3}\backslash \{0\})$.
\end{proof}

\vspace{3pt}

\subsection{Distributional solutions}

We show that $(u,p)$ in Theorem 4.1 are homogeneous solutions to (1.1). In the axisymmetric setting, the equation (1.1) for $u=u^{P}+u^{\phi}e_{\phi}$ is equivalent to 

\begin{equation}
\begin{aligned}
u^{P}\cdot \nabla u^{P}+\nabla p&=\frac{|u^{\phi}|^{2}}{r}e_{r},\\
\nabla \cdot u^{P}&=0,\\
u^{P}\cdot \nabla u^{\phi}+\frac{u^{r}}{r}u^{\phi}&=0.
\end{aligned}
\end{equation}

\vspace{3pt}

\begin{prop}
The functions $(u,p)\in C^{1}(\mathbb{R}^{3}\backslash \{0\})$ for $\alpha>2$ in Theorem 4.1 are ($-\alpha$)-homogeneous solutions to (1.1) in $\mathbb{R}^{3}\backslash \{0\}$.
\end{prop}

\vspace{3pt}

\begin{proof}
For the functions $\psi$, $\Pi$, and $\Gamma$ in (4.3),

\begin{align*}
\Pi'(\psi)&=c_1 \psi |\psi|^{\frac{4}{\beta}}=\frac{c_1}{\rho^{\beta+4}}w|w|^{\frac{4}{\beta}},\\
\left(\frac{1}{2}\Gamma^{2}(\psi)\right)'&=c_2 \psi |\psi|^{\frac{2}{\beta}}=\frac{c_2}{\rho^{\beta+2}}w|w|^{\frac{2}{\beta}}.
\end{align*}\\
Since $w\in C^{2}[-1,1]$ is a solution to (3.4), $\psi\in C^{2}(\overline{\mathbb{R}^{2}_{+}}\backslash (0,0))$ satisfies the Grad--Shafranov equation (3.1).

The poloidal component $u^{P}=\nabla \times (\psi\nabla \phi)\in C^{2}(\mathbb{R}^{3}\backslash \{0\})$ satisfies the divergence-free condition $(4.20)_2$ in $\mathbb{R}^{3}\backslash \{0\}$. By the triple product $(\nabla \psi \times \nabla \phi)\times \nabla \phi=-\nabla \psi/r^{2}$, 

\begin{align*}
(\nabla \times u^{P})\times u^{P}&=-\frac{1}{r^{2}}L\psi \nabla \psi.
\end{align*}\\
By $u^{\phi}=\Gamma/r$, the pressure gradient is expressed as 

\begin{align*}
\nabla \left(p+\frac{1}{2}|u^{P}|^{2}\right)=\nabla \left(\Pi-\frac{1}{2r^{2}}\Gamma^{2}\right)=\left(\Pi'-\frac{1}{2r^{2}}(\Gamma^{2})' \right)\nabla \psi+\frac{|u^{\phi}|^{2}}{r}e_r.
\end{align*}\\
By summing up two equations, 

\begin{align*}
u^{P}\cdot \nabla u^{P}+\nabla p
=(\nabla \times u^{P})\times u^{P}+\nabla \left(p+\frac{1}{2}|u^{P}|^{2} \right)=\frac{|u^{\phi}|^{2}}{r}e_r.
\end{align*}\\
Thus $(4.20)_1$ holds in $\mathbb{R}^{3}\backslash \{0\}$.

By the Clebsch representation of $u^{P}$, 

\begin{align*}
u^{P} (ru^{\phi})=\nabla \times (\psi\nabla \phi)C_2\psi|\psi|^{\frac{1}{\beta}}
= \nabla \times \left(\frac{C_2\beta}{2\beta+1}|\psi|^{2+\frac{1}{\beta}}\nabla \phi \right).
\end{align*}\\
By $(4.20)_2$ and taking divergence, $(4.20)_3$ holds in $\mathbb{R}^{3}\backslash \{0\}$.
\end{proof}

\vspace{3pt}

\begin{prop}
The functions $(u,p)\in C(\mathbb{R}^{3})$ for $\alpha<0$ in Theorem 4.1 are distributional ($-\alpha$)-homogeneous solutions to (1.1) in $\mathbb{R}^{3}$.
\end{prop}

\vspace{3pt}

\begin{proof}
The functions $(u,p)\in C(\mathbb{R}^{3})$ satisfy $(u^{P},p)\in C^{1}(\mathbb{R}^{3}\backslash \{r=0\})$ and $u^{\phi}\in C(\mathbb{R}^{3})$. The proof of Proposition 4.8 shows that $(4.20)_1$ and $(4.20)_2$ hold in $\mathbb{R}^{3}\backslash \{r=0\}$. It is not difficult  to see that $(4.20)_1$ and $(4.20)_2$ hold in $\mathbb{R}^{3}$ in the distributional sense by a cut-off function argument around the $x_3$-axis since $(u,p)\in C(\mathbb{R}^{3})$ is locally bounded. 

The vector field $u^{P} (ru^{\phi})$ is an image of rotation as in the proof of Proposition 4.8. Thus for arbitrary $\Psi\in C^{1}_{c}(\mathbb{R}^{3}\backslash \{r=0\})$,

\begin{align*}
0=\int_{\mathbb{R}^{3}}u^{P} (ru^{\phi}) \cdot \nabla \left(\frac{\Psi}{r}\right)\dd x=\int_{\mathbb{R}^{3}}\left(u^{P}\cdot \nabla \Psi u^{\phi}-\frac{u^{r}}{r}u^{\phi}\Psi\right)\dd x.
\end{align*}\\
This equality is extendable for all $\Psi\in C^{1}_{c}(\mathbb{R}^{3})$ by using the boundedness of $u\in C(\mathbb{R}^{3})$ and a cut-off function around the $x_3$-axis. Thus $(4.20)_3$ holds in $\mathbb{R}^{3}$ in the distributional sense. 
\end{proof}

\vspace{3pt}

\begin{proof}[Proof of (iii) and (iv) of Theorems 1.1, 1.4, and 1.5]
For constants $\alpha$, $C_1$, and $C_2$ satisfying (4.1), we obtain axisymmetric ($-\alpha$)-homogeneous solutions to (1.1) in $\mathbb{R}^{3}\backslash \{0\}$ for $\alpha>2$ and in $\mathbb{R}^{3}$ for $\alpha<0$ by Theorem 4.1 and Propositions 4.8 and 4.9.

For $C_1=0$ and $C_2\in \mathbb{R}\backslash\{0\}$, constructed solutions provide Beltrami flows in Theorem 1.1 (iii) and (iv). For $C_1<0$ and $C_2\in \mathbb{R}$, constructed solutions provide Euler flows with a nonconstant Bernoulli function in Theorem 1.5 (iii) and (iv). For $C_2=0$, they satisfy the regularity property (4.5) for $\alpha<-2$ and provide axisymmetric solutions without swirls in Theorem 1.4 (iii) and (iv). 
\end{proof}

\vspace{3pt}

\begin{rem}
The axisymmetric Beltrami ($-\alpha$)-homogeneous solutions in Theorem 1.1 (iii) and (iv) admit the proportionality factor (1.3) with $C=C_2$ since the Grad--Shafranov equation (3.1) for $C_1=0$ is $-L\psi=\Gamma'\Gamma$ and 

\begin{align*}
\nabla \times u
=\nabla \times (\Gamma\nabla \phi)+(-L\psi)\nabla \phi
=\Gamma'(\nabla \times(\psi\nabla \phi)+\Gamma \nabla \phi )=\Gamma'u.
\end{align*}
\end{rem}

\vspace{3pt}

\begin{proof}[Proof of Theorem 1.10]
For $2<\alpha<3$ and $C_1$, $C_2$ satisfying (4.1), there exist ($-\alpha$)-homogeneous solutions with a positive stream function $\psi=w(\cos\theta)/\rho^{\beta}$ by Theorem 3.1. For positive constants $C$, the level set 

\begin{align*}
\{(z,r)\in \mathbb{R}^{2}_{+}\ |\ \psi(z,r)=C\ \}
\end{align*}\\
is diffeomorphic to the level set in the polar coordinate 

\begin{align*}
\left\{(\rho,\theta)\ \middle|\ \rho=\frac{1}{C^{\frac{1}{\beta}}}w^{\frac{1}{\beta}}(\cos\theta)\ \right\}.
\end{align*}\\
By (1.10$)_1$, the stream function of the irrotational axisymmetric ($-3$)-homogeneous solution is $w_2(\cos\theta)/\rho=\sin^{2}\theta/2\rho$. Its level set in the polar coordinates is  

\begin{align*}
\left\{(\rho,\theta)\ \middle|\ \rho=\frac{1}{2C}\sin^{2}\theta\ \right\}.
\end{align*}\\
Since $w\in C^{3}(-1,1)$ is positive, the two level sets are homeomorphic. Thus the level set $\{(z,r)\ |\ \psi=C\}$ is homeomorphic to the level set $\{(z,r)\ |\ \sin^{2}\theta/2\rho=C\}$.
\end{proof}

\vspace{3pt}

\begin{rem}
The level sets of the stream function $\psi=w(\cos\theta)/\rho^{\beta}$ of solutions in Theorems 1.1, 1.4, and 1.5 for $\alpha\in \mathbb{R}\backslash [0,3)$ are different from those of solutions for $2<\alpha<3$. For $\alpha\geq 3$, $w$ has finite zero points by Remarks 3.2 (ii) and (iii). Thus the level sets $\{\psi=-C\}$ for $C>0$ exist in the $(z,r)$-upper half plane. The level sets $\{\psi=\pm C\}$ are unions of the Jordan curves sharing the origin (multifoils). For $\alpha<0$, the level sets $\{\psi=\pm C\}$ are unbounded near the rays $\{\theta=\theta_0\}$ for the zero point $t_0=\cos\theta_0$ of $w$. 
\end{rem}

\appendix

\section{The nonexistence of two-dimensional reflection symmetric homogeneous solutions}

We demonstrate the nonexistence of rotational two-dimensional ($2$D) reflection symmetric ($-\alpha$)-homogeneous solutions to the Euler equations (1.1) for $-1\leq \alpha\leq 1$ in Theorem 1.7 (i) and (ii) by using the equations on a semicircle.

\subsection{Equations on the circle}

We first derive the $2$D homogeneous solution's equations on a circle \cite{LuoShvydkoy}. We use the polar coordinates $(r,\phi)$ for $x=(x_1,x_2)$ and the associated orthogonal frame

\begin{equation*}
\mathbf{e}_{r}=\left(
\begin{array}{c}
\cos \phi \\
\sin \phi 
\end{array}
\right),\quad 
\mathbf{e}_\phi=\left(
\begin{array}{c}
-\sin\phi \\
\cos\phi 
\end{array}
\right).
\end{equation*}\\
The 2D Euler equation is expressed as 

\begin{equation}
\begin{aligned}
u\cdot \nabla u+\nabla p&=0,\\
\nabla \cdot u&=0.
\end{aligned}
\end{equation}\\
We denote the $\pi/2$ counterclockwise rotation of $u=(u^{1},u^{2})$ by $u^{\perp}=(-u^{2},u^{1})$ and express (A.1) with the rotation $\omega=\partial_{x_1}u^{2}-\partial_{x_2}u^{1}$ and the Bernoulli function $\Pi=p+|u|^{2}/2$ as 

\begin{equation}
\begin{aligned}
\omega u^{\perp}+\nabla \Pi&=0,\\
\nabla \cdot u&=0.
\end{aligned}
\end{equation}\\
By multiplying $u$ by (A.2$)_1$ and by taking rotation to (A.2$)_1$, respectively, 

\begin{align}
u\cdot \nabla \Pi=0,\quad u\cdot \nabla \omega=0.
\end{align}\\
We denote ($-\alpha$)-homogeneous solutions $u$ to (A.2) by 

\begin{align}
u&=\frac{1}{r^{\alpha}}(v+f\mathbf{e}_r),\quad v=a\mathbf{e}_{\phi},
\end{align}\\
and the functions $a(\phi)$ and $f(\phi)$. The rotation of $u$ is expressed as 

\begin{align}
\omega=\frac{1}{r^{\alpha+1}}\left(  (1-\alpha)a-f' \right).
\end{align}\\
Substituting $u$ into (A.2) implies the equations for $(a,f,p)$,  

\begin{align*}
&\omega a+2\alpha \Pi=0,\\
&\omega f+ \Pi'=0,\\
&(1-\alpha)f+a'=0.
\end{align*}\\
The second and third equations imply that $p$ is constant. Thus $(a,f)$ satisfy the equations on $\mathbb{S}^{1}$:

\begin{equation}
\begin{aligned}
&af'=a^{2}+\alpha f^{2}+2\alpha p,\\
&(1-\alpha)f+a'=0.
\end{aligned}
\end{equation}\\
The conditions (A.3) imply 

\begin{align}
a\Pi'=2\alpha f \Pi,\quad a\omega'=(\alpha+1) f \omega.
\end{align}

\vspace{3pt}

\subsection{Radially symmetric solutions}

The simplest homogeneous solutions to (A.1) are radially symmetric solutions. 

\vspace{5pt}

\begin{thm}
All radially symmetric ($-\alpha$)-homogeneous solutions $u\in C^{1}(\mathbb{R}^{2}\backslash \{0\})$ to (A.1) for $\alpha\in \mathbb{R}$ are expressed by (A.4) with the constants 

\begin{equation}
\begin{aligned}
&a,f\in \mathbb{R},\quad a^{2}+f^{2}+2p=0\quad \textrm{for}\ \alpha=1, \\
&a\in \mathbb{R}
,\quad f=0,\quad a^{2}+2\alpha p=0\quad \textrm{for}\ \alpha\in \mathbb{R}\backslash \{1\}.
\end{aligned}
\end{equation}\\
The solution for $\alpha=1$ is irrotational.
\end{thm}

\vspace{3pt}

\begin{proof}
By radial symmetry, $a$ and $f$ in (A.4) are constant. By (A.6$)_2$, $(\alpha-1)f=0$. If $\alpha=1$, $a$ satisfies (A.8$)_1$ by (A.6$)_1$. If $\alpha\neq 1$, $f=0$ and $a$ satisfies (A.8$)_2$ by (A.6$)_1$. By (A.5), $u$ is irrotational for $\alpha=1$.
\end{proof}

\vspace{3pt}

\subsection{Reflection symmetric solutions}

The second simplest homogeneous solutions to (A.1) may be reflection symmetric solutions,

\begin{equation}
\begin{aligned}
u^{1}(x_1,x_2)&=u^{1}(x_1,-x_2), \\
u^{2}(x_1,x_2)&=-u^{2}(x_1,-x_2),\\
p(x_1,x_2)&=p(x_1,-x_2).
\end{aligned}
\end{equation}\\
This symmetry imposes the boundary condition,

\begin{align}
u^{2}(x_1,0)=0.
\end{align}\\
The rotation is an odd function for the $x_2$-variable,

\begin{align}
\omega(x_1,x_2)=-\omega(x_1,-x_2).
\end{align}\\
Thus continuous $\omega$ vanishes on the boundary,

\begin{align}
\omega(x_1,0)=0.
\end{align}\\
The functions $(a,f)$ satisfy equations on the semicircle for $\phi\in (0,\pi)$ with constant $p$:  

\begin{equation}
\begin{aligned}
&af'=a^{2}+\alpha f^{2}+2\alpha p,\\
&(1-\alpha)f+a'=0,\\
&a(0)=a(\pi)=0.
\end{aligned}
\end{equation}\\

\begin{thm}
\noindent
Irrotational reflection symmetric ($-\alpha$)-homogeneous solutions $u\in C^{1}(\mathbb{R}^{2}\backslash \{0\})$ to (A.1) exist if and only if $\alpha \in \mathbb{Z}$. They are given by (A.4) for 

\begin{align}
a=C\sin((\alpha-1)\phi),\quad f=C\cos \left((\alpha-1)\phi\right),\quad C\in \mathbb{R},\quad \alpha(C^{2}+2p)=0.
\end{align}
\end{thm}

\vspace{5pt}

\begin{proof}
The equation (A.5) implies that $(1-\alpha)a-f'=0$. By (A.13$)_2$ and (A.13$)_3$,   

\begin{align*}
&-f''=(\alpha-1)^{2}f\quad \textrm{in}\ (0,\pi),\\
&f'(0)=f'(\pi)=0.
\end{align*}\\
Thus $\alpha\in \mathbb{Z}$ and $f=C\cos ((\alpha-1)\phi)$ for some constant $C$. By $(1-\alpha)a-f'=0$ and (A.13$)_3$, $a=C\sin((\alpha-1)\phi)$. By (A.13$)_1$, the constant $C$ satisfies (A.14$)_4$.
\end{proof}

\vspace{3pt}

\begin{rem}
For $\alpha=1$, the irrotational reflection symmetric ($-1$)-homogeneous solution is $u=(C/r)\mathbf{e}_{r}$ for $C\in \mathbb{R}$ satisfying $C^{2}+2p=0$ (The solution $(1.4)_1$ is the case $C=1$). For $\alpha\in \mathbb{Z}\backslash \{1\}$, the irrotational reflection symmetric ($-\alpha$)-homogeneous solution is $u=(C/r^{n+1})(\sin(n\phi) \mathbf{e}_{\phi}+\cos(n\phi) \mathbf{e}_{r} ) $ for $n=\alpha-1$ and $C\in \mathbb{R}$ satisfying $(n+1)(C^{2}+2p)=0$. The stream function of this solution is $\psi=(-C/n r^{n})\sin(n\phi)$ (The solution $(1.4)_2$ is the case $C=-n$).
\end{rem}

\vspace{3pt}

\subsection{The nonexistence for $-1\leq \alpha\leq 1$}

We demonstrate the nonexistence of rotational reflection symmetric ($-\alpha$)-homogeneous solutions to (A.1) for $-1\leq \alpha\leq 1$ by using the first integral conditions (A.7).

\vspace{3pt}

\begin{proof}[Proof of Theorem 1.7 (i) and (ii)]
For $\alpha=1$, the equations (A.13$)_2$ and (A.13$)_3$ imply $a=0$. By (A.13$)_1$, $f$ is constant. Thus $u$ is irrotational by (A.5). We show that reflection symmetric ($-\alpha$)-homogeneous solutions for $-1\leq \alpha< 1$ are irrotational by using the conditions (A.7$)_1$ for $0\leq \alpha<1$ and (A.7$)_2$ for $-1\leq \alpha<0$. 

For $0<\alpha<1$, we first show that $u\in C^{1}(\mathbb{R}^{2}\backslash \{0\})$ satisfies $a \Pi=0$ in $(0,\pi)$. Suppose that $a \Pi\neq 0$ on some interval $J\subset (0,\pi)$. The equations (A.13$)_1$ and (A.7$)_1$ imply 

\begin{align*}
|\Pi|^{1-\alpha}|a|^{2\alpha}=C,
\end{align*}\\
for some constant $C$. If $C\neq 0$, $J$ is extendable to $(0,\pi)$. However, the condition $0<\alpha<1$ and the boundary condition (A.13$)_3$ imply that the left-hand side vanishes as $\phi\to0$. Thus $C=0$. This contradicts $a\Pi\neq 0$ on $J$. Hence $a\Pi=0$ in $(0,\pi)$.

We show that $\Pi=0$ in $(0,\pi)$. Suppose that $\Pi\neq 0$ on some interval $J\subset (0,\pi)$. Then, $a=0$ and $f=0$ by (A.13$)_2$. Since $p=0$ by (A.13$)_1$, $\Pi=0$. This is a contradiction and hence $\Pi=0$ in $(0,\pi)$. By (A.2), $\omega u^{\perp}=0$. Thus $\omega=0$ in $(0,\pi)$ and $u$ is irrotational. 

For $\alpha=0$, the condition (A.7$)_1$ implies $a\Pi'=0$. Suppose that $\Pi'\neq 0$ on some interval $J\subset (0,\pi)$. Then, $a=0$ and $f=0$ by (A.13$)_2$. Hence $\Pi'=0$. This is a contradiction and hence $\Pi'=0$ in $(0,\pi)$. By (A.2), $\omega=0$ in $(0,\pi)$ and $u$ is irrotational.

For $-1< \alpha<0$, we first show that $u\in C^{2}(\mathbb{R}^{2}\backslash \{0\})$ satisfies $a \omega=0$ in $(0,\pi)$. Suppose that $a\omega\neq 0$ on some interval $J\subset (0,\pi)$. The equations (A.13)$_2$ and (A.7)$_2$ imply 

\begin{align*}
|\omega|^{1-\alpha}|a|^{1+\alpha}=C,
\end{align*}\\
for some constant $C$. By applying the same argument as that for $\Pi$ and $0<\alpha<1$, we conclude that $C=0$. This contradicts $a\omega\neq 0$ on $J$. Hence $a\omega=0$ in $(0,\pi)$.

We show that $\omega=0$ in $(0,\pi)$. Suppose that $\omega\neq 0$ on some interval $J\subset (0,\pi)$. Then, $a=0$ and $f=0$ by (A.13$)_2$. Hence $\omega=0$ by (A.5). This is a contradiction and hence $\omega=0$ in $(0,\pi)$. Thus $u$ is irrotational.

For $\alpha=-1$, the condition (A.7$)_2$ implies that $a\omega'=0$. If $\omega'\neq 0$ on some interval $J\subset (0,\pi)$, $a=0$ and $f=0$ by (A.13$)_2$ and hence $\omega=0$ by (A.5). This is a contradiction. Thus $\omega'=0$ in $(0,\pi)$. Since $\omega\in C^{1}(\mathbb{R}^{2}\backslash \{0\})$, the boundary condition (A.12) implies that $\omega=0$. Thus $u$ is irrotational.
\end{proof}

\vspace{3pt}

\section{The existence of two-and-a-half-dimensional reflection symmetric homogeneous solutions}

We demonstrate the existence of rotational two-dimensional ($2$D) reflection symmetric ($-\alpha$)-homogeneous solutions to (1.1) in Theorem 1.7 (iii) and (iv). We consider more general two-and-a-half-dimensional ($2\nicefrac{1}{2}$D) reflection symmetric solutions $u=(u^{1},u^{2},u^{3})$ to (1.1) satisfying 

\begin{equation}
\begin{aligned}
u^{1}(x_1,x_2)&=u^{1}(x_1,-x_2), \\
u^{2}(x_1,x_2)&=-u^{2}(x_1,-x_2), \\
u^{3}(x_1,x_2)&=-u^{3}(x_1,-x_2).
\end{aligned}
\end{equation}

\vspace{3pt}

\begin{thm}
The following holds for rotational $2\nicefrac{1}{2}$D reflection symmetric ($-\alpha$)-homogeneous solutions $u=u^{H}+u^{3}e_z$ and $r^{2\alpha}p=\textrm{const.}$ to (1.1):

\noindent
(i) For $\alpha>1$, solutions $u\in C^{1}(\mathbb{R}^{2}\backslash \{0\})$ such that $u^{H}\in C^{2}(\mathbb{R}^{2}\backslash \{0\})$ and $u^{3}\in C^{1}(\mathbb{R}^{2}\backslash \{0\})$ exist.\\
\noindent
(ii) For $\alpha<-1$, solutions $u\in C(\mathbb{R}^{2})$ such that $u^{H}\in C^{1}(\mathbb{R}^{2}\backslash \{0\})\cap C(\mathbb{R}^{2})$ and $u^{3}\in C(\mathbb{R}^{2})$ exist.
\end{thm}

\vspace{3pt}

\subsection{The existence of solutions to the autonomous Dirichlet problem}

We construct solutions in Theorem B.1 by solutions of the automonous Dirichlet problem. We consider symmetric solenoidal vector fields (B.1) expressed by the Clebsch representation 

\begin{align}
u=\nabla \times (\psi\nabla z)+G \nabla z.
\end{align}\\
Here, $\nabla=\nabla_{\mathbb{R}^{3}}$ is the gradient in $\mathbb{R}^{3}$ and $(r,\phi,z)$ are the cylindrical coordinates. We denote the function $u^{3}$ by $G$. We assume that the stream function $\psi$ is an odd function for the $x_2$-variable and vanishes on the $x_1$-axis, 

\begin{align*}
\psi(x_1,0)=0.
\end{align*}\\
The rotation of $u$ is expressed as 

\begin{align*}
\nabla \times u=\nabla \times (G\nabla z)+(-\Delta \psi) \nabla z.
\end{align*}\\
For translation-reflection symmetric solutions to the Euler equations (2.1), the Bernoulli function $\Pi=p+|u|^{2}/2$ and the function $G$ are first integrals of $u$, i.e., 

\begin{align*}
u\cdot \nabla \Pi=0, \quad 
u\cdot \nabla G=0.
\end{align*}\\
We assume that they are globally functions of $\psi$, i.e., $\Pi=\Pi(\psi)$, $G=G(\psi)$. By using the triple product, 

\begin{align*}
(\nabla \times u)\times u&=-\frac{1}{2}\nabla G^{2}+(-\Delta \psi)\nabla \psi,\\
\nabla \Pi&=\Pi'(\psi)\nabla\psi.
\end{align*}\\
Here, $\Pi'(\psi)$ denotes the differentiation for the variable $\psi$. Thus $\psi$ satisfies the semilinear Dirichlet problem:

\begin{equation}
\begin{aligned}
-\Delta \psi&=-\Pi'(\psi)+G'(\psi)G(\psi),\quad (x_1,x_2)\in \mathbb{R}^{2}_{+}, \\ 
\psi(x_1,0)&=0,\quad x_1\in  \mathbb{R}.
\end{aligned}
\end{equation}\\
Solutions to the above 2D Dirichlet problem for prescribed $\Pi(\psi)$ and $G(\psi)$ provide $2\nicefrac{1}{2}$D reflection symmetric solutions to (1.1).

We choose particular $\Pi(\psi)$ and $G(\psi)$ to construct $2\nicefrac{1}{2}$D reflection symmetric homogeneous solutions to (1.1). For ($-\alpha$)-homogeneous solutions to (1.1), stream functions are ($-\alpha+1$)-homogeneous and expressed as  

\begin{align*}
\psi(x_1,x_2)=\frac{w(\phi)}{r^{\beta}},\quad \beta=\alpha-1.
\end{align*}\\
The left-hand side of (B.3$)_1$ is expressed as  

\begin{align*}
-\Delta \psi=-\frac{1}{r^{\beta+2}}(\beta^{2}w+w'').
\end{align*}\\
We choose the functions $\Pi(\psi)$ and $G(\psi)$ by 

\begin{equation*}
\begin{aligned}
\Pi(\psi)&=C_1|\psi|^{2+\frac{2}{\beta}}=C_1\frac{|w|^{2+\frac{2}{\beta}}}{r^{2\beta+2}},\\
G(\psi)&=C_2\psi |\psi|^{\frac{1}{\beta}}=C_2\frac{w|w|^{\frac{1}{\beta}}}{r^{\beta+1}},
\end{aligned}
\end{equation*}\\
for constants $C_1$, $C_2\in \mathbb{R}$. The right-hand side of (B.3$)_1$ is is expressed as  

\begin{align*}
-\Pi'(\psi)+G'(\psi)G(\psi)=c\frac{w|w|^{\frac{2}{\beta}}}{r^{\beta+2}}
\end{align*}\\
for the constant 

\begin{align*}
c=\left(-2C_1+C_{2}^{2}\right)\left(1+\frac{1}{\beta}\right).
\end{align*}\\
Then function $w$ satisfies the autonomous Dirichlet problem:

\begin{equation}
\begin{aligned}
-w''&=\beta^{2}w+cw|w|^{\frac{2}{\beta}},\quad \phi\in (0,\pi),\\
w(0)&=w(\pi)=0.
\end{aligned}
\end{equation}

\vspace{3pt}

\begin{thm}
For $\beta\in \mathbb{R}\backslash [-2,0]$ and $c>0$, there exists a solution $w\in C^{2}[0,\pi]$ to (B.4). For $\beta >0$, $w\in C^{3}[0,\pi]$. For $0<\beta<1$, there exists a positive solution to (B.4). 
\end{thm}

\vspace{3pt}

\begin{rems}
(i) Positive solutions to (B.4) for $0<\beta<1$ are symmetric with respect to $\phi=\pi/2$ and decreasing for $\phi>\pi/2$ since (B.4$)_1$ is an autonomous equation \cite[Theorem 1]{GNN}, cf. \cite[Theorem $1'$]{GNN}.

\noindent
(ii) No positive solutions to (B.4) exist for $\beta\in \mathbb{R}\backslash [-2,1)$ and $c>0$, cf. \cite[Remarks 6.3 (ii)]{QS}. In fact, the principal eigenvalue of the operator $-\partial_{\phi}^{2}$ is one and its associated eigenfunction is $\sin\phi$. By multiplying $\sin\phi$ by (B.4$)_1$ and integration by parts,

\begin{align*}
(1-\beta^{2})\int_{0}^{\pi}w \sin\phi \dd \phi=c\int_{0}^{\pi}w |w|^{\frac{2}{\beta}}\sin\phi\dd \phi.
\end{align*}\\
Thus $0<\beta<1$ for positive $w$.

\noindent
(iii) The number of zero points are finite for solutions to (B.4) for $\beta>0$ by the uniqueness of the ODE as observed for the nonautonomous case in Remarks 3.2 (iii).
\end{rems}

\vspace{3pt}

\begin{proof}[Proof of Theorem B.2]
The construction of solutions to (B.4) is easier than that of solutions to (3.4). In fact, the problem (B.4) is expressed as 

\begin{equation}
\begin{aligned}
-L_{\beta}w&=g(w)\quad \textrm{in}\ \Omega,\\
w&=0\qquad \textrm{on}\ \partial \Omega,
\end{aligned}
\end{equation}\\
with the symbols $L_{\beta}=\partial_{\phi}^{2}+\beta^{2}$, $g(w)= c w|w|^{\frac{2}{\beta}}$, and $\Omega=(0,\pi)$. The eigenvalues of the operator $-L_{\beta}$ are explicitly given by $\mu_{n}=n^{2}-\beta^{2}$ for $n=1,2,\dots$. Hence $\mu_1>0$ for $0<\beta<1$ and $\mu_1\leq 0$ for $\beta\in \mathbb{R}\backslash [-2,1)$. The orthonormal basis on $L^{2}(\Omega)$ are $e_n=\sqrt{2/\pi}\sin{n\phi}$. We take $N\in \mathbb{N}$ such that $\mu_1<\mu_2<\cdots<\mu_N\leq 0<\mu_{N+1}<\cdots$ and consider the direct sum decomposition $H^{1}_{0}(\Omega)=Y\oplus Z$ for $Y=\textrm{span}(e_1,\cdots,e_N)$ and $Z=\{z\in H^{1}_{0}(\Omega)\ |\ (z,y)_{L^{2}}=0,\ y \in Y\ \}$. The functional associated with (B.5) is the following:

\begin{align*}
I[w]&=\frac{1}{2}\int_{\Omega}(|w'|^{2}-\beta^{2}|w|^{2})\textrm{d}\phi
-\int_{\Omega}G(w)\textrm{d}\phi,\\
G(w)&=\int_{0}^{w}g(s)\textrm{d}s=\frac{c \beta}{2(\beta+1)}|w|^{2(1+\frac{1}{\beta})}.
\end{align*}\\
This functional satisfies the (PS$)_c$ condition for any $c\in \mathbb{R}$ and desired estimates on subsets. We apply the mountain pass theorem (Lemma 3.4) for $0<\beta<1$, the linking theorem (Lemma 3.6) for $1\leq \beta$, and the saddle point theorem (Lemma 3.7) for $\beta<-2$. Since the equation (B.4) is autonomous, critical points $w\in H^{1}_{0}(\Omega)$ belong to $w\in C^{2}[0,\pi]$. For $\beta>0$, differentiating (B.4$)_1$ by $\phi$ implies that $w\in C^{3}[0,\pi]$. The existence of positive solutions follows from an application of the mountain pass theorem to a modified problem for $0<\beta<1$ as we observed for the nonautonomous equation in (3.14).
\end{proof}

\vspace{3pt}

\subsection{Regularity of two-and-a-half-dimensional reflection symmetric homogeneous solutions}

We construct $2\nicefrac{1}{2}$D reflection symmetric homogeneous solutions by solutions of (B.4) in Theorem B.2. We choose constants $\alpha$, $C_1$, and $C_2$ satisfying 

\begin{align}
\alpha\in \mathbb{R}\backslash [-1,1],\quad C_1,C_2\in \mathbb{R},\quad -2C_1+C_2^{2}>0,
\end{align}\\
so that 

\begin{align}
\beta=\alpha-1\in \mathbb{R}\backslash [-2,0],\quad c=(-2C_1+C_2^{2})\left(1+\frac{1}{\beta}\right)>0.
\end{align}

\begin{lem}
Let $\alpha$, $C_1$, and $C_2$ satisfy (B.6). Let $w\in C^{2}[0,\pi]$ be a solution to (B.4) for $\beta$ and $c$ in (B.7). For the odd extension of $w$ to $[-\pi,0]$, set 

\begin{align*}
\psi&=\frac{w(\phi)}{r^{\beta}},\quad \Pi=C_1|\psi|^{2+\frac{2}{\beta}},\quad G=C_2\psi |\psi|^{\frac{1}{\beta}},\\
u&=\nabla \times (\psi\nabla z)+G\nabla z=u^{H}+u^{3}e_{z},\\
p&=\Pi-\frac{1}{2}|u|^{2}.
\end{align*}\\
Then $u$ is ($-\alpha$)-homogeneous and $r^{2\alpha}p$ is constant. They satisfy the following regularity properties:
 
\noindent
(i) For $\alpha<-1$, 

\begin{align}
u^{H}\in C^{1}(\mathbb{R}^{2}\backslash \{0\})\cap C(\mathbb{R}^{2}),\quad u^{3}\in C(\mathbb{R}^{2}).
\end{align}\\
(ii) For $\alpha>1$, 

\begin{align}
u^{H}\in C^{2}(\mathbb{R}^{2}\backslash \{0\}),\quad u^{3}\in C^{1}(\mathbb{R}^{2}\backslash \{0\}).
\end{align}
\end{lem}

\vspace{3pt}

\begin{proof}
The functions $u$ and $p$ are expressed in terms of $w$ as

\begin{equation*}
\begin{aligned}
r^{\beta+1}u&=w'e_r+\beta we_{\phi}+C_2w|w|^{\frac{1}{\beta}}e_{z}, \\
r^{2\beta+2}p&=-\frac{c\beta }{2(\beta+1)}|w|^{2+\frac{2}{\beta}}-\frac{1}{2}\left(|w'|^{2}+\beta^{2}w^{2} \right).
\end{aligned}
\end{equation*}\\
Differentiating the second equation by $\phi$ implies that $r^{2\beta+2}p$ is independent of $\phi$. For $\beta<-2$, $w\in C^{2}[-\pi,\pi]$ implies that $u^{H}\in C^{1}(\mathbb{R}^{2}\backslash \{0\})$. Since $u$ vanishes at the origin, $u\in C(\mathbb{R}^{2})$. Thus (B.8) holds. For $\beta>0$, $w\in C^{3}[-\pi,\pi]$ implies (B.9).
\end{proof}

\vspace{3pt}

We show that $(u,p)$ in Lemma B.4 are homogeneous solutions to (1.1). In the translation reflection symmetric setting, the equations (1.1) for $u=u^{H}+u^{3}e_{z}$ are equivalent to 

\begin{equation}
\begin{aligned}
u^{H}\cdot \nabla_{\mathbb{R}^{2}} u^{H}+\nabla_{\mathbb{R}^{2}} p&=0,\\
\nabla_{\mathbb{R}^{2}} \cdot u^{H}&=0,\\
u^{H}\cdot \nabla_{\mathbb{R}^{2}} u^{3}&=0.
\end{aligned}
\end{equation}\\

\begin{prop}
The functions $(u,p)$ in Lemma B.4 are ($-\alpha$)-homogeneous solutions to (B.10) in $\mathbb{R}^{2}\backslash \{0\}$ for $\alpha>1$ and in $\mathbb{R}^{2}$ for $\alpha<-1$.
\end{prop}

\vspace{3pt}

\begin{proof}
Since $w\in C^{2}[-\pi,\pi]$, the stream function $\psi\in C^{2}(\mathbb{R}^{2}\backslash \{0\})$ is a reflection symmetric solution to the elliptic equation (B.3$)_1$ in $\mathbb{R}^{2}\backslash \{0\}$ for 

\begin{align*}
\Pi'(\psi)&=2C_1\left(1+\frac{1}{\beta}\right)\frac{w|w|^{\frac{2}{\beta}}}{r^{\beta+2}},\\
\left(\frac{1}{2}G^{2}(\psi) \right)'&=C_2^{2}\left(1+\frac{1}{\beta}\right)\frac{w|w|^{\frac{2}{\beta}}}{r^{\beta+2}}.
\end{align*}\\
We suppress the subscript for $\nabla=\nabla_{\mathbb{R}^{2}}$ and use the symbol $\nabla^{\perp}=(-\partial_2,\partial_1)$. By multiplying $\nabla \psi$ by (B.3$)_1$, 

\begin{align*}
-\Delta \psi \nabla \psi+\nabla \left(\Pi-\frac{1}{2}G^{2}\right)=0.
\end{align*}\\
Since $u^{H}=-\nabla^{\perp} \psi$, $-\Delta \psi=\nabla^{\perp} \cdot u^{H}$, and $p=\Pi-|u^{H}|^{2}/2-|G|^{2}/2$, 

\begin{align*}
\nabla^{\perp} \cdot u^{H}(u^{H})^{\perp}+\nabla \left(p+\frac{1}{2}|u^{H}|^{2}\right)=0.
\end{align*}\\
By the identity, 

\begin{align*}
u^{H}\cdot \nabla u^{H}=\nabla^{\perp} \cdot u^{H}(u^{H})^{\perp}+\nabla \frac{1}{2}|u^{H}|^{2},
\end{align*}\\
(B.10$)_1$ and (B.10$)_2$ hold in $\mathbb{R}^{2}\backslash \{0\}$. For $\alpha>1$, taking divergence to 

\begin{align*}
u^{H}u^{3}=-\nabla^{\perp} \left(C_2\frac{\beta}{2\beta+1}|\psi|^{2+\frac{1}{\beta}}\right),
\end{align*}\\
implies that (B.10$)_3$ holds in $\mathbb{R}^{2}\backslash \{0\}$. Thus $(u,p)$ is a ($-\alpha$)-homogeneous solution to (B.10) in $\mathbb{R}^{2}\backslash \{0\}$ for $\alpha>1$. 

For $\alpha<-1$, (B.10$)_1$ and (B.10$)_2$ in $\mathbb{R}^{2}\backslash \{0\}$ hold for $u^{H}\in C^{1}(\mathbb{R}^{2}\backslash \{0\})\cap C(\mathbb{R}^{2})$ and (B.10$)_3$ in $\mathbb{R}^{2}\backslash \{0\}$ holds for $u^{3}\in C(\mathbb{R}^{2})$ in the distributional sense. Since $(u,p)$ vanishes at the origin, by the cut-off function argument around the origin, $(u,p)$ satisfies (B.10) in $\mathbb{R}^{2}$ in the distributional sense.
\end{proof}

\vspace{3pt}

\begin{proof}[Proof of Theorems B.1, 1.7 (iii) and (iv) and 1.10]
For the constants $\alpha$, $C_1$ and $C_2$ satisfying (B.6), Lemma B.4 and Proposition B.5 imply the existence of rotational reflection symmetric $2\nicefrac{1}{2}$D ($-\alpha$)-homogeneous solutions to (1.1) in Theorem B.1. In particular, taking $C_2=0$ implies the existence of rotational $2$D reflection symmetric ($-\alpha$)-homogeneous solutions in Theorem 1.7 (iii) and (iv). For $1<\alpha<2$, the existence of positive solutions to (B.4) implies the existence of $2$D rotational reflection symmetric ($-\alpha$)-homogeneous solutions whose stream function level sets are homeomorphic to those of the irrotational ($-2$)-homogeneous solution in Theorem 1.8.  
\end{proof}

\vspace{3pt}

\begin{rem}
For $\alpha\in \mathbb{R}\backslash [-1,1]$, $C_1=0$ and $C_2\neq0$, the solutions in Theorem B.1 are $2\nicefrac{1}{2}$D reflection symmetric Beltrami ($-\alpha$)-homogeneous solutions to (1.1). By $-\Delta\psi=G'G$ and the Clebsch representation for $\nabla=\nabla_{\mathbb{R}^{3}}$,

\begin{align*}
\nabla \times u=G'(\nabla \times (\psi \nabla z)+G\nabla z)=G'u.
\end{align*}\\
For $\alpha>1$, $u$ admits the proportionality factor $\varphi=G'=C_2(1+1/\beta)|\psi|^{1/\beta}$. 
\end{rem}

\bibliographystyle{alpha}
\bibliography{ref}

\end{document}